\documentclass[12pt]{amsart}
\usepackage{a4wide,enumerate,xcolor,comment}
\usepackage{amsmath}
\usepackage{scrextend} % for labeling environment
\usepackage{mathrsfs} % \mathscr (calligraphic font) for filtration
\allowdisplaybreaks

\let\pa\partial
\let\na\nabla
\let\eps\varepsilon
\newcommand{\N}{{\mathbb N}}
\newcommand{\R}{{\mathbb R}}
\newcommand{\diver}{\operatorname{div}}
\newcommand{\dd}{{\mathrm d}}
\newcommand{\Prob}{\mathbb{P}}
\newcommand{\E}{\mathbb{E}}

\newtheorem{theorem}{Theorem}
\newtheorem{lemma}[theorem]{Lemma}

\newtheorem{remark}[theorem]{Remark}
\newtheorem{corollary}[theorem]{Corollary}

\begin{document}

\title[Fluctuations in the moderate regime]{Fluctuations around the mean-field limit for attractive Riesz potentials in the moderate regime}

\author[L. Chen]{Li Chen}
\address{University of Mannheim, School of Business Informatics and Mathematics, 68131 Mann\-heim, Germany}
\email{chen@math.uni-mannheim.de}

\author[A. Holzinger]{Alexandra Holzinger}
\address{Mathematical Institute, University of Oxford, Andrew Wiles Building, Radcliffe Observatory Quarter, Woodstock Road, Oxford, United Kingdom OX2 6GG}
\email{alexandra.holzinger@maths.ox.ac.uk} 

\author[A. J\"ungel]{Ansgar J\"ungel}
\address{Institute of Analysis and Scientific Computing, Vienna University of Technology, Wiedner Hauptstra\ss e 8--10, 1040 Wien, Austria}
\email{juengel@tuwien.ac.at} 

\date{\today}

\thanks{The first author acknowledges partial support from the German Research Foundation (DFG), grant CH 955/8-1. The last author acknowledges partial support from the Austrian Science Fund (FWF), grants P33010 and F65. This work has received funding from the European Research Council (ERC) under the European Union's Horizon 2020 research and innovation programme, ERC Advanced Grant Nonlocal-CPD with no.~883363 and ERC Advanced Grant NEUROMORPH with no.~101018153.} 

\begin{abstract}
A central limit theorem is shown for moderately interacting particles in the whole space. The interaction potential approximates singular attractive or repulsive potentials of sub-Coulomb type. It is proved that the fluctuations become asymptotically Gaussians in the limit of infinitely many particles. The methodology is inspired by the classical work of Oelschl\"ager on fluctuations for the porous-medium equation. The novelty in this work is that we can allow for attractive potentials in the moderate regime and still obtain asymptotic Gaussian fluctuations. The key element of the proof is the mean-square convergence in expectation for smoothed empirical measures associated to moderately interacting $N$-particle systems with rate $N^{-1/2-\eps}$ for some $\eps>0$. To allow for attractive potentials, the proof uses a quantitative mean-field convergence in probability with any algebraic rate and a law-of-large-numbers estimate as well as a systematic separation of the terms to be estimated in a mean-field part and a law-of-large-numbers part.
\end{abstract}

% \paragraph{Keywords:}  
 
% \paragraph{AMS classification:}  

\maketitle

%%%%%%%%%%%%%%%%%%%%%%%%%%%%%%%%%%%%%%%%%%%%%%%%%%%%%%%%%%%%%%%%%%%

\setcounter{tocdepth}{1}
\tableofcontents

\section{Introduction}

Systems of stochastic interacting particles arise in many applications like mathematical physics, population biology, and sociology. If the number of particles is very large, the numerical computation is very costly. Therefore, it is reasonable to determine the particle dynamics in the (mean-field) limit of infinitely many particles. Depending on the strength of interaction, the singularity of the interaction potential, and the type of force (repulsive or attractive), various convergence results of the empirical measure towards a deterministic PDE solution can be derived. This can be seen as a law-of-large-numbers result for interacting particles of mean-field type. 

One might ask for a stochastic correction of this deterministic limit, in particular whether a central limit theorem can be recovered in the large population limit. In this paper, we show that asymptotically, the fluctuations around the deterministic mean-field limit become Gaussian, where the variance is determined by the dual backward time-evolution of the linearized limiting nonlinear PDE. To prove this result, the mean-square convergence in expectation with algebraic rate for smoothed empirical measures for moderately interacting particles with attractive sub-Coulomb potentials is shown. Compared to the previous works \cite{Oel87, JoMe98}, which are concerned with fluctuations in the moderate mean-field regime, we allow for aggregating particles (for sub-coulomb potentials) and a central limit theorem scaling of the fluctuations process with $\sqrt{N}$, where $N$ denotes the number of particles. 

The technical novelties of this paper are as follows. First, we prove $L^2(\R^d)$ estimates of the smoothed empirical measure using a convergence-in-probability result with any algebraic rate, inspired by works of Pickl; see \cite{HLP20,LaPi17}. Compared to these works, we do not use a cut-off of the interaction kernel and purely work with mollified potentials. Second, we use the $L^2(\R^d)$ estimates to derive an asymptotic central limit theorem for moderate particle systems of aggregating or repulsive types, following ideas of Oelschl\"ager \cite{Oel87} developed for the porous medium equation.

\subsection{The setting}

We consider a system of $N$ interacting particles with positions $X_i^\eta(t)$, satisfying the stochastic differential equations
\begin{equation}\label{1.X}
\begin{aligned}
  \dd X_i^\eta(t) &= \frac{\kappa}{N}\sum_{j=1}^{N}\na V^\eta
  (X_i^\eta(t)-X_j^\eta(t))\dd t + \sqrt{2\sigma}\dd W_i(t), 
  \quad t>0, \\
  X_i^\eta(0) &= \zeta_i \quad\mbox{in }\R^d,\ i=1,\ldots,N.
\end{aligned}
\end{equation}
Here, the space dimension is $d\ge 3$, the parameter $\kappa=\pm 1$ describes the type of interaction ($\kappa=1$ for aggregation and $\kappa=-1$ for repulsion), $\sigma>0$ is the diffusion constant, $W_i$ are independent Brownian motions on $\R^d$, and $\zeta_1,\ldots,\zeta_N$ are independent, identically distributed (i.i.d.) random variables with the common probability density function $u_0$. The potential $V^\eta$ is given by
\begin{equation}\label{1.Veta}
  V^\eta = \chi^\eta*\Phi, \quad \Phi(x)=|x|^{-\lambda},
\quad 0<\lambda<d-2,
\end{equation}
where $\chi^\eta(x)=\eta^{-d}\chi(|x|/\eta)$ is a smooth mollification kernel with $\chi=\xi*\xi$ and $\xi\in C_0^\infty(\R^d)$. We refer to Section \ref{sec.main} for the complete list of assumptions. The function $\Phi$ is a Riesz kernel, and the condition $0<\lambda<d-2$ refers to potentials of sub-Coulomb type. We could treat Coulomb potentials with $\lambda=d-2$ for our convergence in probability result (Theorem \ref{thm.prob}), but only using a suitable cut-off, which does not allow us to show a fluctuation theorem; see Remark \ref{rem.coulomb2}. However, we can allow for general potentials $V^\eta$ satisfying suitable growth conditions with respect to $\eta$; see Remark \ref{rem.general}. As in \cite{Oel87}, a key assumption in this paper is that the potential is a convolution square, i.e., there exists $\Psi$ such that $\Phi=\Psi*\Psi$. For the Riesz kernel in \eqref{1.Veta}, this assumption is satisfied; see Lemma \ref{lem.Psi}. 

In the so-called moderate regime, the mollification parameter $\eta$ is coupled to the number of particles $N$ such that $\eta(N)\to 0$ if $N\to \infty$. Standard choices of $\eta$ are $N^{-\beta}$ or $\log(N)^{-\beta}$ for suitable ranges of $\beta$; see the classical work \cite{Oel85} as well as \cite{MeRo87,JoMe98, CDHJ21,ORT23, KOdeS24} and references therein. It is known \cite{ORT23} that in the moderate regime $\eta=N^{-\beta}$ for some $0<\beta<1/(2d-2)$, particle system \eqref{1.X} propagates chaos in the many-particle limit $N\to\infty$ towards the non-linear SDE system
\begin{equation*}%\label{1.Xlimit}
\begin{aligned}
  \dd X_i(t) &= \kappa\nabla \Phi \ast u(t, X_i(t)) \dd t 
  + \sqrt{2\sigma} \dd W_i(t), \quad t>0, \\
  X_i(0) &= \zeta_i \quad\mbox{in }\R^d,\ i=1,\ldots,N,
\end{aligned}
\end{equation*} 
where $u$ is the law of $X_i$ for all $i=1,\ldots,N$, solving the problem
\begin{equation}\label{1.u} 
  \pa_t u = \sigma\Delta u - \kappa\diver(u\na\Phi*u), \quad t>0, \quad
  u(0)=u_0\quad\mbox{in }\R^d
\end{equation} 
in the weak sense. Indeed, introducing the empirical measure
\begin{equation}\label{1.mu}
  \mu^\eta(t) = \frac{1}{N}\sum_{i=1}^N\delta_{X_i^\eta(t)},
  \quad t>0,
\end{equation}
where $\delta$ is the Dirac delta distribution, the work \cite{ORT23} shows that $\mu^\eta(t)$ converges towards the PDE solution $u$ of \eqref{1.u}.

The main aim of this paper is to determine the limiting asymptotic behavior of the fluctuations process 
\begin{align}\label{1.fluctuation}
\mathcal{F}^\eta(t):=\sqrt{N}(\mu^\eta(t)-\bar{u}^\eta(t))
\end{align} as $N\to \infty, \eta = N^{-\beta} \to 0$ for some $\beta >0$, where $\bar{u}^\eta$ is the unique solution to the so-called intermediate nonlocal diffusion equation
\begin{equation}\label{1.baru}
  \pa_t \bar{u}^\eta = \sigma\Delta\bar{u}^\eta - \kappa\diver
  (\bar{u}^\eta\na V^\eta*\bar{u}^\eta),\quad t>0, \quad
  \bar{u}^\eta(0) = u_0\quad\mbox{in }\R^d.
\end{equation} 
This PDE can be seen as the mean-field limit in the weak interaction regime, i.e.\ in the large population limit $N\to \infty$ but for fixed $\eta>0$ independent of $N$. The associated intermediate particle description for fixed $\eta >0$ reads as
\begin{equation}\label{1.barX}
\begin{aligned}
  \dd\bar{X}_i^\eta(t) &= \kappa\big(\bar{u}^\eta*\na V^\eta(t,\bar{X}_i^\eta(t))\big)dd t 
  + \sqrt{2\sigma}\dd W_i(t), \quad t>0, \\
  \bar{X}_i^\eta(0) &= \zeta_i\quad\mbox{in }\R^d,\ i=1,\ldots,N.
\end{aligned}
\end{equation}
We notice that all positions $\bar{X}_i^\eta$ are independent and possess the common density function $\bar{u}^\eta$. Equations \eqref{1.barX} still depend on the particle number $N$ via the interaction radius $\eta=N^{-\beta}$.

The choice of the fluctuations process in \eqref{1.fluctuation} with respect to $\bar{u}^\eta$ rather than $u$ is motivated by the work of Oelschl\"ager \cite{Oel87}. He showed that the corrected fluctuations process $\widetilde{\mathcal{F}}^\eta=\sqrt{N}(\widetilde{\mu}^\eta-\widetilde{u}+C_N)$ recovers Gaussian behavior in the limit $N\to\infty$, where $\lim_{N\to\infty}C_N(x,t)=0$, $\widetilde{u}$ solves the viscous porous medium equation, and $\widetilde{\mu}^\eta$ denotes the corresponding empirical measure of the interacting particle system. The correction $C_N$ is needed since convergence towards the viscous porous medium equation $\widetilde{u}$ is expected to be too slow to compensate the $\sqrt{N}$ scaling. In the present article, in comparison to \cite{Oel87}, we do not use an asymptotic expansion to derive the deterministic correction $C_N$, but we are able to compare directly with the solution $\bar{u}^\eta$ of the nonlocal equation \eqref{1.baru}. The correction $C_N$ used by Oelschl\"ager corresponds to the PDE error $\bar{u}^\eta-u$ in the present work.

A different approach for fluctuations in the moderate regime was suggested by Jourdain and M\'el\'eard \cite{JoMe98}, where a scaling of the form $(\log N)^{r}$ for some $r>0$ instead of $\sqrt{N}$ is used to define the fluctuations process. For this scaling, only a deterministic limit can be expected. To the best of the authors knowledge, \cite{JoMe98, Oel87} are the only works which consider fluctuations in the moderate mean-field regime.

Our main result (Theorem \ref{thm.fluct}) is a central limit theorem for the intermediate fluctuations process $\mathcal{F}^\eta$ defined in \eqref{1.fluctuation}. If this process converges to some limit process $\mathcal{F}$, it can be heuristically seen as a correction of the mean-field behavior, since
$$
  \mu^\eta(t) = u(t) + (\bar{u}^\eta-u)(t) + N^{-1/2}\mathcal{F}^\eta(t) 
  \sim u(t) + N^{-1/2}\mathcal{F}(t),
$$
using that $\bar{u}^\eta \to u$ for $\eta \to0$. This means that the particle dynamics can be captured for sufficiently large $N$ by the mean-field limit $u(t)$ plus some stochastic correction $\mathcal{F}$ with scaling $N^{-1/2}$. If $\mathcal{F}$ is a Gaussian process, this corresponds to a central limit theorem. Indeed, Theorem \ref{thm.fluct} below shows that the limiting behavior is asymptotically Gaussian, in particular
\begin{align}\label{1.fluctuations}
  \langle \mathcal{F}(t), \phi \rangle 
  \sim \mathcal{N}\bigg(0,\langle u_0,(T^t_\phi (0))^2 \rangle 
  - \langle u_0, T^t_\phi (0)\rangle^2 
  + 2\sigma\int_{0}^t \langle u(s), |\nabla T^t_\phi (s)|^2\rangle 
  \dd s\bigg),
\end{align}
where $u$ solves \eqref{1.u} and $T^t_\phi$ is the unique solution to the dual backward problem of the linearized version of \eqref{1.u}, i.e., it solves for fixed $t>0$ and test function $\phi$ the equation
\begin{align}\label{1.back}
  -\pa_s v - \sigma\Delta v 
  = \kappa(\na\Phi*u)\cdot\na v - \kappa\na\Phi*(u\na v)
  \quad\mbox{in }\R^d,\ s\in(0,t), \quad v(t)=\phi.
\end{align} 
We refer the reader to Lemma \ref{lem.estT} for existence and uniqueness results as well as regularity estimates for \eqref{1.back}. The proof that the laws of $\langle\mathcal{F}^\eta,\phi\rangle$ are asymptotically normally distributed is inspired by Oelschl\"ager \cite[Sec.~3]{Oel87}. He proved a related result for the porous medium equation, i.e.\ with potential $\Phi = \delta_0$ and $\kappa=-1$, assuming conditions on the Fourier transform of $V^\eta$ and its derivatives. We do not need such conditions and can include attractive potentials. However, our framework is currently limited to Riesz kernels and does not allow for attractive Dirac kernels in the limit. It is an open question whether a central limit theorem can be recovered also in that case.

To derive the fluctuations result, we need a convergence rate of order $N^{-1/2-\eps}$ with $\eps>0$ in $L^\infty((0,T); L^2(\R^d)) \cap L^2((0,T); H^1(\R^d))$ norm of the smoothed empirical measure towards the smoothed intermediate PDE solution; see Theorem \ref{thm.L2} below. The necessity of this type of convergence can be formally understood from the stochastic differential equation for $\mathcal{F}^\eta$ in the weak form 
\begin{align}\label{1.flucSDE}
  \dd\langle\mathcal{F}^\eta,\psi\rangle 
  &= \frac{\sqrt{2\sigma}}{\sqrt{N}}
  \sum_{i=1}^N\na\psi(X_i^\eta)\cdot\dd W_i 
  + \sqrt{N}R(t)\dd t,
\end{align}
where $\psi$ is a test function and $R(t)$ an error term depending on the difference $\mu_N - \bar{u}^\eta$; see Lemma \ref{lem.dF}. In the limit $N\to\infty$, the error term $R(t)$ can be shown to converge to zero, since a smoothed version of the empirical measure converges in $L^2(\R^d)$ norm faster than $N^{-1/2}$ towards a smoothed version of $\bar{u}^\eta$, which is guaranteed by our second result, Theorem \ref{thm.L2}.

The {\em key technical step of this work} (Theorem \ref{thm.L2}) is the mean-square convergence of a smoothed empirical measure associated to \eqref{1.X} towards the smoothed solution to the intermediate equation \eqref{1.barX}. More precisely, let $Z^\eta=\xi^\eta*\Psi$ such that $V^\eta=Z^\eta*Z^\eta$, where $\xi^\eta(x)=\eta^{-d}\xi(|x|/\eta)$. The smoothed quantities are
\begin{equation*}
  f^\eta:= Z^\eta*\mu^\eta, \quad g^\eta := Z^\eta*\bar{u}^\eta.
\end{equation*}
Choosing $\beta>0$ in an appropriate range, for any $T>0$, there exist $\eps>0$ and $C>0$ such that
\begin{equation*}%\label{1.mfe}
  \E\bigg(\sup_{0<t<T}\|(f^\eta-g^\eta)(t)\|_{L^2(\R^d)}^2
  + \sigma\int_0^T\|\na(f^\eta-g^\eta)(t)\|_{L^2(\R^d)}^2\dd t\bigg)
  \leq CN^{-1/2-\eps}.
\end{equation*}
As mentioned above, in the limit $\eta=N^{-\beta}\to 0$, the function $\bar{u}^\eta$ converges to the solution to 
\begin{equation*} %\label{1.u} 
  \pa_t u = \sigma\Delta u - \kappa\diver(u\na\Phi*u), \quad t>0, \quad
  u(0)=u_0\quad\mbox{in }\R^d;
\end{equation*} 
see Lemma \ref{lem.wLinfty}, which implies the convergence of $f^\eta$ towards $\Psi*u$. Observe that formally $g^\eta$ and $f^\eta$ converge to $\Psi*u$, not to $u$, because of the smoothing with the interaction kernel $V^\eta$. This convolution structure appears naturally because of the double convolution structure of $V^\eta$ in the estimates of the fluctuations process; see \eqref{1.flucSDE}. The structure of the $L^2(\R^d)$ norm of $f^\eta-g^\eta$ corresponds to the modulated energy used by Serfaty and co-authors \cite{Ser17,Ser20,Ser24} for singular repulsive kernels. However, in the present setting, for fixed $N$, our interaction kernel is smooth.

The convergence $\mu^\eta \to u$ for particle system \eqref{1.X} in Lebesgue spaces was proved in \cite{ORT23}. 
In the proof of the fluctuation result (Theorem \ref{thm.fluct}), we show the result of Oelschl\"ager \cite{Oel87}, where only the repulsive regime is admissible, for aggregating equations of sub-Coulomb type. In order to allow for attractive regimes, we use a convergence result in probability with arbitrary algebraic convergence rate, which is guaranteed by our third main result. (We refer the reader to Section \ref{sec.key} for a more detailed explanation of the main ideas for the proof of Theorem \ref{thm.L2}.)

Our {\em preparatory result} (Theorem \ref{thm.prob}) for Theorem \ref{thm.L2} concerns the convergence in probability of $X_i^\eta-\bar{X}_i^\eta$: If $\beta>0$ is sufficiently small and $0<\alpha<1/2$ suitably chosen, for any $\gamma>0$ and $T>0$, there exists $C(\gamma,T)>0$ such that for all $0<t<T$,
\begin{align}\label{1.propab}
  \Prob\Big(\max_{i=1,\ldots,N}|X_i^\eta(t)-\bar{X}_i^\eta(t)|
  >N^{-\alpha}\Big) \le C(\gamma,T)N^{-\gamma}.
\end{align}
This provides a convergence result with any algebraic rate. For this result, we need the sub-Coulomb condition $\lambda<d-2$; see Section \ref{sec.key} for more details. The proof uses ideas from \cite{HLP20,LaPi17}, where the Coulomb potential with a cutoff was considered. The main idea relies on a Gronwall-type argument and a law-of-large-numbers estimate; see Lemma \ref{lem.law}. {Existing (quantitative) convergence results \cite{KOdeS24,ORT23,Tom23} concerning particle system \eqref{1.X} do not cover \eqref{1.propab}, since we need an arbitrary algebraic convergence rate with a cutoff $N^{-\alpha}$ depending on the number of particles for $0< \alpha < 1/2$.

\subsection{State of the art}

Since the seminal work of Kac \cite{Kac56} on the many-particle limit of the spatially homogeneous Boltzmann equation, many works have been devoted to interacting particle systems and the associated mean-field limit. Depending on the concrete situation, there exist many types of mean-field limits. Because of the extensive literature, we mention only some of the relevant works and refer to \cite{ChDi22,ChDi22a,JaWa17} for more references.

The fluctuations problem around a limiting mean-field PDE was investigated already in the 1970s and 1980s focusing on jump-type particle systems \cite{McK75,Tan82}. One of the earliest works on fluctuations of interacting diffusions is due to It\^o \cite{Ito83} for one-dimensional Brownian motions, and the limit of the corresponding fluctuations is a Gaussian process. Fluctuations in the path space were first obtained by Tanaka and Hitsuda \cite{TaHi81}, later generalized in \cite{Szn84}. 
It is shown in \cite{FeMe97} that the fluctuations for weakly interacting diffusions are governed by a linear stochastic PDE. Another earlier result is due to Dittrich \cite{Dit88}, who proved the convergence in law of the fluctuations process $\widetilde{\mathcal{F}}^\eta=\sqrt{N}(\mu^\eta-F)$, where $F$ solves a reaction-diffusion equation in the setting of Poisson initial distributions. The effect of spatial damping on the interactions is investigated in \cite{LuSt16}, showing that if the damping is sufficiently weak, the fluctuations remain Gaussian. While the paper \cite{FeMe97} requires smooth kernels, this condition was extended to singular kernels on the torus, including the two-dimensional Biot--Savart and Coulomb kernels, in \cite{WZZ23}. Gaussian fluctuations for singular Coulomb potentials in repulsive settings were studied in \cite{Ser24, BLS18} under some constrains on the dimension by the use of energy estimates. Recently, in \cite{GGK22} a central limit theorem was shown for stochastic gradient decent models in the mean-field setting under smoothness assumptions on the given data of the mean-field equation.

As already mentioned, the fluctuations process around the mean-field limit for \textit{moderately interacting particles} was studied by Oelschl\"ager in \cite{Oel87} in the porous medium setting, i.e.\ $\kappa=-1$ and $\Phi = \delta_0$, as well as by Jourdain and M\'el\'eard in \cite{JoMe98} in a more general setting but still for $\Phi = \delta_0$, however, without the central limit theorem scaling $\sqrt{N}$. Oelschl\"ager showed that the corrected fluctuations process has a Gaussian behavior in the limit $N\to\infty$. Later, Jourdain and M\'el\'eard \cite{JoMe98} studied the fluctuations around the mean-field limit for moderately interacting particles but with a logarithmic scaling instead of $\sqrt{N}$, such that no Gaussian behavior can be obtained.  The authors are not aware of any other work concerning fluctuations for moderately interacting particles. 

Additionally, it seems that no fluctuations result is available in the literature for singular attractive mean-field interaction in the moderate or weak interaction regime.}When finalizing this draft, we came across an announcement in \cite{RS24} regarding a central limit theorem for logarithmic and Riesz flows in the whole space, which is currently in preparation.

In the following, we give a short summary of mean-field limits related to our setting. 

First, we focus on moderately interacting mean-field limits.  Oelschl\"ager introduced in \cite{Oel84,Oel85,Oel87} the concept of moderately interacting particles for gradient systems in the repulsive case, $\Phi =\delta_0$ and $\eta =N^{-\beta}$ with $0<\beta<1/(d+2)$, followed by classical works of M\'el\'eard and co-authors \cite{JoMe98,MeRo87} as well as Sznitman \cite{Szn84, Szn91}. In 2000, Stevens \cite{Ste00} derived a chemotaxis model with an algebraic connection between $\eta$ and $N$. Later, Figalli and Philipowski \cite{FP08} used the ideas of moderately interacting particles to derive the porous medium equation with $\eta = (\log N)^{-\nu}$ for some $\nu >0$.  Recently, moderate regimes have been used to derive cross-diffusion models; see \cite{CG24,CDHJ21,CDJ19,LCZ}. A quantitative propagation-of-chaos result for particle system \eqref{1.X} was shown in \cite{ORT23} using a semigroup approach developed by Flandoli and co-authors \cite{FLO19} in the moderate regime. Additionally, in a slightly different setting, the work \cite{ORT20} also considers interaction kernels which approximate Riesz potentials.

Second, the choice of $\kappa$ distinguishes the type of force. The repulsive case $\kappa=-1$ was first studied by Oelschl\"ager \cite{Oel85,Oel87} for moderately interacting particles, $\Phi = \delta_0$. Under the assumption that the regularized kernel $\na V^\eta$, which approximates a Dirac delta, is Lipschitz continuous for fixed $\eta>0$, it is well known that propagation of chaos holds; see, e.g., \cite{JoMe98,MeRo87,Szn91}. However, except for \cite{Oel87}, these results require either a logarithmic scaling of $\eta$ or they are not quantitative. The first result for aggregating forces $\kappa=1$ is due to Stevens \cite{Ste00}, who proved the mean-square convergence in probability for smoothed quantities. The (stronger) mean-square convergence in the path space under a (weaker) logarithmic scaling between $\eta$ and $N$ was derived in \cite{CGK18}. Recently, in \cite{ORT20, ORT23} quantitative results for Riesz kernels in the moderate setting for $\kappa=\pm 1$ were obtained.

Third, physical applications often require a singular potential. Earlier works were concerned with the derivation of Navier--Stokes equations from particle systems with logarithmic potentials in two space dimensions \cite{MaPu82,Osa86}. Propagation of chaos for the Landau equation with moderately soft potentials was shown in \cite{FoHa16}. Coulomb potentials (with cutoff) are allowed in the derivation of the Vlasov--Poisson--Fokker--Planck equation in $\R^3$ in the weakly interacting regime \cite{CCS19,HLP20}. The mean-field limit for weakly interacting particles with (repulsive) Coulomb or (attractive) Newton potentials towards the Vlasov--Poisson system was proved in \cite{LaPi17}. The results of Jabin and Wang \cite{JaWa16} were extended to Riesz potentials \cite{Due16} and to Coulomb as well as super-Coulomb potentials \cite{Ser20} using the modulated energy argument inspired by \cite{Ser17}. The modulated energy of Serfaty \cite{Ser17} and the relative entropy of Jabin and Wang \cite{JaWa16} were recently combined by Bresch et al.\ \cite{BJW23} to treat Coulomb potentials via the modulated free energy method. Very recently, a combination of modulated energy and relative entropy methods was used in \cite{RS24} to derive a quantitative convergence result on the whole space for Riesz potentials. A related work is \cite{BCG23}, where the $L^1(\R)$ convergence of the primitive of the empirical measure including time-dependent weights was proved for opinion dynamics with Coulomb interactions. More general noise terms (white in time and colored in space) were considered in \cite{GuLu23} for Biot--Savart or repulsive Poisson kernels.

Fourth, the results differ according to the type of convergence:
convergence in law of the sequence of empirical measures, e.g.\ in \cite{Oel85,MeRo87} (which is equivalent to the propagation of chaos in the sense of \cite[Prop.~2.2]{Szn91}); convergence in probability of the empirical measures \cite{Ste00,Tom23}; convergence in probability in the $p$-Wasserstein distance of probability measures \cite{LaPi17} (see also the review \cite{CCP17} for second order systems); mean-square convergence in expectation of smoothed empirical measures \cite{Oel87,Oel90};  mean-square convergence in expectation of the stochastic processes \cite{CDHJ21,CDJ19}; convergence of the marginals of the $N$-particle distribution in the $L^1(\R^d)$ norm \cite{JaWa16}; convergence of the empirical measures in the relative entropy \cite{CHH23}. Tightness results, not providing convergence rates, were employed in \cite{JTT18,Tom23} to prove the convergence of the empirical measures in the weak sense. Further details and references can be found in \cite{ChDi22a}. 

Fifth, the interaction strength $\eta$ may depend on the particle number in a logarithmic way ($\eta^{-\beta}=\log N$ for some $\beta>0$) or in an algebraic way ($\eta=N^{-\beta}$). In this article, we consider the more natural algebraic rate, leading to a central limit theorem for moderate regimes.

Recent works are concerned with {\em quantitative} convergence results for attractive, singular kernels in the moderately interacting regime. The convergence in probability for the Newtonian $N$-particle flow towards the lifted effective flow for any algebraic rate was shown in \cite{FePi23}. The proof is based on a Gronwall estimate for the maximal distance between the exact microscopic and the mean-field dynamics. The $L^p(\R^d)$ convergence of smoothed empirical measures was proved in \cite{ORT20,ORT23} with the rate $N^{-\gamma}$ for some $\gamma<1/(d+2)$, whose precise value depends on $p$ and the singularity of the (super-Coulomb) potential. The proof uses semigroup theory and the regularity of stochastic convolution integrals.

Finally, we mention the recent papers on small noise fluctuations of solutions to Dean--Kawasaki stochastic PDEs. It was shown in \cite{CoFi23,CFIR23} that the law of fluctuations predicted by a discretized Dean--Kawasaki equation approximates the law of fluctuations of the original weakly interacting particle system. The convergence in probability of density fluctuations associated to a Dean--Kawasaki equation towards the solution to a linearized Langevin equation was proved in \cite{ClFe23}. 

Our work combines some features of the literature: We allow for attractive, singular potentials and obtain the strong notion of mean-square convergence in expectation of order $N^{-1/2-\eps}$ needed for a central limit theorem to hold. We prove the convergence of smoothed empirical measures associated to \eqref{1.X} towards smoothed solutions to the nonlocal diffusion equation \eqref{1.baru}, which implies a fluctuations result and shows that the fluctuations become asymptotically Gaussian in the large population limit $N\to \infty$. The potentials need to be of sub-Coulomb type. The originality of our results is that we determine an explicit algebraic convergence rate sufficient to study fluctuations around the mean-field limit. 

\subsection{Main results}\label{sec.main}

We impose the following assumptions.

\begin{itemize}
\item[(A1)] Parameters: $d\ge 3$, $\sigma>0$, $\kappa=\pm 1$, $T>0$.
\item[(A2)] $W_1,\ldots,W_N$ are independent $d$-dimensional Brownian motions on the filtered probability space $(\Omega,\mathscr{A},(\mathscr{A}_t)_{t\ge 0},\Prob)$, and the initial data $\zeta_1,\ldots,\zeta_N$ are $\mathscr{A}^0$-measurable independent and identically distributed (i.i.d.) square-integrable random variables with the common density function $u_0 \in L^1(\R^d)\cap L^\infty(\R^d)$, where $\mathscr{A}^0\subset \mathscr{A}$ is independent of $\mathscr{A}_t$ for $t\geq 0.$
\item[(A3)] Sub-Coulomb potential: $\Phi(x)=|x|^{-\lambda}$ for $x\in\R^d$ with $0<\lambda<d-2$.
\item[(A4a)] Initial data: Let the initial data $u_0$ fulfill $\|u^0\|_{L^{p^*}(\R^d)}< C(p^*)$, where $p^*=d/(d-\lambda)$ and $C(p^*)>0$ depends on $\sigma$, $d$, $\lambda$, and the constants of some Sobolev inequalities; see Lemma \ref{lem.wLinfty}.
\item[(A4b)] Stronger assumption on the initial data: $u_0\in W^{K,1}(\R^d)$ $\cap W^{K,\infty}(\R^d)$ (needed in Theorem \ref{thm.fluct}) with $\Vert u_0 \Vert_{L^{p*}(\R^d)} < C(p^*)$, where $K>1/(2\beta)$.
\end{itemize}

Let $\xi\in C_0^\infty(\R^d)$ be a radially symmetric probability density function, i.e.\ $\xi(x)=\xi(|x|)$. We introduce the mollification kernel $\xi^\eta(x)=\eta^{-d}\xi(x/\eta)$ for $x\in\R^d$ and set $\chi= \xi*\xi$. A simple computation shows that $\chi^\eta = \xi^\eta*\xi^\eta$. We set $V^\eta:=\chi^\eta*\Phi$. Then we can write
\begin{equation*}
  V^\eta = Z^\eta*Z^\eta,\quad\mbox{where }Z^\eta = \xi^\eta*\Psi,
\end{equation*}
where $\Psi(x)=c_{d,\lambda}|x|^{-(\lambda+d)/2}$ for some constant $c_{d,\lambda}\in \R$; see Lemma \ref{lem.Psi}. Our first result concerns the convergence in probability with any algebraic rate.

\begin{theorem}[Convergence in probability]
\label{thm.prob}
Let Assumptions (A1)--(A4a) hold and let
$$
  0 < \beta < \frac{1}{4\lambda+12}, \quad
  \beta(\lambda +3) < \alpha < \frac12 - \beta(\lambda+1).
$$
Let $(X_i^\eta)_{i=1,\ldots,N}$ and $(\bar{X}_i^\eta)_{i=1,\ldots,N}$ be the solutions to systems \eqref{1.X} and \eqref{1.barX}, respectively. Then for any $\gamma>0$ and $T>0$, there exists a constant $C(\gamma,T)>0$ such that for all $0<t<T$ and $N\in\N$,
\begin{equation*}
  \Prob\Big(\max_{i=1,\ldots,N}|X_i^\eta(t)-\bar{X}_i^\eta(t)|
  > N^{-\alpha}\Big) \le C(\gamma,T)N^{-\gamma}.
\end{equation*}
\end{theorem}

The choice of $\beta$ implies that the admissible interval for $\alpha$ is not empty. If $\beta$ is chosen close to zero, the parameter $\alpha$ can be chosen in the interval $(0,1/2)$, while for $\beta$ close to the maximal value $1/(4\lambda+12)$ and large values of $\lambda$ (hence higher space dimensions), $\alpha$ needs to be chosen around $1/4$.

\begin{remark}[Sub-Coulomb potentials]\label{rem.subcoulomb}\rm
The proof of Theorem \ref{thm.prob} is based on an estimate of the form $\||D^2 V^\eta|*\bar{u}^\eta (t)\|_{L^\infty(\R^d)} \leq C$. This can only be obtained for $\lambda<d-2$ which limits our framework to sub-Coulomb potentials. In the case $\lambda = d-2$ with an additional cut-off around the origin, the estimate $\||D^2 V^\eta|* \bar{u}^\eta (t) \|_{L^\infty(\R^d)} \leq C \log N$ can be shown, which is still sufficient to derive convergence in probability; see \cite{LaPi17}. This cut-off, however, violates the convolution structure $V^\eta = Z^\eta \ast Z^\eta$ and cannot be easily used for the proof of the $L^2(\R^d)$ convergence; see Theorem \ref{thm.L2}. Using a suitable cut-off, this result also holds for Coulomb potentials.
\qed\end{remark}

The next result of this paper is the convergence of the smoothed empirical measures to a smoothed solution to the intermediate diffusion equation \eqref{1.baru}. We recall the definitions
\begin{equation}\label{1.fg}
  f^\eta := Z^\eta*\mu^\eta, \quad g^\eta := Z^\eta*\bar{u}^\eta,
\end{equation}
where $\mu^\eta$ is defined in \eqref{1.mu} and $\bar{u}^\eta$ solves \eqref{1.baru}.

\begin{theorem}[Mean-square convergence with $N^{-1/2-\eps}$ rate]\label{thm.L2}
Let Assumptions (A1)--(A4a) hold and let 
$$
  0 < \beta < \frac{1}{8\lambda+12}. 
$$
Then for any $T>0$, there exists $\eps>0$ and a constant $C=C(\beta,d,T)>0$ such that, for all $N\in\N$,
$$
  \E\bigg(\sup_{0<t<T}\|(f^\eta-g^\eta)(t)\|_{L^2(\R^d)}^2
  + \sigma\int_0^T\|\na(f^\eta-g^\eta)(t)\|_{L^2(\R^d)}^2\dd t\bigg)
  \leq CN^{-1/2-\eps}.
$$
\end{theorem}

A similar theorem was proved by Oelschl\"ager \cite{Oel87} for the so-called corrected fluctuations in the repulsive case $\kappa=-1$ and with $\Phi = \delta_0$. With this choice of $\kappa$, a certain quadratic expression (see \eqref{4.M}) is nonpositive and can be neglected. This is no longer possible if $\kappa=1$ (attractive case). We are able to treat this case by exploiting the mean-field convergence in probability (Theorem \ref{thm.prob}), which helps us to absorb the terms with $\kappa=1$. 

\begin{remark}\label{rem.general}\rm
Theorem \ref{thm.L2} can be generalized to a more general class of potentials $\Phi$. Indeed, let $\Phi=\Psi*\Psi$, where $\Psi$ is symmetric and nonnegative. Additionally, we assume that for $0<\lambda<d-2$, there exists $C>0$ such that for all $0<t<T$ and $k=0,1,2$,
\begin{equation}\label{1.prop}
\begin{aligned}
  & \|\mathrm{D}^k V^\eta\|_{L^\infty(\R^d)} \le CN^{\beta(\lambda+k)}, 
  \\
  & \|Z^\eta\|_{L^2(\R^d)} \le CN^{\beta \lambda/2}, \
  \|\mathrm{D}^k V^\eta*\bar{u}^\eta(t)\|_{L^\infty(\R^d)} \le C, \\
  &\|\na Z^\eta\|_{L^2(\R^d)} \le CN^{\beta (\lambda+1)/2}, \quad
  \||\mathrm{D}^2 V^\eta|*\bar{u}^\eta(t)\|_{L^\infty(\R^d)} \le C.
\end{aligned}
\end{equation}
Then the result of Theorem \ref{thm.L2} holds. We show in Section \ref{sec.riesz} that the Riesz kernels $\Phi(x)=|x|^{-\lambda}$ satisfy \eqref{1.prop}. As mentioned in Remark \ref{rem.subcoulomb}, the last estimate in \eqref{1.prop} can only be obtained for $\lambda<d-2$, which limits our framework to the sub-Coulomb case. This estimate is needed in our proof of convergence in probability (see Theorem \ref{thm.prob}). If Theorem \ref{thm.prob} holds also for $\lambda=d-2$, then the proof of fluctuations around the mean-field limit can be extended to Coulomb potentials.
\qed\end{remark}

Our final result concerns the central limit theorem for the intermediate fluctuations process, which is a consequence of Theorem \ref{thm.L2}.

\begin{theorem}[Asymptotic fluctuations limit via characteristic function]
\label{thm.fluct}\
Let the assump\-tions of Theorem \ref{thm.L2} and additionally Assumption (A4b) hold, and let $\mathcal{F}_0$ be the Gaussian random distribution such that
$$
  \E|\langle\mathcal{F}^\eta(0),\phi\rangle-\langle\mathcal{F}_0,
  \phi\rangle|\to 0\quad\mbox{as }\eta=N^{-\beta}\to 0
$$ 
for all Schwartz functions $\phi$. Then for all Schwartz functions $\phi$, $0\le s<t<\infty$, and $\theta\in\R$,
\begin{align*}
  \E\bigg|\E\bigg(\exp\big[\mathrm{i}\theta\langle\mathcal{F}^\eta(t),
  \phi\rangle \big]\bigg|\mathscr{F}_s\bigg) - \exp\bigg(\mathrm{i}\theta
  \langle\mathcal{F}^\eta(s),T_\phi^t(s)\rangle 
   - \sigma\theta^2\int_s^t\langle u(r),|\na T_\phi^t(r)|^2\rangle
  \dd r \bigg)\bigg|\to 0
\end{align*}
as $\eta=N^{-\beta}\to 0$, where $u$ solves \eqref{1.u}, 
$$
  \mathscr{F}_s = \sigma\big(\langle\mathcal{F}^\eta(r),
  \phi \rangle: 0\leq r\leq s, \phi \mbox{ is a Schwartz function}\big),
$$
$u$ is the solution to \eqref{1.u}, $T_\phi^t$ is the solution to the backward problem $\mathcal{L}^*\psi=0$ in $\R^d\times(0,t)$, $\psi(t)=\phi$, and $\mathcal{L}^*$ is the dual of the operator $\mathcal{L}$ of the linearized version of \eqref{1.u} (i.e.\ $T_\phi^t$ solves \eqref{1.back}).
%and $\mathcal{N}(\mu_0,\sigma_0^2)$ denotes the normal distribution with mean $\mu_0$ and standard deviation $\sigma_0$.
\end{theorem}

{We observe that, since the initial conditions are i.i.d., $\langle\mathcal{F}_0, \phi\rangle$ is indeed Gaussian with mean zero and variance $\langle u_0,\phi^2\rangle - \langle u_0, \phi\rangle^2$.
The proof of Theorem \ref{thm.fluct} is based on the $N^{-1/2-\eps}$ decay from Theorem \ref{thm.L2} and uses the strategy of \cite[Sec.~3]{Oel87} to estimate the characteristic function of the intermediate fluctuations process. An immediate consequence of Theorem \ref{thm.fluct} is the following corollary.

\begin{corollary}[Asymptotic central limit theorem] 
Let the assumptions of Theorem \ref{thm.fluct} hold. Then, for any Schwartz function $\phi$ and $t>0$, as $\eta=N^{-\beta}\to 0$,
\begin{align*}
  \mathrm{Law}(\langle\mathcal{F}^\eta(t),\phi\rangle|\mathscr{F}_0)
  \to \mathcal{N}\bigg(\langle\mathcal{F}_0,T_\phi^t(0)\rangle,
  2\sigma\int_0^t\langle u(s),|\na T_\phi^t(s)|^2\rangle\dd s\bigg).
\end{align*}
\end{corollary}

Theorem \ref{thm.fluct} shows that the (asymptotic) limiting fluctuations process is normally distributed. Indeed, it follows from this theorem and the fact that the initial conditions are i.i.d.\ that the asymptotic limit $\mathcal{F}$ fulfills \eqref{1.fluctuations}, i.e.
\begin{align*}
  \langle \mathcal{F}(t), \phi \rangle 
  \sim \mathcal{N}\bigg(0,\langle u_0,(T^t_\phi (0))^2 \rangle 
  - \langle u_0, T^t_\phi (0)\rangle^2 
  + 2\sigma\int_{0}^t \langle u(s), |\nabla T^t_\phi (s)|^2\rangle 
  \dd s\bigg).
\end{align*}
Formally, we can expect that the limiting distribution $\mathcal{F}$ of $\mathcal{F}^\eta$ is a generalized Ornstein--Uhlenbeck process, namely a formal solution to the following stochastic PDE:
\begin{align*}
  \partial_t\mathcal{F} = \sigma\Delta\mathcal{F} 
  - \kappa\diver(\mathcal{F}\na\Phi*u+u\na\Phi*\mathcal{F})
  - \sqrt{2\sigma}\nabla \cdot(\sqrt{u}\xi),
\end{align*} 
in $\R^d\times(0,T)$ with initial datum $\mathcal{F}(0)=\mathcal{F}_0$, where $\xi\in L^2((0,T)\times\R^d;\R^d)\to L^2(\Omega;\R^d)$ is a vector-valued space-time white noise and $u$ solves \eqref{1.u}.  

The paper is organized as follows. We sketch the technical proofs of the main theorems in Section \ref{sec.key}. The law-of-large-numbers estimate is shown in Section \ref{sec.law}. Our first result, the mean-field estimate of Theorem \ref{thm.prob}, is proved in Section \ref{sec.prob}. With these preparations, the lengthy proof of Theorem \ref{thm.L2} is detailed in Section \ref{sec.quant}, and Theorem \ref{thm.fluct} is shown in Section \ref{sec.fluct}. Some technical proofs are deferred to the appendices: Estimates \eqref{1.prop} for the Riesz kernels \eqref{1.Veta} are proved in Appendix \ref{sec.riesz}, while the uniform bounds for the solutions to \eqref{1.baru} and the solution $T_\phi^t$ to the backward dual problem are investigated in Appendix \ref{sec.pde}.

%%%%%%%%%%%%%%%%%%%%%%%%%%%%%%%%%%%%%%%%%%%%%%%%%%%%%%%%%%%%%%%

\section{Ideas of the proofs}\label{sec.key}

We show the central limit theorem (Theorem \ref{thm.fluct}) by exploiting the $L^2(\R^d)$ convergence for the convoluted quantities $Z^\eta\ast(\mu^\eta - \bar{u}^\eta)$ (Theorem \ref{thm.L2}) with rate $N^{-1/2-\eps}$ for some $\eps >0$. This $L^2(\R^d)$ convergence for aggregating or repulsive particles can be obtained by using a quantitative convergence-in-probability result for the difference of the solutions to particle system \eqref{1.X} and the smoothed intermediate particle system \eqref{1.barX} (Theorem \ref{thm.prob}). In the following, we explain the key ideas for our three main theorems.

\subsection{Central limit theorem} 

The proof of the central limit theorem (Theorem \ref{thm.fluct}) follows ideas of \cite{Oel87}. Using the results of Theorem \ref{thm.L2}, the main idea is the calculate the expected value of the (conditional) characteristic function $\mathbb{E}(\exp(\mathrm{i}\theta\langle \mathcal{F}^\eta(t), \phi \rangle)|\mathscr{F}_s)$ for a given smooth test function $\phi$ (which does not depend on time) and $\theta \in \R$. This conditional characteristic function corresponds to the unique conditional law of an Ornstein--Uhlenbeck process; see, e.g., \cite{HS78} or \cite{FGN13}.

First, we show that, for any test function $\psi(x,t)$,
\begin{align}\label{2.F}
  &\dd\langle\mathcal{F}^\eta,\psi\rangle
  = -\langle\mathcal{F}^\eta,\mathcal{L}^*\psi\rangle\dd t
  + \frac{\sqrt{2\sigma}}{\sqrt{N}}\sum_{i=1}^N\na\psi(X_i^\eta(t))
  \cdot\dd W_i(t) \\
  &\phantom{x}+ \kappa\bigg(\frac{1}{\sqrt{N}}\langle\mathcal{F}^\eta,
  (\na V^\eta*\mathcal{F}^\eta)\cdot\na\psi\rangle
  + \big\langle\mathcal{F}^\eta,(\na V^\eta*\bar{u}^\eta)
  \cdot\na\psi - \na V^\eta*(\bar{u}^\eta\na\psi)\big\rangle \nonumber\\
  &\phantom{x}- \big\langle\mathcal{F}^\eta,
  (\na\Phi*u)\cdot\na\psi - \na\Phi*(u\na\psi)\big\rangle\bigg)\dd t, \nonumber
\end{align} 
where $\mathcal{L}^*\psi = -\pa_s\psi - \sigma\Delta\psi 
- \kappa((\na\Phi*u)\cdot\na\psi - \na\Phi*(u\na\psi))$; see Lemma \ref{lem.dF}. Let $T_\phi^t(s)$ be for fixed test function $\phi$ the dual linear backwards time evolution process; see Lemma \ref{lem.estT}. The choice $\psi(x,s) = T_\phi^t(s)$ for a smooth function $\phi$ leads to $ \mathcal{L}^*\psi=0$, and the first term on the right-hand side of \eqref{2.F} vanishes, while the other terms on the right-hand side converge to zero in expectation.

Next, by Theorem \ref{thm.L2} and a suitable iteration argument detailed in Lemma \ref{lem.F23}), it is shown that the last three terms in \eqref{2.F} do not contribute to the limiting fluctuation behavior, i.e.
\begin{align*}
    \E\bigg|\E\big(\exp(\mathrm{i}\theta\langle\mathcal{F}^\eta(t),
    \phi\rangle)|\mathscr{F}_s\big)
    - \exp\bigg(\mathrm{i}\theta\langle\mathcal{F}^\eta(s),T_\phi^t(s)
    \rangle-\theta^2 \sigma\int_s^t\langle u(r),|\na T^t_\phi(r)|^2\rangle\dd r\bigg)\bigg|\to 0.
\end{align*}
This shows that (since at $t=0$ a classical central limit theorem holds) at any time $t>0$, the asymptotic fluctuation behavior of $\langle \mathcal{F}^\eta(t), \phi \rangle$ remains Gaussian for $N\to \infty$ with mean zero and variance $\langle u_0,(T^t_\phi (0))^2 \rangle 
- \langle u_0, T^t_\phi (0)\rangle^2 
+ 2\sigma\int_{0}^{t}\langle u(s), |\nabla T_\phi^t(s)|^2 \rangle\dd s$. We refer to Section \ref{sec.fluct} for more details.

\subsection{Mean-field limit in $L^2(\R^d)$}

The proof of Theorem \ref{thm.L2} is based on two ideas: a {\em law-of-large-numbers estimate}, which measures the difference of the arithmetic mean and the expectation with respect to $\bar{X}_i^\eta$; and a {\em mean-field estimate}, expressing the convergence of $X_i^\eta-\bar{X}_i^\eta$ in probability (Theorem \ref{thm.prob}).

\subsubsection*{Law-of-large-numbers estimate} 

Roughly speaking, we wish to derive estimates for the probability that the difference between the arithmetic mean of some particle interactions and the convolution with the common density function lies outside a ball of radius $N^{-\theta}$ for any $\theta\ge 0$. More precisely, let
$$
  \mathcal{B}_{\theta,\psi_\eta}^N(t) 
  := \bigcup_{i=1}^N\bigg\{\omega\in\Omega:
  \bigg|\frac{1}{N}\sum_{j=1}^N
  \psi_\eta(\bar{X}_i^\eta(t)-\bar{X}_j^\eta(t)) 
  - (\psi_\eta*\bar{u}^\eta)(\bar{X}_i^\eta(t))\bigg|
  > N^{-\theta}\bigg\},
$$
where $\psi_\eta$ is a given bounded function. We claim in Lemma \ref{lem.law} below that, for any $m\in\N$ and $T>0$, there exists $C>0$ such that
\begin{equation}\label{1.probA}
  \Prob(\mathcal{B}_{\theta,\psi_\eta}^N(t)) 
  \le C(m)\|\psi_\eta\|_{L^\infty(\R^d)}^{2m}N^{2m(\theta-1/2)+1}.
\end{equation}
A possible choice is $\psi_\eta=\na V^\eta$, whose $L^\infty(\R^d)$ norm is bounded by $N^{\beta(\lambda+1)}$ (see Lemma \ref{lem.DkV}). Choosing $0<\theta<1/2$ and $m\in\N$ sufficiently large, inequality \eqref{1.probA} provides an estimate with factor $N^{-\gamma}$ for any $\gamma>0$. This factor allows us to absorb in the proof of Theorem \ref{thm.L2} other terms growing algebraically with $N$.

We prove estimate \eqref{1.probA} by applying Markov's inequality:
$$
  \Prob(\mathcal{B}_{\theta,\psi_\eta}^N(t))\le N^{2m\theta+1}
	\E\bigg(\frac{1}{N^{2m}}\bigg|\sum_{j=1}^N h_{ij}(t)\bigg|^{2m}\bigg)
	= N^{2m(\theta-1)+1}\E\bigg(\bigg(\sum_{j,k=1}^N
	h_{ij}h_{ik}\bigg)^m\bigg),
$$
where $h_{ij}=\psi_\eta(\bar X_i^\eta-\bar X_j^\eta)
-(\psi_\eta*\bar{u}^\eta)(\bar{X}^\eta_i)$. Note that due the symmetry of particle system \eqref{1.X}, the right-hand side is independent of $i=1,\ldots, N$. It turns out that the expectation on the right-hand side vanishes if $h_{ij}$ appears only once in the product and that the cardinality of all products $\prod_{k=1}^{2m}h_{ij_k}$, where $h_{ij_k}$ appears at least twice, is bounded by $N^m$ (up to some multiplicative constant). As each of the products is bounded by $\|\psi_\eta\|_{L^\infty(\R^d)}^{2m}$, we conclude estimate \eqref{1.probA}. 

\subsubsection*{Mean-field estimate} 

We sketch the proof of Theorem \ref{thm.prob}, the convergence in probability
$$
  \Prob\Big(\max_{i=1,\ldots,N}|X_i^\eta(t)-\bar{X}_i^\eta(t)|
  > N^{-\alpha}\Big) \le C(\gamma,T)N^{-\gamma},
$$
for arbitrary $\gamma>0$. To this end, we apply first Markov's inequality:
\begin{align*}
  \Prob\Big(\max_{i=1,\ldots,N}
  |X_i^{\eta}(t)-\bar X_i^\eta(t)|>N^{-\alpha}\Big)
  \le \E(S_\alpha^k(t)),
\end{align*}
where $S_\alpha^k(t)=N^{\alpha k}\max_{i=1,\ldots,N}|(X_i^{\eta}-\bar X_i^\eta)(t\wedge\tau_\alpha)|^k$, $\tau_\alpha$ is a suitable stopping time such that $|S_\alpha^k| \leq 1$, and $k\in\N$ is an arbitrary number. To estimate the expectation of $S_\alpha^k(t)$, we insert equations \eqref{1.X} and \eqref{1.barX} for $X_i^\eta$ and $\bar{X}_i^\eta$, respectively, and add and subtract $\na V^\eta(\bar X_i^\eta-\bar X_j^\eta)$. Then we need to estimate for $i=1,\ldots,N$,
\begin{align*}
  Q_1 &= \E\bigg(N^{\alpha k}\int_0^t\bigg|\frac{1}{N}\sum_{j=1}^N
	\big[\na V^\eta((X_i^{\eta}-X_j^{\eta})(s))
	- \na V^\eta((\bar X_i^\eta-\bar X_j^\eta)(s))\big]\bigg|^k\dd s\bigg), \\
  Q_2 &= \E\bigg(N^{\alpha k}\int_0^t\bigg|\frac{1}{N}\sum_{j=1}^N
	\na V^\eta((\bar X_i^{\eta}-\bar X_j^{\eta})(s))
	- (\na V^\eta*\bar u^\eta)(s,\bar X_i^\eta(s))\bigg|^k\dd s\bigg).
\end{align*}

The term $Q_2$ is estimated by a law-of-large-numbers argument, i.e., we apply estimate \eqref{1.probA}. The idea is to split $\Omega$ into the set $\mathcal{B}:=\mathcal{B}_{\theta,\psi_\eta}^N$ with $\psi_\eta=\na V^\eta$ and its complement $\mathcal{B}^c$. By definition of $\mathcal{B}^c$, the integrand of $Q_2$ is bounded on this set by $CN^{-\theta k}$, while on the set $\mathcal{B}$, the term $Q_2$ is bounded by $\|\na V^\eta\|_{L^\infty(\R^d)}^k\le CN^{\beta(\lambda+1) k}$ (see Lemma \ref{lem.DkV}) and by $\Prob(\mathcal{B})\le C\|\na V^\eta\|_{L^\infty(\R^d)}^{2m}N^{2m(\theta-1/2)+1}$ (see Lemma \ref{lem.law} or estimate \eqref{1.probA}). Hence, 
\begin{align*}
  Q_2 &\le CN^{\alpha k}\big(N^{-\theta k} 
  + N^{\beta (\lambda+1) k + 2m(\theta+\beta (\lambda+1)-1/2)+1}\big).
\end{align*}
The exponent becomes negative (with an arbitrary value) if we choose $k\in\N$ and $m\in\N$ sufficiently large, leading to $Q_2\le CN^{-\gamma}$ for any $\gamma>0$. We refer to the proof of Theorem \ref{thm.prob} for more details on the choice of the parameters.

The estimate of $Q_1$ is more involved. The idea is to perform a Taylor expansion of $\na V^\eta$ around $\bar{X}_i^\eta-\bar{X}_i^\eta$ up to second order. After some manipulations, the most delicate term is
\begin{align}\label{1.Q3}
  Q_3 = \E\bigg(\int_0^{t\wedge\tau_\alpha} S_\alpha^k(s)
    \max_{i=1,\ldots,N}\bigg(\frac{1}{N}\sum_{j=1}^N\big|\mathrm{D}^2
    V^\eta(\bar{X}_i^\eta(s)-\bar{X}_j^\eta(s))\big|\bigg)^k\dd s\bigg).
\end{align} 
This expression can be controlled by estimating the convolution $|D^2 V^\eta| \ast \bar{u}^\eta $. Terms like $\mathrm{D}^2 V^\eta*\bar{u}^\eta
= V^\eta*\mathrm{D}^2\bar{u}^\eta$ are bounded if $\bar{u}^\eta$ is smooth and if $0<\lambda<d$ (see Lemma \ref{lem.DkVbaru}). However, since the absolute value in \eqref{1.Q3} is inside the convolution, we need the stronger condition $0<\lambda<d-2$ to estimate the modulus $|\mathrm{D}^2 V^\eta|$. Under this sub-Coulomb assumption, the expression $|\mathrm{D}^2 V^\eta|*\bar{u}^\eta$ is bounded. After further manipulations, we find that
\begin{align*}
  Q_3 \le C\||\mathrm{D}^2V^\eta|*\bar{u}^\eta\|_{L^\infty(\R^d)}^k
  \int_0^{t}\E(S_\alpha^k(s))\dd s 
  \le C\int_0^{t}\E(S_\alpha^k(s))\dd s.
\end{align*}
Thus, we obtain finally
\begin{equation*}%\label{1.final}
  \Prob\Big(\max_{i=1,\ldots,N}
  |X_i^{\eta}(t)-\bar X_i^\eta(t)|>N^{-\alpha}\Big)
  \le \E(S_\alpha^k(t)) \le CN^{-\gamma} 
  + C\int_0^{t}\E(S_\alpha^k(s))\dd s.
\end{equation*}
This leads -- after applying Gronwall's lemma -- to the desired result. Observe that the sub-Coulomb condition is only needed in the estimate of $Q_3$.

\subsubsection*{Reformulation and estimation}

After these preparations, we sketch the proof of Theorem \ref{thm.L2}. We wish to estimate $f^{\eta}-g^\eta
= Z^\eta*(\mu^\eta-\bar{u}^\eta)$ in the $L^2(\R^d)$ norm. For this, we reformulate this difference by using It\^o's formula for the stochastic process $X_i^{\eta}$ and by inserting definition \eqref{1.mu} of the empirical measure $\mu^{\eta}$ to find after some reformulations detailed in Section \ref{sec.quant} that
\begin{align}\label{1.KLM}
  \E\bigg(&\sup_{0<\tau<t}\|(f^{\eta}-g^\eta)(\tau)\|_{L^2(\R^d)}^2
  - \|(f^{\eta}-g^\eta)(0)\|_{L^2(\R^d)}^2\bigg) \\
  &+ 2\sigma\E\int_0^t\|\na(f^{\eta}-g^{\eta})\|_{L^2(\R^d)}^2\dd s
  \leq \frac{2\sigma t}{N}\big|\Delta V^\eta(0)\big| 
  + K(t) + L(t) + M(t), \nonumber
\end{align}
where
\begin{align*}
  K(t) &= \sqrt{8\sigma}\E\bigg(\sup_{0<\tau<t}
  \frac{1}{N}\sum_{i=1}^N\int_0^t
  \big(\nabla V^\eta *(\mu^{\eta}-\bar u^\eta)\big)
  (s,X_i^{\eta}(s))\dd W_i(s)\bigg), \\	
  L(t) &= 2\kappa\E\bigg(\sup_{0<\tau<t}\int_0^\tau\big\langle
  \mu^{\eta}-\bar u^\eta, (\na V^\eta*\bar u^\eta)
  \cdot(\na Z^\eta*(f^{\eta}-g^\eta))\big\rangle\dd s\bigg), \\
  M(t) &= 2\kappa\E\bigg(\sup_{0<\tau<t}\int_0^\tau\big\langle
  \mu^{\eta},|\na Z^\eta*(f^{\eta}-g^\eta)|^2\big\rangle\dd s\bigg),
\end{align*}
and $\langle\cdot,\cdot\rangle$ is the dual product. The idea is to estimate each of the terms on the right-hand side of \eqref{1.KLM} such that they are either of order $N^{-1/2-\eps}$ or can be absorbed by the gradient term on the left-hand side of \eqref{1.KLM}. Observe that the condition $\sigma>0$ is needed here; see Remark \ref{rem.sigma} for a generalization.

By Lemma \ref{lem.DkV} (bound for $\mathrm{D}^k V^\eta$), the first term on the right-hand side of \eqref{1.KLM} is bounded by $CN^{\beta(\lambda+2)-1}$. This expression is of order $N^{-1/2-\eps}$ for some $\eps>0$ if we assume that $\beta<1/(2\lambda+4)$. 
By the Burkholder--Davis--Gundy inequality, the stochastic integral 
$K(t)$ can be estimated by $C/N+M(t)$, such that it remains to
estimate $L(t)$ and $M(t)$. 

For these terms, we employ similar techniques as described in the previous subsections. Indeed, we add and subtract several terms in $M(t)$ to split the difficulty into mean-field estimates between particle system \eqref{1.X} and intermediate system \eqref{1.barX} as well as  law-of-large numbers estimates between system \eqref{1.barX} and intermediate PDE \eqref{1.baru}. We apply the law-of-large-numbers argument of Lemma \ref{lem.law} and the mean-field estimate of Theorem \ref{thm.prob} applied to the set $\mathcal{C}_\alpha=\{\omega\in\Omega:\max_{i=1,\ldots,N}|X_i^\eta-\bar{X}_i^\eta|>N^{-\alpha}\}$. The arguments are possible by choosing $\beta>0$ sufficiently small (with an explicit bound depending on $\lambda$) and by taking $1/4 + \beta c_1<\alpha<1/2 - \beta c_2$ for some positive constants $c_1$ and $c_2$ depending on $d$ and $\lambda$. After some computations, this leads to
$$
  M(t) \le C(t)N^{-1/2-\eps}, \quad
  L(t) \le C(t)N^{-1/2-\eps} 
  + \sigma\mathbb{E}\int_0^t\|\na(f^\eta-g^\eta)\|_{L^2(\R^d)}^2\dd s,
$$
and the last integral can be absorbed by the left-hand side of \eqref{1.KLM}. Putting the estimates together, we deduce from \eqref{1.KLM} that
\begin{align*}
  \E\bigg(&\sup_{0<\tau<t}\|(f^{\eta}-g^\eta)(\tau)\|_{L^2(\R^d)}^2
  \bigg) + \sigma\E\int_0^t\|\na(f^{\eta}-g^{\eta})
  \|_{L^2(\R^d)}^2\dd s \\
  &\le C(T)N^{-1/2-\eps} + \E\|(f^{\eta}-g^\eta)(0)\|_{L^2(\R^d)}^2,
\end{align*}
and an estimate of $(f^\eta-g^\eta)(0)$ concludes the proof of Theorem \ref{thm.L2}. 

The proof of Theorem \ref{thm.L2} heavily relies on the mean-field estimate of Theorem \ref{thm.prob} and hence on the sub-Coulomb condition $0<\lambda<d-2$. Whenever the convergence in probability of Theorem \ref{thm.prob} can be proved, we infer the result of Theorem \ref{thm.L2}. In other words, the sub-Coulomb restriction is not directly needed in the proof of Theorem \ref{thm.L2}.

\section{Law of large numbers}\label{sec.law}

We prove two estimates in probability for the i.i.d.\ processes $\bar{X}_i^\eta(t)$ and their associated laws $\bar{u}^\eta(t)$. 

\begin{lemma}[Law of large numbers]\label{lem.law}
Let $(\bar{X}^\eta)_{i=1}^N$ be the solution to system \eqref{1.barX} and let $\bar u^\eta$ be the density function associated to each $\bar{X}^\eta_i$, i.e.\ the solution to \eqref{1.baru}. Given $\theta\ge 0$ and $\phi_\eta\in L^\infty(\R^d)$, $\psi_\eta\in L^\infty(\R^d;\R^n)$ with $n\in\{1,d,d\times d\}$, we define the sets
\begin{align*}
  \mathcal{A}_{\theta,\phi_\eta}^N(t) 
  &:= \bigg\{\omega\in\Omega:
  \bigg|\frac{1}{N}\sum_{i=1}^N\phi_\eta(\bar{X}_i^\eta(t,\omega)) 
  - \int_{\R^d}\phi_\eta(x)\bar{u}^\eta(t,x)\dd x\bigg|>N^{-\theta}\bigg\}, \nonumber \\ %\label{1.A}
  \mathcal{B}_{\theta,\psi_\eta}^N(t) 
  &:= \bigcup_{i=1}^N\bigg\{\omega\in\Omega:
  \bigg|\frac{1}{N}\sum_{j=1}^N\psi_\eta
  \big(\bar{X}_i^\eta(t)-\bar{X}_j^\eta(t)\big)
  - (\psi_\eta*\bar{u}^\eta)(\bar{X}_i^\eta(t))\bigg| > N^{-\theta}\bigg\}. %\label{1.B}
\end{align*}
Then, for every $m\in\N$ and $T>0$, there exists $C(m)>0$ such that for all $0<t<T$,
\begin{align*}
  \Prob(\mathcal{A}_{\theta,\phi_\eta}(t)) 
	&\le C(m)\|\phi_\eta\|_{L^\infty(\R^d)}^{2m}N^{2m(\theta-1/2)}, \\
	\Prob(\mathcal{B}_{\theta,\psi_\eta}(t)) 
	&\le C(m)\|\psi_\eta\|_{L^\infty(\R^d)}^{2m}N^{2m(\theta-1/2)+1}.
\end{align*}
\end{lemma}

\begin{proof}
For simplicity, we denote $\psi:= \psi_\eta$ and $\phi:=\phi_\eta$ but keep in mind those dependencies. To estimate the probability of $\mathcal{B}_{\theta,\psi}(t)$, we set
$$
  h_{ij}(t,\omega) := \psi\big(\bar X_i^\eta(t,\omega)-\bar X_j^\eta(t,\omega)\big) - (\psi*\bar u^\eta)(\bar X_i^\eta(t,\omega))
$$
for $t\ge 0$, $\omega\in\Omega$, and $i,j=1,\ldots,N$. Then
$\mathcal{B}_{\theta,\psi}^N(t)=\textstyle\bigcup_{i=1}^N\mathcal{B}_i^N(t)$, where
$$
  \mathcal{B}_i^N(t) := \bigg\{\omega\in\Omega:
  \bigg|\frac{1}{N}\sum_{j=1}^N h_{ij}(t,\omega)\bigg| 
  > N^{-\theta}\bigg\}.
$$
It follows from the Markov inequality for $m\in\N$ that
\begin{align}\label{2.Btheta}
  \Prob(\mathcal{B}_{\theta,\psi}^N(t)) 
  &\le N\max_{i=1,\ldots,N}\Prob(\mathcal{B}_i^N(t)) 
  \le N^{2m\theta+1}\max_{i=1,\ldots,N}
  \E\bigg(\bigg|\frac{1}{N}\sum_{j=1}^N h_{ij}(t)\bigg|^{2m}\bigg) \\
  &= N^{2m(\theta-1)+1}\max_{i=1,\ldots,N}
  \E\bigg(\bigg|\sum_{j=1}^N h_{ij}(t)\bigg|^{2m}\bigg) \nonumber \\
  &= N^{2m(\theta-1)+1}\max_{i=1,\ldots,N}
  \E\bigg(\bigg(\sum_{j,k=1}^N h_{ij}(t) h_{ik}(t)\bigg)^m\bigg). 
  \nonumber 
\end{align}

We distinguish two cases for the summands on the right-hand side. First, we consider summands for which there exists an index  $j\in\{1,\ldots,N\}$ such that $h_{ij}$ appears only once in the product (recall that $i=1,\ldots,N$ is fixed). We claim that the fact that $\bar X_i^\eta$ and $\bar X_j^\eta$ are i.i.d.\ for $i\neq j$ implies that these summands have zero mean, i.e.
$$
  \E\bigg(h_{ij}(t)\prod_{n=1,\, k_n\neq j}^{2m-1}h_{i k_n}(t)\bigg) = 0.
$$
To prove the claim, we assume that $\psi$ is scalar. Our argument can be extended to vector- or matrix-valued functions by taking the sum over their components. Let $K$ denote the set of different indices $k_\ell$ appearing in the product $\prod_{k_n\neq j}h_{i k_n}$, and for each $\ell\in\{1,\ldots,N\}$, let $\alpha_\ell$ denote its multiplicity in this product. Then, by Fubini's theorem,
\begin{align*}
  \E\bigg(h_{ij}(t)&\prod_{n=1,\, k_n\neq j}^{2m-1}h_{i k_n}(t)\bigg)
  = \int_{\R^d}\cdots\int_{\R^d}\bigg(\int_{\R^d}\big(
  \psi(x_i-x_j)-(\psi*\bar u^\eta)(x_i)\big)\bar u^\eta(x_j)
  \dd x_j\bigg) \\
  &\times\prod_{\ell\in K}\big(\psi(x_i-x_\ell)-(\psi*\bar u^\eta)(x_i)
  \big)^{\alpha_\ell}\bar u^\eta(x_\ell)\bar u^\eta(x_i)\dd x_i
  \bigotimes_{\ell\in K}\dd x_\ell.
\end{align*}
Since $\|\bar u^\eta\|_{L^1(\R^d)}=1$, the inner integral with index $j$ vanishes,
$$
  \int_{\R^d}\big(\psi(x_i-x_j)-(\psi*\bar u^\eta)(x_i)\big)
  \bar u^\eta(x_j)\dd x_j
  = (\psi*\bar u^\eta)(x_i) - (\psi*\bar u^\eta)(x_i)
  \int_{\R^d}\bar u^\eta(x_j)\dd x_j = 0,
$$
which proves the claim.

Next, we consider summands in which each term $h_{ij}$ appears at least twice in the product. We do not know whether those terms have zero expectation, but we can estimate the number of (potentially) nonvanishing terms. Indeed, we collect these summands in the set
$$
  \mathcal{N}_i := \bigg\{\prod_{n=1}^{2m}h_{ij_n}:\mbox{all indices }
  j_n\mbox{ appear at least twice}\bigg\}.
$$
We claim that the cardinality $|\mathcal{N}_i|$ of this set is, up to some factor, bounded by $N^m$. Indeed, let $A$ be the set of indices in the products $\prod_{\alpha\in A} h_{\alpha}^{r_\alpha}\in\mathcal{N}_i$ with $r_\alpha\in\{2,\ldots,2m\}$. Since all appearing indices appear at least twice, we have $|A|\le m$.

To estimate the cardinality of $\mathcal{N}_i$, consider first the case
$|A|=m$. This means that we have $m$ different indices, i.e., each
index appears exactly twice. We can choose $\tbinom{N}{m}$ such sets of indices, which is bounded by $C(m)N^m$ for some constant $C(m)>0$. We also need to take into account all permutations of such a selection of indices. Then the size of the set of products $\prod_{\alpha\in A}h_{\alpha}^{r_\alpha}$ can be estimated from above by $(2m)!C(m)N^m$, which is again of order $C(m)N^m$. If $|A|=m-n$ for some $n\in\{0,\ldots,m-1\}$, we can argue in a similar way, and the cardinality of the set of products is estimated by $C(m)N^{m-n}$. Adding all these cardinalities, we find that $|\mathcal{N}_i|\le C(m)(N^m+N^{m-1}+\cdots+N)\le C(m)N^m$. 

The expectation of $|\sum_{j=1}^N h_{ij}|^{2m}$ can be written as the sum of expectations of products of the form $\prod_{k=1}^{2m}h_{ij_k}$, which are individually bounded from above by $C\|\psi\|_{L^\infty(\R^d)}^{2m}$. This leads to
$$
  \E\bigg(\bigg|\sum_{j=1}^N h_{ij}\bigg|^{2m}\bigg)
  \le C|\mathcal{N}_i|\|\psi\|_{L^\infty(\R^d)}^{2m}
  \le C(m)N^m\|\psi\|_{L^\infty(\R^d)}^{2m}.
$$
We infer from \eqref{2.Btheta} that 
$$
  \Prob(\mathcal{B}_{\theta,\psi}^N(t))
  \le C(m)N^{2m(\theta-1)+1} N^m \|\psi\|_{L^\infty(\R^d)}^{2m}
  = C(m)N^{2m(\theta-1/2)+1}\|\psi\|_{L^\infty(\R^d)}^{2m}.
$$

It remains to show the estimate for $\mathcal{A}_{\theta,\phi}^N(t)$. 
By Markov's inequality,
\begin{align*}
  \Prob(\mathcal{A}_{\theta,\phi}^N(t)) 
  &\le N^{2m\theta}\E\bigg(\bigg|\frac{1}{N}
  \sum_{i=1}^N\phi(\bar X_i^\eta(t)) - \int_{\R^d}\phi(x)
  \bar u^\eta(t,x)\dd x\bigg|^{2m}\bigg) \\
  &\le N^{2m(\theta-1)}\E\bigg(\bigg|\sum_{i=1}^N h_i\bigg|^{2m}\bigg)
  = N^{2m(\theta-1)}\E\bigg(\bigg|\sum_{i,j=1}^N h_ih_j\bigg|^m\bigg),
\end{align*}
where $h_i(t) := \phi(\bar X_i^\eta(t)) - \int_{\R^d}\phi(x)\bar u^\eta(t,x)\dd x$. Similarly as before, the expectation of all terms in the sum such that one index $i\in\{1,\ldots,N\}$ appears only once vanish,
$$
  \E\bigg(h_i\sum_{n=1,\, k_n\neq i}^{2m-1}h_{k_n}\bigg)
  = \E(h_i)\E\bigg(\sum_{n=1,\, k_n\neq i}^{2m-1}h_{k_n}\bigg) = 0,
$$
since $(h_i)_{i=1}^N$ is i.i.d.\ and $\E(h_i)=0$. To estimate the remaining terms, we introduce
$$
  \mathcal{N} := \bigg\{\prod_{n=1}^{2m} h_{i_n}:\mbox{all indices }i_n
  \mbox{ appear at least twice}\bigg\}.
$$
Its size can be estimated as before, leading to
$|\mathcal{N}|\le C(m)(N^m+N^{m-1}+\cdots+N)\le C(m)N^m$. Finally, we deduce from $\E(\prod_{n=1}^{2m}h_{i_n})\le C\|\phi\|_{L^\infty(\R^d)}^{2m}$ that
$$
  \Prob(\mathcal{A}_{\theta,\phi}^N(t)) 
  \le C(m)N^{2m(\theta-1/2)}\|\phi\|_{L^\infty(\R^d)}^{2m}.
$$
This finishes the proof.
\end{proof}

%%%%%%%%%%%%%%%%%%%%%%%%%%%%%%%%%%%%%%%%%%%%%%%%%%%%%%%%%%%%%%%%%%%%

\section{Convergence in probability}\label{sec.prob}

We split the proof of Theorem \ref{thm.prob} in several parts. 

\subsection{Preparations}

We start with some definitions. Let $\alpha>0$ be given as in the theorem and let $k\in\N$. We define the stopping time
$$
  \tau_\alpha(\omega) := \inf\Big\{t\in(0,T):\max_{i=1,\ldots,N}
  |(X_i^\eta-\bar{X}_i^\eta)(t,\omega)| \ge N^{-\alpha}\Big\}\wedge T
$$
and the random variable
$$
  S_\alpha^k(t) := \Big(N^{\alpha}\max_{i=1,\ldots,N}
  |(X_i^\eta-\bar{X}_i^\eta)(t\wedge\tau_\alpha)|\Big)^k \le 1.
$$
Then $B_\alpha(t):=\{\omega\in\Omega:S_\alpha^k(t)=1\}$ describes the set of $\omega\in \Omega$ such that $t\leq \tau_\alpha(\omega)$. This set does not depend on $k$, since $S_\alpha^k(t)=1$ is equivalent to $S^k_\alpha(t)^{1/k}=1$, and $S^k_\alpha(t)^{1/k}$ does not depend on $k$. It follows from the continuity of the paths of $X_i^\eta$ and $\bar{X}_i^\eta$ and the fact that if $\max_{i=1,\ldots,N}|X_i^\eta(t, \omega)-\bar{X}_i^\eta(t,\omega)|
> N^{-\alpha}$ for a fixed $t>0$ then $t>\tau_\alpha(\omega)$, that
\begin{align*}
  \Prob\Big(\max_{i=1,\ldots,N}|(X_i^\eta-\bar{X}_i^\eta)(t)| 
  > N^{-\alpha}\Big)
  &\le \Prob\Big(\max_{i=1,\ldots,N}
  |(X_i^\eta-\bar{X}_i^\eta)(t\wedge\tau_\alpha)|=N^{-\alpha}\Big) \\ 
  &= \Prob(B_\alpha(t)) = \Prob(S_\alpha^k(t)=1)\le\E(S_\alpha^k(t)),
\end{align*}
where the last estimate follows from Markov's inequality.

Now, if we show that for every $\gamma>0$ and $T>0$, there exists $k\in\N$ and $C=C(\gamma,k,T)>0$ such that $\E(S_\alpha^k(t))\le CN^{-\gamma}$, the proof is finished. To prove this claim, we insert the integral formulations \eqref{1.X} and \eqref{1.barX} and add and subtract $\na V^\eta(\bar{X}_i^\eta-\bar{X}_j^\eta)$. Indeed, it holds for $i=1,\ldots,N$ that
\begin{align}\label{3.diffX}
  |(&X_i^\eta-\bar{X}_i^\eta)(t\wedge\tau_\alpha)|^k \\
  &\le C(k,T)\int_0^{t\wedge\tau_\alpha}\bigg|\frac{1}{N}\sum_{j=1}^N
  \na V^\eta(X_i^\eta(s)-X_j^\eta(s))
  - (\na V^\eta*\bar{u}^\eta)(s,\bar{X}_i^\eta(s))\bigg|^k\dd s 
  \nonumber \\
  &\le C(k,T)\int_0^{t\wedge\tau_\alpha}
  \big(|I_{1,i}(s)|^k+|I_{2,i}(s)|^k\big)\dd s, \nonumber 
\end{align}
where
\begin{align*}
  I_{1,i}(s) &= \bigg|\frac{1}{N}\sum_{j=1}^N\big[
  \na V^\eta(X_i^\eta(s)-X_j^\eta(s))
  - \na V^\eta(\bar{X}_i^\eta(s)-\bar{X}_j^\eta(s))\big]\bigg|, \\
  I_{2,i}(s) &= \bigg|\frac{1}{N}\sum_{j=1}^N
  \na V^\eta(\bar{X}_i^\eta(s)-\bar{X}_j^\eta(s)) - (\na V^\eta*\bar{u}^\eta)(s,\bar{X}_i^\eta(s))\bigg|.
\end{align*}

\subsection{Estimate of $I_{2,i}(t)$}

The term $I_{2,i}(t)$ is estimated by a law-of-large numbers argument. We apply Lemma \ref{lem.law} with $\psi_\eta=\na V^\eta$ and $\theta>0$, which will be chosen at the end of the proof. With the notation of Lemma \ref{lem.law} we have $\mathcal{B}_{\theta,\na V^\eta}^N(s) = \bigcup_{i=1}^N\{I_{2,i}(s)>N^{-\theta}\}$. As we want to bound $\E(S_\alpha^k(t))$, we compute the expectation of $N^{\alpha k}\max_{i=1,\ldots,N}\int_0^{t\wedge\tau_\alpha}|I_{2,i}(s)|^k\dd s$ by splitting $\Omega$ into the two sets $\mathcal{B}_{\theta,\na V^\eta}^N(s)$ and its complement $\mathcal{B}_{\theta,\na V^\eta}^N(s)^c$. Taking into account that $I_{2,i}(s)\le N^{-\theta}$ on $\mathcal{B}_{\theta,\na V^\eta}^N(s)^c$ for $i=1,\ldots,N$ and applying Lemma \ref{lem.law}, for any $m\in\N$,
\begin{align*}
  \E\bigg(&N^{\alpha k}\max_{i=1,\ldots,N}
  \int_0^{t\wedge\tau_\alpha}|I_{2,i}(s)|^k\dd s\bigg) \\
  &\le \E\bigg(N^{\alpha k}\max_{i=1,\ldots,N}
  \int_0^{t\wedge\tau_\alpha}|I_{2,i}(s)|^k
  \mathrm{1}_{\mathcal{B}_{\theta,\na V^\eta}^N(s)^c}\dd s\bigg)\\
  &\phantom{xx}+ \E\bigg(N^{\alpha k}\max_{i=1,\ldots,N}\int_0^{t\wedge\tau_\alpha}|I_{2,i}(s)|^k
  \mathrm{1}_{\mathcal{B}_{\theta,\na V^\eta}^N(s)}\dd s\bigg) \\
  &\le TN^{\alpha k}N^{-\theta k} + C(T)N^{\alpha k}
  \|\na V^\eta\|_{L^\infty(\R^d)}^k \sup_{0<s<T}
  \Prob({\mathcal{B}_{\theta,\na V^\eta}^N(s)}) \\
  &\le TN^{(\alpha-\theta)k} + C(m,T)N^{\alpha k}\|\na V^\eta\|_{L^\infty(\R^d)}^{k+2m}N^{2m(\theta-1/2)+1},
\end{align*}
where in the last but one step we have applied Young's convolution inequality to estimate 
$$
  \|\na V^\eta*\bar{u}^\eta\|_{L^\infty(\R^d)}
  \le \|\na V^\eta\|_{L^\infty(\R^d)}\|\bar{u}^\eta\|_{L^1(\R^d)}
  \le C\|\na V^\eta\|_{L^\infty(\R^d)}.
$$
It follows from $\|\na V^\eta\|_{L^\infty(\R^d)}\le CN^{\beta(\lambda+1)}$ (Lemma \ref{lem.DkV}) that
\begin{align}\label{3.I2i}
  \E\bigg(&N^{\alpha k}\max_{i=1,\ldots,N}
  \int_0^{t\wedge\tau_\alpha}|I_{2,i}(t)|^k\dd s\bigg) \\
  &\le C(k,m,T)\big(N^{(\alpha-\theta)k}
  + N^{\alpha k+\beta(\lambda +1)(k+2m)+m(2\theta-1)+1}
  \big). \nonumber
\end{align}
We will determine the parameters $k$ and $m$ at the end of the proof.

\subsection{Estimate of the term with integrand $I_{1,i}(t)$}

This estimate is more technical than the previous one. We perform a Taylor expansion of $\na V^\eta$ around $(\bar{X}_i^\eta-\bar{X}_j^\eta)(s)$ with a linear term and a quadratic remainder:
\begin{align}\label{3.I1i}
  \mathbb{E}\Big(&N^{\alpha k} \int_{0}^{t \wedge \tau_\alpha} \max_{i=1,\ldots,N}I_{1,i}(s)^k \dd s\Big) \\
  &\le C(k)\E\bigg(N^{\alpha k}\int_0^{t\wedge\tau_\alpha}
  \max_{i=1,\ldots,N}\bigg|\frac{1}{N}\sum_{j=1}^N
  \mathrm{D}^2 V^\eta(\bar{X}_i^\eta(s)-\bar{X}_j^\eta(s)) \nonumber \\
  &\phantom{xx}\times
  \big((X_i^\eta-X_j^\eta)(s)
  - (\bar{X}_i^\eta-\bar{X}_j^\eta)(s)\big)\bigg|^k\dd s
  \bigg) \nonumber \\
  &\phantom{xx}{}+ C(k)\|\mathrm{D}^3 V^\eta\|_{L^\infty(\R^d)}^k
  \E\bigg(N^{\alpha k}
  \int_0^{t\wedge\tau_\alpha}\max_{i=1,\ldots,N}\frac{1}{N}\sum_{j=1}^N
  \big|\big((X_i^\eta-\bar{X}_i^\eta)(s) \nonumber \\
  &\phantom{xx}{}- (X_j^\eta - \bar{X}_j^\eta)(s)\big)
  \big|^{2k}\dd s\bigg). \nonumber
\end{align}
We collect in the first term on the right-hand side the differences with the same index, i.e. $(X_i^\eta-X_j^\eta) - (\bar{X}_i^\eta-\bar{X}_j^\eta)
= (X_i^\eta-\bar{X}_i^\eta) - (X_j^\eta-\bar{X}_j^\eta)$ such that
\begin{align*}
  &\mathbb{E}\Big(N^{\alpha k} \int_{0}^{t \wedge \tau_\alpha} \max_{i=1,\ldots,N}I_{1,i}(s)^k \dd s\Big)
  \le C(k)(I_{11}+I_{12}+I_{13})(t), \quad\mbox{where} \\
  & I_{11}(t) = \E\bigg(N^{\alpha k}\int_0^{t\wedge\tau_\alpha}
  \max_{i=1,\ldots,N}\bigg|\frac{1}{N}\sum_{j=1}^N\mathrm{D}^2
  V^\eta(\bar{X}_i^\eta-\bar{X}_j^\eta)
  (X_i^\eta-\bar{X}_i^\eta)(s)\bigg|^k\dd s\bigg), \\
  & I_{12}(t) = \E\bigg(N^{\alpha k}\int_0^{t\wedge\tau_\alpha}
  \max_{i=1,\ldots,N}\bigg|\frac{1}{N}\sum_{j=1}^N\mathrm{D}^2 
  V^\eta(\bar{X}_i^\eta-\bar{X}_j^\eta)
  (X_j^\eta-\bar{X}_j^\eta)(s)\bigg|^k\dd s\bigg), \\
  & I_{13}(t) = \|\mathrm{D}^3 V^\eta\|_{L^\infty(\R^d)}^k
  \E\bigg(N^{\alpha k}\int_0^{t}\max_{i=1,\ldots,N}
  \big|(X_i^\eta-\bar{X}_i^\eta)(s\wedge\tau_\alpha)\big|^{2k}
  \dd s\bigg).
\end{align*}

We start with the term $I_{13}$. It follows from Fubini's theorem,
the estimate $\|\mathrm{D}^3 V^\eta\|_{L^\infty(\R^d)}$ $\le CN^{\beta(\lambda+3)}$ from Lemma \ref{lem.DkV}, the definition of $S_\alpha^k(s)$, and $S_\alpha^k(s)^2\le S_\alpha^k(s)$ (since $S_\alpha^k(s) \leq 1$) that
\begin{equation}\label{3.I13}
  I_{13}(t) \le C(k)N^{\beta(\lambda+3)k}
  \int_0^t\E\big(N^{-\alpha k}S_\alpha^k(s)^2\big)\dd s
  \le C(k)N^{\beta(\lambda+3)k-\alpha k}\int_0^t\E(S_\alpha^k(s))\dd s.
\end{equation}
Note that we need the definition of the stopping time $\tau_\alpha$, which
guarantees that $S_\alpha^k(t)\le 1$. 

Next, we estimate
$I_{11}(t)\le I_{111}(t)+I_{112}(t)$ by adding and subtracting
$(\mathrm{D}^2 V^\eta*\bar{u}^\eta)(\bar{X}_i^\eta)$:
\begin{align*}
  I_{111}(t) &\le \E\bigg(\int_0^{t\wedge\tau_\alpha}S_\alpha^k(s)
  \max_{i=1,\ldots,N}\bigg|\frac{1}{N}\sum_{j=1}^N\mathrm{D}^2 V^\eta(\bar{X}_i^\eta(s)-\bar{X}_j^\eta(s)) \\
  &\phantom{xx}{}- (\mathrm{D}^2 V^\eta 
  * \bar{u}^\eta)(s,\bar{X}_i^\eta(s))\bigg|^k\dd s\bigg), \\
  I_{112}(t) &\le \E\bigg(\int_0^{t\wedge\tau_\alpha}S_\alpha^k(s)
  \max_{i=1,\ldots,N}\big|(\mathrm{D}^2 V^\eta 
  * \bar{u}^\eta)(s,\bar{X}_i^\eta(s))\big|^k\dd s\bigg).
\end{align*}
For $I_{111}(t)$, we apply Lemma \ref{lem.law} for $m\in\N$ (which will be chosen at the end of the proof) with $\psi_\eta=\mathrm{D}^2 V^\eta$ and $\theta=0$ and split $\Omega$ into $\mathcal{B}_{0,\mathrm{D}^2 V^\eta}(s)$ and $\mathcal{B}_{0,\mathrm{D}^2 V^\eta}(s)^c$. Then Fubini's theorem and the bound $\|\mathrm{D}^2 V^\eta\|_{L^\infty(\R^d)}\le CN^{\beta(\lambda+2)}$ from Lemma \ref{lem.DkV} lead to
\begin{align*}
  I_{111}(t) &\le \E\bigg(\int_0^{t\wedge\tau_\alpha}S_\alpha^k(s)
  \max_{i=1,\ldots,N}\bigg|\frac{1}{N}\sum_{j=1}^N\mathrm{D}^2 V^\eta(\bar{X}_i^\eta(s)-\bar{X}_j^\eta(s)) \\
  &\phantom{xxxx}- (\mathrm{D}^2 V^\eta * \bar{u}^\eta)
  (s,\bar{X}_i^\eta(s))\bigg|^k
  \mathrm{1}_{B_{0,\mathrm{D}^2 V^\eta}(s)^c}\dd s\bigg) \\
  &\phantom{xx}+ C(T)\|\mathrm{D}^2 V^\eta\|_{L^\infty(\R^d)}^k
  \sup_{0<s<T}\Prob(B_{0,\mathrm{D}^2 V^\eta}(s)) \\
  &\le N^{-0}\int_0^t\E(S_\alpha^k(s))\dd s 
  + C(m,T)N^{\beta (\lambda +2)k} N^{2m\beta(\lambda+2)}
  N^{2m(0-1/2)+1}.
\end{align*}
The estimate for $I_{112}(t)$ simply follows from Fubini's theorem 
and the bound for $\mathrm{D}^2 V^\eta*\bar{u}^\eta$ from Lemma \ref{lem.DkVbaru}, yielding
$$
  I_{112}(t) \le \|\mathrm{D}^2V^\eta * \bar{u}^\eta
  \|_{L^\infty(\R^d)}^k\int_0^t\E(S_\alpha^k(s))\dd s
  \le C(k)\int_0^t\E(S_\alpha^k(s))\dd s.
$$
We conclude that
\begin{equation}\label{3.I11}
  I_{11}(t) \le C(k)\int_0^t\E(S_\alpha^k(s))\dd s 
  + C(k,m,T)N^{\beta(\lambda+2)(k+2m)-m+1},
\end{equation}
where the constant $C(k)>0$ depends on the $L^\infty(0,T;L^{\infty}(\R^d)\cap L^{1}(\R^d))$ norm
of $\bar{u}^\eta$, which is bounded uniformly in $\eta$ thanks to Lemma \ref{lem.wLinfty}.

Finally, we estimate $I_{12}(t)$ by first taking the modulus of $X_j^\eta(s)-\bar{X}_j^\eta(s)$ inside the sum:
\begin{align*}
  I_{12}(t) &\le \E\bigg(N^{\alpha k}\int_0^{t\wedge\tau_\alpha}
  \max_{i=1,\ldots,N}\bigg(\frac{1}{N}\sum_{j=1}^N\big|\mathrm{D}^2 
  V^\eta(\bar{X}_i^\eta(s)-\bar{X}_j^\eta(s))\big|
  \big|X_j^\eta(s)-\bar{X}_j^\eta(s)\big|\bigg)^k\dd s\bigg) \\
  &\le \E\bigg(\int_0^{t\wedge\tau_\alpha} S_\alpha^k(s)
  \max_{i=1,\ldots,N}\bigg(\frac{1}{N}\sum_{j=1}^N\big|\mathrm{D}^2
  V^\eta(\bar{X}_i^\eta(s)-\bar{X}_j^\eta(s))\big|\bigg)^k\dd s\bigg),
\end{align*}
where we used the definition of $S_\alpha^k(s)$ for $s<\tau_\alpha$. Similarly as in the estimate for $I_{11}(t)$, we add and subtract
$|\mathrm{D}^2 V^\eta|*\bar{u}^\eta(s,\bar{X}_i^\eta(s))$, which yields $I_{12}(t)\le C(I_{121}(t)+I_{122}(t))$, where
\begin{align*}
  I_{121}(t) &= \E\bigg(\int_0^{t\wedge\tau_\alpha} 
  S_\alpha^k(s)\max_{i=1,\ldots,N}
  \bigg|\frac{1}{N}\sum_{j=1}^N\big|\mathrm{D}^2 V^\eta
  (\bar{X}_i^\eta(s)-\bar{X}_j(s))\big| \\
  &\phantom{xx}{}- \big(|\mathrm{D}^2 V^\eta|*\bar{u}^\eta\big)
  (s,\bar{X}_i^\eta(s))\bigg|^k\dd s\bigg), \\
  I_{122}(t) &= \E\bigg(\int_0^{t\wedge\tau_\alpha}
  S_\alpha^k(s)\max_{i=1,\ldots,N}
  \big((|\mathrm{D}^2 V^\eta|*\bar{u}^\eta)(s,\bar{X}_i^\eta(s))
  \big)^k\dd s\bigg).
\end{align*}
Lemma \ref{lem.law} with $\psi_\eta=|\mathrm{D}^2 V^\eta|$ and $\theta=0$ and the property $S_\alpha^k(t)\le 1$ imply that, for any $m\in\N$,
\begin{align*}%\label{3.I21}
  I_{121}(t)&\le \int_0^t\E(S_\alpha^k(s))\dd s 
  + C(T)\|\mathrm{D}^2 V^\eta\|_{L^\infty(\R^d)}^k
  \sup_{0<s<T}\Prob(B^N_{0,|\mathrm{D}^2 V^\eta|}(s)) \\
  &\le \int_0^t\E(S_\alpha^k(s))\dd s 
  + C(T)\|\mathrm{D}^2 V^\eta\|_{L^\infty(\R^d)}^{k+2m}N^{2m(0-1/2)+1}
  \nonumber \\
  &\le \int_0^t\E(S_\alpha^k(s))\dd s 
  + C(m,T)N^{\beta(\lambda+2)(k+2m)-m+1}. \nonumber
\end{align*}
By Lemma \ref{lem.DkVbaru}, $|\mathrm{D}^2 V^\eta|*\bar{u}^\eta$ is bounded (here we need $\lambda<d-2$). Hence,
\begin{align*}%\label{3.I22}
  I_{122}(t) &\le C\||\mathrm{D}^2 V^\eta|*\bar{u}^\eta
  \|_{L^\infty(0,T;L^\infty(\R^d))}^k
  \int_0^t\E(S_\alpha^k(s))\dd s
  \le C(k)\int_0^t\E(S_\alpha^k(s))\dd s.
\end{align*}
These estimates show that
\begin{equation}\label{3.I12}
  I_{12}(t) \le C(k)\int_0^t\E(S_\alpha^k(s))\dd s
	+ C(k,m,T)N^{\beta(\lambda+2)(k+2m)-m+1}.
\end{equation}

Summarizing, we insert estimates \eqref{3.I13}, \eqref{3.I11}, and \eqref{3.I12} into \eqref{3.I1i} to find that
\begin{align*}
  \mathbb{E}\Big(N^{\alpha k} \int_{0}^{t \wedge \tau_\alpha} 
  \max_{i=1,\ldots,N}I_{1,i}(s)^k \dd s\Big) 
  &\le C(k)\big(1+N^{\beta(\lambda+3)k-\alpha k}\big)
  \int_0^t\E(S_\alpha^k(s))\dd s \\
  &\phantom{xx}+ C(k,m,T)N^{\beta(\lambda+2)(k+2m)-m+1}.
\end{align*}
Combining this estimate with \eqref{3.I2i}, we conclude from \eqref{3.diffX} that
\begin{align*}
  \E(S_\alpha^k(t)) &= \E\Big(N^{\alpha k}\max_{i=1,\ldots,N}
  \big|X_i^\eta(t\wedge\tau_\alpha)-\bar{X}_i^\eta(t\wedge\tau_\alpha)
  \big|^k\Big) \\
  &\le C(k)\big(1+N^{\beta(\lambda +3)k-\alpha k}\big)
  \int_0^t\E(S_\alpha^k(s))\dd s \\
  &\phantom{xx}+ C(k,m,T)\big(N^{(\alpha-\theta)k}  
	+ N^{\alpha k + \beta(\lambda +1)(k+2m) + m(2\theta-1)+1}
	+ N^{\beta (\lambda +2)(k+2m)-m+1}\big).
\end{align*}

\subsection{End of the proof}

Since $\alpha\ge\beta(\lambda+3)$ by assumption, the factor $N^{\beta(\lambda+3)k-\alpha k}$ is bounded for all $N$. We claim that for any given $\gamma>0$ and $(\alpha,\beta)$ chosen according to the theorem, we can choose $k$, $\theta$, and $m$ such that
the remaining terms are bounded by $N^{-\gamma}$. Indeed, let  
$\theta\in(\alpha,1/2-\beta(\lambda +1))$ (this requires that $\alpha<1/2-\beta(\lambda +1)$). The parameters $k$ and $m$ are chosen as follows:
\begin{itemize}
\item We take $k\in\N$ sufficiently large such that $(\alpha-\theta)k\le -\gamma$. This is possible since $\alpha-\theta<0$.
\item We choose $m\in\N$ sufficiently large such that $\beta (\lambda+2)(k+2m)-m+1 \le -\gamma$, which is equivalent to
$$
  2m(\beta(\lambda +2)-1/2) \leq -\gamma  -1 -\beta(\lambda +2)k.
$$
This choice is possible since $\beta<1/(2\lambda+4)$.
\item Next, we take $m\in \N$ sufficiently large (or even larger) such that $\alpha k+\beta(\lambda +1)(k+2m)+m(2\theta-1)+1\le -\gamma$,
which is equivalent to
$$
  m(2\beta(\lambda +1) +2\theta -1) \leq -\gamma -\alpha k 
  - \beta(\lambda+1)k -1.
$$ 
Choosing $m\in \N$ large enough is possible since $\theta<1/2-\beta(\lambda +1)$.
\end{itemize}
We infer that
$$
  \E(S_\alpha^k(t)) \le C(k)\int_0^t\E(S_\alpha^k(s))\dd s 
  + C(k,m,T)N^{-\gamma}.
$$
Gronwall's lemma implies that $\E(S_\alpha^k(t))\le
C(k,m,T)N^{-\gamma}$. Since $\gamma>0$ was arbitrary, this concludes the proof.

%%%%%%%%%%%%%%%%%%%%%%%%%%%%%%%%%%%%%%%%%%%%%%%%%%%%%%%%%%%%%%%%%%%%

\section{Mean-square mean-field estimate}\label{sec.quant}

In this section, we prove Theorem \ref{thm.L2}. We first reformulate $\|(f^\eta-g^\eta)(t)\|_{L^2(\R^d)}$, estimate the resulting terms, and then derive a bound for $\|(f^\eta-g^\eta)(0)\|_{L^2(\R^d)}$.

\subsection{First reformulation} 

We first reformulate the $L^2(\R^d)$ norm of $(f^\eta-g^\eta)(t)$
in terms of $V^\eta$, $\mu^\eta$, and $\bar{u}^\eta$. For this, we expand
\begin{align} \label{4.J123}
  &\|(f^\eta-g^\eta)(t)\|_{L^2(\R^d)}^2 
  - \|(f^\eta-g^\eta)(0)\|_{L^2(\R^d)}^2
  = J_1 + J_2 + J_3, \quad\mbox{where} \\
  & J_1 = \|f^\eta(t)\|_{L^2(\R^d)}^2 - \|f^\eta(0)\|_{L^2(\R^d)}^2, 
  \nonumber \\
  & J_2 = \|g^\eta(t)\|_{L^2(\R^d)}^2 - \|g^\eta(0)\|_{L^2(\R^d)}^2, 
  \nonumber \\
  & J_3 = -2\big(\langle f^\eta(t),g^\eta(t)\rangle 
  - \langle f^\eta(0),g^\eta(0)\rangle\big). \nonumber
\end{align}
	
{\em Step 1: Reformulation of $J_1$.}
By definition \eqref{1.fg} of $f^\eta$, $V^\eta=Z^\eta*Z^\eta$, the symmetry of $Z^\eta$ (which is a consequence of the symmetry of $\xi$ and $\Psi$), and the change of variable $y=x-X_i^\eta(t)$, we compute
\begin{align}\label{4.fL2}
  \|f^\eta(t)\|_{L^2(\R^d)}^2 &= \bigg\|\frac{1}{N}\sum_{i=1}^N 
  Z^\eta(\boldsymbol{\cdot} - X_i^\eta(t))\bigg\|_{L^2(\R^d)}^2
  = \frac{1}{N^2}\int_{\R^d}\bigg(\sum_{i=1}^N Z^\eta(x-X_i^\eta(t))
  \bigg)^2 \dd x \nonumber \\
  &= \frac{1}{N^2}\sum_{i,j=1}^N\int_{\R^d} Z^\eta\big(y+X_j^\eta(t)
  - X_i^\eta(t)\big) Z^\eta(y)\dd y \\
  &= \frac{1}{N^2}\sum_{i,j=1}^N (Z^\eta*Z^\eta)(X_i^\eta(t)-X_j^\eta(t))
  = \frac{1}{N^2}\sum_{i,j=1}^N V^\eta(X_i^\eta(t)-X_j^\eta(t)). 
  \nonumber
\end{align}
To reformulate the last expression, we apply It\^{o}'s formula.
For this, we rewrite the particle system \eqref{1.X}. In the following,
we omit the argument $t$ whenever this simplifies the notation. Recall that, by definition of the empirical measure, system \eqref{1.X} can be written as 
$$
  \dd X_i^\eta = \kappa(\nabla V^\eta *\mu^\eta )(X_i^\eta)\dd t
  + \sqrt{2\sigma}\dd W_i, \quad i=1,\ldots,N,
$$
and consequently, for the vector $\mathbf{X}_{ij}=(X_i^\eta,X_j^\eta)^T \in \R^{2d}$ for some $i,j\in\{1,\ldots,N\}$,
$$
  \dd\mathbf{X}_{ij} = \kappa\begin{pmatrix}
  (\nabla V^\eta *\mu^\eta)(X_i^\eta) \\
  (\nabla V^\eta *\mu^\eta)(X_j^\eta)
  \end{pmatrix}\dd t
  + \sqrt{2\sigma}\begin{pmatrix} \dd W_i \\ \dd W_j \end{pmatrix}.
$$
We introduce $g(\mathbf{X}_{ij})=V^\eta(X_i^\eta-X_j^\eta)$. The derivatives are
\begin{align*}
  \mathrm{D}g(\mathbf{X}_{ij}) &= \begin{pmatrix}
  \na V^\eta(X_i^\eta-X_j^\eta) \\ 
  -\na V^\eta(X_i^\eta-X_j^\eta) \end{pmatrix} \in \R^{2d}, \\
  \mathrm{D}^2 g(\mathbf{X}_{ij}) &= \begin{pmatrix}
  \mathrm{D}^2 V^\eta(X_i^\eta-X_j^\eta) 
  & -\mathrm{D}^2 V^\eta(X_i^\eta-X_j^\eta) \\
  -\mathrm{D}^2 V^\eta(X_i^\eta-X_j^\eta) 
  & \mathrm{D}^2 V^\eta(X_i^\eta-X_j^\eta)
  \end{pmatrix} \in \R^{2d \times 2d}.
\end{align*}
Abbreviating $Y_{ij}=X_i^\eta-X_j^\eta$ , It\^{o}'s formula gives
\begin{align}\label{4.ito}
  \dd g(\mathbf{X}_{ij}) 
  &= \bigg(\kappa\mathrm{D}g(\mathbf{X}_{ij})\cdot\begin{pmatrix}
  (\na V^\eta*\mu^\eta)(X_i^\eta) \\ 
  (\na V^\eta*\mu^\eta)(X_j^\eta)
  \end{pmatrix} + \frac12\operatorname{tr}\big(2\sigma\mathrm{D}^2
  g(\mathbf{X}_{ij})\big)\bigg)\dd t \\
  &\phantom{xx}+ \sqrt{2\sigma}\mathrm{D}g(\mathbf{X}_{ij})
  \cdot\begin{pmatrix} \dd W_i \\ \dd W_j \end{pmatrix} \nonumber \\
  &= \kappa\na V^\eta(Y_{ij})\cdot\big((\na V^\eta*\mu^\eta)(X_i^\eta)
  - (\na V^\eta*\mu^\eta)(X_j^\eta)\big)\dd t \nonumber \\
  &\phantom{xx}+ 2\sigma\Delta V^\eta(Y_{ij})\dd t
  + \sqrt{2\sigma}\na V^\eta(Y_{ij})
  \cdot(\dd W_i-\dd W_j). \nonumber
\end{align} 
After summation over $i,j=1,\ldots,N$ with $i\neq j$ and using the symmetry of $V^\eta$, i.e.\ $\na V^\eta(Y_{ij})=-\na V^\eta(Y_{ji})$ in the first term on the right-hand side, the integral formulation of \eqref{4.ito} becomes
\begin{align*}
  \sum_{i\neq j}&\big(g(\mathbf{X}_{ij}(t))-g(\mathbf{X}_{ij}(0))\big)
  = 2\kappa\sum_{i\neq j}\int_0^t\na V^\eta(Y_{ij}(s))\cdot
  (\na V^\eta*\mu^\eta)(X_i^\eta(s))\dd s \\
  &\phantom{xx}+ 2\sigma\sum_{i\neq j}\int_0^t
  \Delta V^\eta(Y_{ij}(s))\dd s
  + 2\sqrt{2\sigma}\sum_{i\neq j}
  \int_0^t\na V^\eta(Y_{ij}(s))\cdot\dd W_i(s),
\end{align*}
where we have used for the It\^o integral the definition of $Y_{ij}$ and
\begin{align*}
  \sum_{i,j=1,\,i\neq j}^N &\int_{0}^t\na V^\eta
  (X_i^\eta - X_j^\eta)\cdot\dd W_i - \sum_{i,j=1,\,i\neq j}^N \int_{0}^t\nabla V^\eta( X_i^\eta - X_j^\eta)\cdot\dd W_j \\
  &= 2\sum_{i,j=1,\,i\neq j}^N \int_{0}^t\na V^\eta
  (X_i^\eta - X_j^\eta)\cdot\dd W_i,
\end{align*}
which follows from the anti-symmetry of $\na V^\eta$ and the fact that $W_i$ are independent Brownian motions for $i=1,\ldots,N$.

We deduce from the definitions of $g$ and $J_1$, the fact that the difference $V^\eta(Y_{ij}(t))-V^\eta(Y_{ij}(0))$ vanishes for $i=j$ (since $Y_{ii}(t)=0$ for all $t\geq 0$), and from formulation \eqref{4.fL2} that
\begin{align*}
  J_1 &= \frac{1}{N^2}\sum_{i\neq j}
  \big(V^\eta(Y_{ij}(t))-V^\eta(Y_{ij}(0))\big)
  = \frac{1}{N^2}\sum_{i\neq j}
  \big(g(\mathbf{X}_{ij}(t))-g(\mathbf{X}_{ij}(0))\big) \\
  &= \frac{2\kappa}{N^2}\sum_{i\neq j}\int_0^t\na V^\eta(Y_{ij}(s))
  \cdot(\na V^\eta*\mu^\eta)(X_i^{\eta}(s))\dd s \\
  &\phantom{xx}+ \frac{2\sigma}{N^2}\sum_{i\neq j}\int_0^t
  \Delta V^\eta(Y_{ij}(s))\dd s
  + \frac{2\sqrt{2\sigma}}{N^2}
  \sum_{i\neq j}\int_0^t\na V^\eta(Y_{ij}(s))\cdot\dd W_i(s).
\end{align*}
It follows from $N^{-1}\sum_{j=1}^N\na V^\eta(Y_{ij})
=(\na V^\eta*\mu^\eta)(X_i^\eta)$ that
\begin{align}\label{4.J1}
  J_1 &= \frac{2\kappa}{N}\sum_{i=1}^N\int_0^t
  |(\na V^\eta*\mu^\eta)(X_i^\eta(s))|^2\dd s \\
  &\phantom{xx}+ \frac{2\sigma}{N}\sum_{i=1}^N\int_0^t
  (\Delta V^\eta*\mu^{\eta})(X_i^\eta(s))\dd s 
  - \frac{2\sigma}{N}\int_0^t\Delta V^\eta(0)\dd s \nonumber \\
  &\phantom{xx}+ \frac{2\sqrt{2\sigma}}{N}\sum_{i=1}^N\int_0^t
  (\na V^\eta*\mu^\eta)(X_i^\eta(s))\dd W_i(s). \nonumber
\end{align}
Note that for all terms on the right-hand side, we have written the sum over $i\neq j$ as the sums over $i$ and $j$ minus the sum of the diagonal $i=j$. In the sum over $i=j$, we first need to evaluate $\na V^\eta(Y_{ii}(s))=\na V^\eta(0)$, which vanishes due to anti-symmetry of $\na V^\eta$. However, the expression
$\Delta V^\eta(Y_{ii}(s))=\Delta V^\eta(0)$ does generally not vanish, explaining the third term on the right-hand side.

{\em Step 2: Reformulation of $J_2$.}
We deduce from the symmetry of $Z^\eta$ that
\begin{equation}\label{4.symm}
  \langle Z^\eta*u,v\rangle = \langle u,Z^\eta*v\rangle
  \quad\mbox{for all }u,v\in L^2(\R^d),
\end{equation}
where the dual product $\langle\cdot,\cdot\rangle$ corresponds here to the inner product in $L^2(\R^d)$. Therefore,
$$
  \|g^\eta(t)\|_{L^2(\R^d)}^2 
  = \langle Z^\eta*\bar{u}^\eta(t),Z^\eta*\bar{u}^\eta(t)\rangle
  = \langle \bar{u}^\eta(t),Z^\eta*Z^\eta*\bar{u}^\eta(t)\rangle
  = \langle \bar{u}^\eta(t),V^\eta*\bar{u}^\eta(t)\rangle.
$$
Thus, taking $V^\eta*\bar{u}^\eta$ as a test function
in the weak formulation of equation \eqref{1.baru} for $\bar{u}^\eta$,
\begin{align*}
  \|g^\eta(t)\|_{L^2(\R^d)}^2 
  &= \langle\bar{u}^\eta(0),V^\eta*\bar{u}^\eta(0)\rangle
  + \int_0^t\langle\bar{u}^\eta,V^\eta*\pa_t\bar{u}^\eta\rangle\dd s \\
  &\phantom{xx}+ \sigma\int_0^t\langle\Delta\bar{u}^\eta,
  V^\eta*\bar{u}^\eta\rangle\dd s
  - \kappa\int_0^t\langle V^\eta*\bar{u}^\eta,
  \diver(\bar{u}^\eta\na V^\eta*\bar{u}^\eta)\rangle\dd s,
\end{align*}
and, after integrating by parts in the third term on the right-hand side,
\begin{align}\label{4.J2}
  J_2 &= \|g^\eta(t)\|_{L^2(\R^d)}^2 - \|g^\eta(0)\|_{L^2(\R^d)}^2
  = \int_0^t\langle\bar{u}^\eta,V^\eta*\pa_t\bar{u}^\eta\rangle\dd s \\
  &\phantom{xx}
  + \sigma\int_0^t\langle\bar{u}^\eta,
  \Delta V^\eta*\bar{u}^\eta\rangle\dd s 
  - \kappa\int_0^t\langle V^\eta*\bar{u}^\eta,
  \diver(\bar{u}^\eta\na V^\eta*\bar{u}^\eta)\rangle\dd s. \nonumber 
\end{align}

{\em Step 3: Reformulation of $J_3$.}
We determine $J_3$ by first calculating the mixed term
\begin{align*}
  \langle f^\eta(t),g^\eta(t)\rangle
  &= \langle\mu^\eta(t)*Z^\eta,\bar{u}^\eta(t)*Z^\eta\rangle
  = \langle\mu^\eta(t),Z^\eta*Z^\eta*\bar{u}^\eta(t)\rangle \\
  &= \langle\mu^\eta(t),V^\eta*\bar{u}^\eta(t)\rangle
  = \frac{1}{N}\sum_{i=1}^N V^\eta*\bar{u}^\eta(t,X_i^\eta(t)),
\end{align*}
where we have again used the symmetry of $Z^\eta$ (see \eqref{4.symm}). By It\^{o}'s lemma, applied to every term $V^\eta*\bar{u}^\eta(t,X_i^\eta(t))$, as in \eqref{4.ito},
\begin{align*}
  J_3 &= -\frac{2}{N}\sum_{i=1}^N\big(V^\eta*\bar{u}^\eta(t,X_i^\eta(t))
  - V^\eta*\bar{u}^\eta(0,X_i^\eta(0))\big) \\
  &= -\frac{2}{N}\sum_{i=1}^N\int_0^t\big[\pa_t(V^\eta*\bar{u}^\eta)
  + \kappa(\na V^\eta*\bar{u}^\eta)(\na V^\eta*\mu^\eta) 
  + \sigma\Delta V^\eta*\bar{u}^\eta\big](s,X_i^\eta(s))\dd s \\
  &\phantom{xx}- \frac{2\sqrt{2\sigma}}{N}\sum_{i=1}^N
  \int_0^t\na V^\eta*\bar{u}^\eta(s,X_i^\eta(s))\dd W_i(s).
\end{align*}
We use $\pa_t (V^\eta*\bar{u}^\eta) = V^\eta*\pa_t \bar{u}^\eta$ and insert equation \eqref{1.baru} for $\bar{u}^\eta$ (observe that $\pa_t \bar{u}^\eta \in L^2(\R^d)$ for fixed $\eta>0$):
\begin{align*}
  2\int_0^t&\pa_t(V^\eta*\bar{u}^\eta)(s,X_i^\eta(s))\dd s
  = \int_0^t V^\eta*\pa_t\bar{u}^\eta(s,X_i^\eta(s))\dd s
  + \int_0^t V^\eta*\pa_t\bar{u}^\eta(s,X_i^\eta(s))\dd s \\
  &= \int_0^t V^\eta*\pa_t\bar{u}^\eta(s,X_i^\eta(s))\dd s \\
  &\phantom{xx}+ \int_0^t\big[\sigma \Delta V^\eta*\bar{u}^\eta
  (s,X_i^\eta(s))
  - \kappa\na V^\eta*\big(\bar{u}^\eta\na V^\eta*\bar{u}^\eta\big)
  (s,X_i^\eta(s)\big)\big]\dd s,
\end{align*}
which allows us to write $J_3$ as
\begin{align}\label{4.J3}
  J_3 &= -\frac{1}{N}\sum_{i=1}^N\int_0^t V^\eta*\pa_t
  \bar{u}^\eta(s,X_i^\eta(s))\dd s \\
  &\phantom{xx}- \frac{1}{N}\sum_{i=1}^N\int_0^t
  \big[\sigma \Delta V^\eta*\bar{u}^\eta
  - \kappa\na V^\eta*\big(\bar{u}^\eta\na V^\eta*\bar{u}^\eta
  \big)\big](s,X_i^\eta(s))\dd s \nonumber \\
  &\phantom{xx}- \frac{2}{N}\sum_{i=1}^N\int_0^t
  \big[\kappa(\na V^\eta*\bar{u}^\eta)\cdot(\na V^\eta*\mu^\eta)
  + \sigma\Delta V^\eta*\bar{u}^\eta\big](s,X_i^\eta(s))\dd s \nonumber \\
  &\phantom{xx}- \frac{2\sqrt{2\sigma}}{N}\sum_{i=1}^N
  \int_0^t\na V^\eta*\bar{u}^\eta(s,X_i^\eta(s))\cdot\dd W_i(s). 
  \nonumber 
\end{align}

We combine expressions \eqref{4.J1}, \eqref{4.J2}, \eqref{4.J3} for $J_1$, $J_2$, $J_3$, respectively, to find from \eqref{4.J123} that
\begin{align}\label{4.K1to6}
  & \|(f^\eta-g^\eta)(t)\|_{L^2(\R^d)}^2 
  - \|(f^\eta-g^\eta)(0)\|_{L^2(\R^d)}^2
  = (K_1+\cdots+K_6)(t), \quad\mbox{where} \\
  & K_1(t) = -\frac{2\sigma}{N}\int_0^t\Delta V^\eta(0)\dd s
  = -\frac{2\sigma t}{N}\Delta V^\eta(0), \nonumber \\
  & K_2(t) = \frac{\sigma}{N}\sum_{i=1}^N\int_0^t\big(2(
  \Delta V^\eta*\mu^\eta)(X_i^\eta(s)) 
  + \langle\bar{u}^\eta(s),
  (\Delta V^\eta*\bar{u}^\eta)(s)\rangle \nonumber \\
  &\phantom{K_1(t)xx}- 3\Delta V^\eta*\bar{u}^\eta(s,X_i^\eta(s))\big)\dd s, \nonumber \\
  & K_3(t) = \int_0^t\langle\bar{u}^\eta(s),
  V^\eta*\pa_t\bar{u}^\eta(s)\rangle\dd s
  - \frac{1}{N}\sum_{i=1}^N\int_0^t 
  V^\eta*\pa_t\bar{u}^\eta(s,X_i^\eta(s))\dd s, \nonumber \\
  & K_4(t) = -\kappa\int_0^t\langle V^\eta*\bar{u}^\eta,
  \diver(\bar{u}^\eta\na V^\eta*\bar{u}^\eta)\rangle\dd s \nonumber \\
  &\phantom{K_1(t)xx}+ \frac{\kappa}{N}\sum_{i=1}^N\int_0^t
  \na V^\eta*(\bar{u}^\eta\na V^\eta*\bar{u}^\eta)
  (s,X_i^\eta(s))\dd s, \nonumber \\
  & K_5(t) = \frac{2\kappa}{N}\sum_{i=1}^N\int_0^t
  \big(|\na V^\eta*\mu^\eta|^2 - (\na V^\eta*\bar{u}^\eta)
  \cdot(\na V^\eta*\mu^\eta)\big)(X_i^\eta(s))\dd s, \nonumber \\
  & K_6(t) = \frac{\sqrt{8\sigma}}{N}\sum_{i=1}^N\int_0^t
  \na V^\eta*(\mu^\eta-\bar{u}^\eta)(s,X_i^\eta(s))\cdot\dd W_i(s).
  \nonumber 
\end{align}
In the next subsection, we rewrite $K_2,\ldots,K_5$ and directly estimate $K_1$ and $K_6$ at the end.

\subsection{Second reformulation}\label{sec.second}

We reformulate the terms $K_2,\ldots,K_5$ in \eqref{4.K1to6} in such a way that some terms can be combined or cancel out. 

{\em Step 1: Reformulation of $K_2(t)$.} Using the definition of $\mu^\eta$, we write
\begin{align*}
  K_2(t) &= 2\sigma\int_0^t\big\langle\mu^\eta(s),
  \Delta V^\eta*(\mu^\eta-\bar{u}^\eta)(s)\big\rangle\dd s \\
  &\phantom{xx}+ \sigma\int_0^t\big\langle(
  \bar{u}^\eta-\mu^\eta)(s),\Delta V^\eta*\bar{u}^\eta(s)
  \big\rangle\dd s.
\end{align*}
Because of $V^\eta=Z^\eta*Z^\eta$ and the symmetry of $Z^\eta$, an integration by parts shows that
$$
  K_2(t) = -2\sigma\int_0^t\langle\na(f^\eta-g^\eta),
  \na f^\eta\rangle\dd s + \sigma\int_0^t\langle
  \na(f^\eta-g^\eta),\na g^\eta\rangle\dd s.
$$

{\em Step 2: Reformulation of $K_3(t)$.}
With the definition $V^\eta=Z^\eta*Z^\eta$,	property \eqref{4.symm}, and equation \eqref{1.baru} for $\bar{u}^\eta$, we infer that
\begin{align*}
  K_3(t) &= \int_0^t\langle\bar{u}^\eta,
  Z^\eta*Z^\eta*\pa_t\bar{u}^\eta\rangle\dd s
  - \int_0^t\langle\mu^\eta,Z^\eta*Z^\eta*\pa_t\bar{u}^\eta\rangle\dd s \\
  &= \int_0^t\langle(\bar{u}^\eta-\mu^\eta)*Z^\eta,
  Z^\eta*\pa_t\bar{u}^\eta\rangle\dd s
  = \int_0^t\langle g^\eta-f^\eta,Z^\eta*\pa_t\bar{u}^\eta\rangle\dd s \\
  &= \sigma\int_0^t\langle g^\eta-f^\eta,\Delta Z^\eta*\bar{u}^\eta
  \rangle\dd s - \kappa\int_0^t\big\langle  
  g^\eta-f^\eta,Z^\eta*\diver(\bar{u}^\eta\na V^\eta*\bar{u}^\eta)
  \big\rangle\dd s \\
  &= \sigma\int_0^t\langle\na(f^\eta-g^\eta),\na g^\eta\rangle\dd s
  - \kappa\int_0^t\big\langle\na Z^\eta*(f^\eta-g^\eta),
  \bar{u}^\eta\na V^\eta*\bar{u}^\eta\big\rangle\dd s,
\end{align*}
where we integrated by parts in the last step. The first term on the right-hand side is the same as the last term in $K_2(t)$, which shows that
\begin{align}\label{4.K2K3}
  K_2(t) + K_3(t) &= -2\sigma \int_{0}^t \big\langle\na(f^\eta-g^\eta),
  \na(f^\eta-g^\eta)\big\rangle\dd s \\
  &\phantom{xx}-\kappa\int_0^t\big\langle
  \na Z^\eta*(f^\eta-g^\eta),\bar{u}^\eta\na V^\eta*\bar{u}^\eta
  \big\rangle\dd s. \nonumber
\end{align}

{\em Step 3: Reformulation of $K_4(t)$.}
Using the symmetry of $Z^\eta$ and property \eqref{4.symm} again, the first term of $K_4(t)$ becomes
\begin{align*}
  -\kappa \langle V^\eta&*\bar{u}^\eta,\diver(\bar{u}^\eta
  \na V^\eta*\bar{u}^\eta)\rangle
  = -\kappa\langle g^\eta,\na Z^\eta*(\bar{u}^\eta
  \na V^\eta*\bar{u}^\eta)\rangle.
\end{align*}
Similarly, the second term in $K_4(t)$ becomes
\begin{align*}
  \frac{\kappa}{N}&\sum_{i=1}^N\na V^\eta*(\bar{u}^\eta
  \na V^\eta*\bar{u}^\eta)(X_i^\eta)
 = \kappa\big\langle\mu^\eta*Z^\eta,\na Z^\eta*(\bar{u}^\eta
  \na V^\eta*\bar{u}^\eta)\rangle \\
  &= \kappa\langle f^\eta,\na Z^\eta
  *(\bar{u}^\eta\na V^\eta*\bar{u}^\eta)\rangle.
\end{align*}
It follows from the anti-symmetry of $\na Z^\eta$ that
\begin{align*}
  K_4(t) &= \kappa\int_0^t\langle f^\eta-g^\eta,
  \na Z^\eta*(\bar{u}^\eta\na V^\eta*\bar{u}^\eta)\rangle\dd s \\
  &= -\kappa\int_0^t\langle \na Z^\eta*(f^\eta-g^\eta),
  \bar{u}^\eta\na V^\eta*\bar{u}^\eta\rangle\dd s,
\end{align*}
which equals the second term of $K_3(t)$. Hence, together with \eqref{4.K2K3},
\begin{align*}%\label{4.K2K3K4}
  K_2(t) + K_3(t) + K_4(t)
  &= -2\sigma \int_{0}^t \big\langle\na(f^\eta-g^\eta),
  \na(f^\eta-g^\eta)\big\rangle\dd s \\
  &\phantom{xx}- 2\kappa\int_0^t\big\langle
  \na Z^\eta*(f^\eta-g^\eta),
  \bar{u}^\eta\na V^\eta*\bar{u}^\eta\big\rangle\dd s. \nonumber
\end{align*}

{\em Step 4: Reformulation of $K_5(t)$.}
We use $\na V^\eta=Z^\eta*\na Z^\eta$ to rewrite $K_5(t)$ according to
\begin{align*}
  K_5(t) &= \frac{2\kappa}{N}\sum_{i=1}^N\int_0^t(\na V^\eta*\mu^\eta)
  (X_i^\eta)\cdot\big(Z^\eta*\na Z^\eta*(\mu^\eta(X_i^\eta)
 - \bar{u}^\eta(X_i^\eta))\big)\dd s \\
  &= 2\kappa\int_0^t\big\langle\mu^\eta,(\na V^\eta*\mu^\eta)
  \cdot(Z^\eta*\na(f^\eta-g^\eta))\big\rangle\dd s.
\end{align*}

{\em Step 5: Reformulation of $K_2(t)+\cdots+K_5(t)$.}
We add the expressions for $K_2,\ldots,K_5$ derived in the previous steps:
\begin{align}\label{4.aux}
  (K_2+\cdots+K_5)(t) &= -2\sigma\int_0^t\big\langle\na(f^\eta-g^\eta),
  \na(f^\eta-g^\eta)\big\rangle\dd s \\
  &\phantom{xx}- 2\kappa\int_0^t\big\langle
  \na Z^\eta*(f^\eta-g^\eta),\bar{u}^\eta
  \na V^\eta*\bar{u}^\eta\big\rangle\dd s \nonumber \\
  &\phantom{xx}+ 2\kappa\int_0^t\big\langle\mu^\eta,
  (\na V^\eta*\mu^\eta)\cdot(\na Z^\eta*(f^\eta-g^\eta))
  \big\rangle\dd s.	\nonumber
\end{align}
We rewrite the second term on the right-hand side 
by adding and subtracting some terms in the second argument of the dual bracket,
\begin{equation*}
  \bar{u}^\eta\na V^\eta*\bar{u}^\eta
  = (\bar{u}^\eta-\mu^\eta)\na V^\eta*\bar{u}^\eta
  + \mu^\eta\na V^\eta*(\bar{u}^\eta-\mu^\eta) 
  + \mu^\eta\na V^\eta*\mu^\eta.
\end{equation*}
Inserting this expression, we see that the last term cancels with the last integral on the right-hand side of \eqref{4.aux}. Furthermore, taking into account that $\na V^\eta*(\bar{u}^\eta-\mu^\eta)=\na Z^\eta*Z^\eta*(\bar{u}^\eta-\mu^\eta)
=-\na Z^\eta*(f^\eta-g^\eta)$, we have
\begin{align*}
  -2\kappa\int_0^t&\big\langle\na Z^\eta*(f^\eta-g^\eta),
  (\bar{u}^\eta-\mu^\eta)\na V^\eta*\bar{u}^\eta
  + \mu^\eta\na V^\eta*(\bar{u}^\eta-\mu^\eta)\big\rangle\dd s \\
  &=- 2\kappa\int_0^t\big\langle\bar{u}^\eta-\mu^\eta, 
  (\na V^\eta*\bar{u}^\eta)
  \cdot(\na Z^\eta*(f^\eta-g^\eta))\big\rangle\dd s \\
  &\phantom{xx}+ 2\kappa\int_0^t\big\langle\mu^\eta,
  |\na Z^\eta*(f^\eta-g^\eta)|^2\big\rangle\dd s,
\end{align*}
and we end up with
\begin{align}\label{4.K2345}
  (K_2+\cdots+K_5)(t) 
  &= -2\sigma\int_0^t\|\na(f^\eta-g^\eta)\|_{L^2(\R^d)}^2\dd s \\
  &\phantom{xx}- 2\kappa\int_0^t\big\langle\bar{u}^\eta-\mu^\eta, 
  (\na V^\eta*\bar{u}^\eta)\cdot(\na Z^\eta*(f^\eta-g^\eta))
  \big\rangle\dd s \nonumber \\
  &\phantom{xx}+ 2\kappa\int_0^t\big\langle\mu^\eta,
  |\na Z^\eta*(f^\eta-g^\eta)|^2\big\rangle\dd s. \nonumber
\end{align}
In the repulsive case $\kappa=-1$, the last term is nonpositive and can be used to absorb other terms; see \cite{Oel87}. However, in the attractive case $\kappa=1$, we need to estimate this expression, which complicates the proof considerably.

{\em Step 6: End of the reformulation.}
We insert expression \eqref{4.K2345} for $K_2+\cdots+K_5$ into \eqref{4.K1to6}, take the supremum over $0<t<T$, and the expectation:
\begin{align}\label{4.KLM}
  \E\Big(&\sup_{0<t<T}\|(f^\eta-g^\eta)(t)\|_{L^2(\R^d)}^2\Big) 
  + 2\sigma\E\int_0^T\|\na(f^\eta-g^\eta)(s)\|_{L^2(\R^d)}^2\dd s \\
  &\le \E\|(f^\eta-g^\eta)(0)\|_{L^2(\R^d)}^2
  + \E\Big(\sup_{0<t<T}|K_1(t)+K_6(t)|\Big) + L(T) + \kappa M(T), \nonumber
\end{align}
where (since $|\kappa|=1$)
\begin{align}
  L(T) &= 2\E\bigg(\sup_{0<t<T}\bigg|\int_0^t\big\langle
  \mu^\eta-\bar{u}^\eta, (\na V^\eta*\bar{u}^\eta)
  \cdot(\na Z^\eta*(f^\eta-g^\eta))\big\rangle\dd s\bigg|\bigg), \label{4.L} \\
  M(T) &= 2\E\bigg(\sup_{0<t<T}\int_0^t\big\langle
  \mu^\eta,|\na Z^\eta*(f^\eta-g^\eta)|^2\big\rangle
  \dd s\bigg). \label{4.M}
\end{align}

The term $K_1(t)$ can be estimated directly by using Lemma \ref{lem.DkV}:
\begin{equation}\label{4.K1}
  K_1(t) = -\frac{2\sigma t}{N}\Delta V^\eta(0)
  \le C(T)N^{-1}\|\mathrm{D}^2 V^\eta\|_{L^\infty(\R^d)}
  \le C(T)N^{\beta(\lambda +2)-1} \le C(T)N^{-1/2-\eps}
\end{equation}
for some $\eps>0$ if $\beta<1/(2\lambda+2)$.
To estimate $K_6(t)$, defined in \eqref{4.K1to6}, we use the Burkholder--Davis--Gundy inequality and Jensen's inequality:
$$
  \E\Big(\sup_{0<t<T}|K_6(t)|\Big) 
  \le C\E(\langle K_6\rangle_T^{1/2})
  \le C(\E\langle K_6\rangle_T)^{1/2},
$$
where $\langle K_6\rangle_T$ is the quadratic variation process of $K_6$ at time $T>0$. Since for different particles, the Brownian motions $W_i$ are independent, the quadratic variation becomes
\begin{align*}
  (\E\langle K_6\rangle_T)^{1/2} 
  &= \bigg(\frac{8\sigma}{N^2}\sum_{i=1}^N\E\int_0^T\big|
  \na V^\eta*\big(\mu^\eta(s,X_i^\eta(s))
  - \bar{u}^\eta(s,X_i^\eta(s))\big)\big|^2\dd s\bigg)^{1/2} \\
  &= \bigg(\frac{8\sigma}{N}\E\int_0^T\big\langle\mu^\eta(s),
  |\na V^\eta*(\mu^\eta-\bar{u}^\eta)(s)|^2\big\rangle
  \dd s\bigg)^{1/2} \\
  &= \bigg(\frac{8\sigma}{N}\E\int_0^T\big\langle\mu^\eta(s),
  |\na Z^\eta*(f^\eta - g^\eta)(s)|^2\big\rangle\dd s\bigg)^{1/2}.
\end{align*}
We infer from Young's inequality and definition \eqref{4.M} of $M(t)$ that
\begin{equation}\label{4.K6}
  \E\Big(\sup_{0<t<T}|K_6(t)|\Big) 
  \le\frac{C(\sigma)}{N} + 2\E\int_0^T\big\langle\mu^\eta,
  |\na Z^\eta*(f^\eta - g^\eta)|^2\big\rangle\dd s
  \le \frac{C(\sigma)}{N} + M(T),
\end{equation} 
where we used the fact that $M(T)$ is non-negative. Note that if $\kappa=-1$, we could use the nonpositive term $\kappa M(T)$ in \eqref{4.KLM} to absorb the last term in \eqref{4.K6}. 

It remains to estimate $L(T)$ and $M(T)$. We start with the estimate of $M(T)$ before turning to the slightly more involved (but similar) calculation of $L(T)$.

%%%%%%%%%%%%%

\subsection{Estimation of $M(T)$}

By inserting definition \eqref{1.mu} of the empirical measure, we can formulate \eqref{4.M} as 
\begin{align*}
  M(T) &= 2\E\bigg(\int_0^T \frac{1}{N} \sum_{i = 1}^N 
  \big|\nabla V^\eta*(\mu^\eta-\bar{u}^\eta) \big|^2 
  (X_i^\eta)\dd s\bigg) \\
  &= 2\E\bigg(\int_0^T \frac{1}{N} \sum_{i = 1}^N 
  \bigg|\frac{1}{N} \sum_{j = 1}^N \nabla V^\eta(X_i^\eta - X_j^\eta) 
  - \nabla V^\eta*\bar{u}^\eta(X_i^\eta)\bigg|^2\dd s\bigg).
\end{align*}
By adding and subtracting $\nabla V^\eta(\bar{X}_i^\eta - \bar{X}_j^\eta)$ as well as $\nabla V^\eta*\bar{u}^\eta (\bar{X}_j^\eta)$, we write $M(T) \leq M_1 + M_2 + M_3$, where
\begin{align*}
  M_1 &:= 2\mathbb{E}\bigg(\int_{0}^T \frac{1}{N}\sum_{i = 1}^N 
  \bigg|\frac{1}{N}\sum_{j = 1}^N 
  \big(\nabla V^\eta(X_i^\eta-X_j^\eta)-\nabla V^\eta(\bar{X}_i^\eta  
  -\bar{X}_j^\eta)\big) \bigg|^2\dd s\bigg) \\
  M_2 &:= 2\mathbb{E}\bigg(\int_{0}^T \frac{1}{N}\sum_{i = 1}^N 
  \bigg|\frac{1}{N}\sum_{j = 1}^N 
  \big(\nabla V^\eta(\bar{X}_i^\eta-\bar{X}_j^\eta)
  -\nabla  V^\eta\ast \bar{u}^\eta(\bar{X}_i^\eta)\big) 
  \bigg|^2\dd s\bigg) \\
  M_3 &:= 2\mathbb{E}\bigg(\int_{0}^T \frac{1}{N}\sum_{i = 1}^N 
  \big|\na V^\eta*\bar{u}^\eta(\bar{X}_i^\eta) 
  - \nabla V^\eta\ast \bar{u}^\eta(X_i^\eta) \big|^2\dd s\bigg).
\end{align*}

{\em Step 2a: Estimation of $M_{1}$.} To apply the mean-field estimate in Theorem \ref{thm.prob}, we introduce the set
\begin{equation}\label{Calpha}
  \mathcal{C}_\alpha(t) := \bigg\{\omega\in\Omega:\max_{i=1,\ldots,N}
  |X_i^\eta(t,\omega)-\bar{X}_i^\eta(t,\omega)|>N^{-\alpha}\bigg\},
\end{equation}
where $\alpha>0$ will be chosen later. By Theorem \ref{thm.prob}, for
any $\gamma>0$ and $T>0$, there exists $C(\gamma,T)>0$ such that
\begin{equation}\label{PCalpha}
  \sup_{0<t<T}\Prob(\mathcal{C}_\alpha(t))\le C(\gamma,T)N^{-\gamma}.
\end{equation}
The idea is to split $\Omega$ into the set $\mathcal{C}_\alpha(s)$ and its complement $\mathcal{C}_\alpha(s)^c$ with $s\in(0,t)$ and to estimate $M_{1}$ on these two sets separately.
This yields $M_1=M_{11}+M_{12}$, where
\begin{align*}
  M_{11} &:= 2\mathbb{E}\bigg(\int_{0}^T \frac{1}{N}\sum_{i = 1}^N 
  \bigg|\frac{1}{N}\sum_{j = 1}^N \big(\nabla V^\eta(X_i^\eta-X_j^\eta)
  -\nabla V^\eta(\bar{X}_i^\eta -\bar{X}_j^\eta)\big) \bigg|^2
  \mathrm{1}_{\mathcal{C}_\alpha(s)}\dd s\bigg), \\
  M_{12} &:= 2\mathbb{E}\bigg(\int_{0}^T \frac{1}{N}\sum_{i = 1}^N 
  \bigg|\frac{1}{N}\sum_{j = 1}^N \big(\nabla V^\eta(X_i^\eta-X_j^\eta)
  -\nabla V^\eta(\bar{X}_i^\eta -\bar{X}_j^\eta)\big) \bigg|^2
  \mathrm{1}_{\mathcal{C}_\alpha(s)^c}\dd s\bigg).
\end{align*}
We deduce from definition \eqref{PCalpha} of $\mathcal{C}_\alpha(s)$, Fubini's theorem, and Lemma \ref{lem.DkV} that
\begin{align*}
  M_{11} \leq  C \|\nabla V^\eta\|_{L^\infty(\R^d)}^2\int_{0}^T
  \Prob(\mathcal{C}_\alpha(s)) \dd s 
  \leq C N^{2\beta(\lambda +1)-\gamma} \leq CN^{-1/2-\eps},
\end{align*}
where the last step follows after choosing $\gamma>0$ sufficiently large ($\eps>0$ can be arbitrary here). On the complement of $\mathcal{C}_\alpha(s)$, the mean-value theorem and again Lemma \ref{lem.DkV} lead to 
\begin{align*}
  M_{12} &\leq  2 \| \mathrm{D}^2 V^\eta\|_{L^\infty(\R^d)}^2\mathbb{E}\bigg(\int_{0}^T 
  \frac{1}{N}\sum_{i = 1}^N \frac{1}{N}\sum_{j = 1}^N 
  \big|X_i^\eta - X_j^\eta - \bar{X}_i^\eta + \bar{X}_j^\eta \big|^2\mathrm{1}_{\mathcal{C}_\alpha(s)^c}\dd s\bigg)\\
&\leq C N^{2\beta(\lambda +2)}N^{-2\alpha} \le CN^{-1/2-\eps}.
\end{align*}
The last step follows if we are able to choose $\alpha>0$ such that $2\beta(\lambda +2) -2\alpha < -1/2$, which is equivalent to $\alpha > 1/4 + \beta(\lambda +2)$. This choice is admissible in Theorem \ref{thm.prob}, since $1/4 + \beta(\lambda +2) < \alpha < 1/2 - \beta(\lambda +1)$ is nonempty if $\beta < 1/(8\lambda+12)$. The estimates for $M_{11}$ and $M_{12}$ show that there exists $\eps>0$ such that
\begin{align*}%\label{estimateM1}
  M_1 \leq CN^{-1/2-\eps}.
\end{align*}

\textit{Step 2b: Estimation of $M_{3}$.} Similar as in the previous step, we apply Theorem \ref{thm.prob} with the set $\mathcal{C}_\alpha(s)$, defined in \eqref{Calpha}, possibly choosing a different $\alpha>0$; see below. Again, by Theorem \ref{thm.prob}, estimate \eqref{PCalpha} holds. We split $\Omega$ into the set $\mathcal{C}_\alpha(s)$ and its complement $\mathcal{C}_\alpha(s)^c$ with $s\in(0,t)$ to find that $M_3=M_{31}+M_{32}$, where
\begin{align*}
  M_{31} &:= 2\mathbb{E}\bigg(\int_{0}^T \frac{1}{N}\sum_{i = 1}^N 
  \big|\nabla V^\eta*\bar{u}^\eta(\bar{X}_i^\eta) 
  - \nabla V^\eta*\bar{u}^\eta(X_i^\eta) \big|^2\mathrm{1}_{\mathcal{C}_\alpha(s)}\dd s\bigg), \\
  M_{32} &:= 2\mathbb{E}\bigg(\int_{0}^T \frac{1}{N}\sum_{i = 1}^N 
  \big|\nabla V^\eta*\bar{u}^\eta(\bar{X}_i^\eta) 
  - \nabla V^\eta*\bar{u}^\eta(X_i^\eta) \big|^2
  \mathrm{1}_{\mathcal{C}_\alpha(s)^c}\dd s\bigg).
\end{align*}
We infer from Lemma \ref{lem.DkVbaru} that the $L^\infty(\R^d)$ norm of 
$\nabla V^\eta*\bar{u}^\eta(\bar{X}_i^\eta)$ is bounded. Then, by definition of $\mathcal{C}_\alpha(s)$, similarly as in Step 2a,
\begin{align*}
  M_{31} \leq C(T) \sup_{0<s<T} \Prob(\mathcal{C}_\alpha(s)) 
  \le C(T)N^{-\gamma} \leq C(T)N^{-1/2-\eps},
\end{align*}
by taking $\gamma > 1/2$. The second term $M_{32}$ is also estimated like in Step 2a:
\begin{align*}
  M_{32} \leq C(T) \|\mathrm{D}^2 V^\eta*\bar{u}^\eta 
  \|_{L^\infty(\R^d)}^2 N^{-2\alpha} 
  \leq C(T)N^{-2\alpha} \leq C(T)N^{-1/2-\eps}
\end{align*}
for $\alpha > 1/4$, which is possible according to the choice of $\beta$ in Theorem \ref{thm.prob}. Consequently, for some $\eps>0$,
\begin{align*}
  M_3 \leq C(T)N^{-1/2-\eps}.
\end{align*}

\textit{Step 2c: Estimation of $M_{2}$.} The term $M_{2}$ is treated by the law-of-large-numbers estimate of
Lemma \ref{lem.law}. For this, we introduce for some fixed $\theta>0$ (which will be chosen later):
\begin{equation}\label{DthetaN}
  \mathcal{B}_\theta(s) := \bigcup_{i=1}^N \bigg\{
  \omega\in\Omega:\bigg|\frac{1}{N}\sum_{j=1}^N
  \na  V^\eta(\bar{X}_i^\eta(s)-\bar{X}_j^\eta(s)) 
  - \na  V^\eta*\bar{u}^\eta(s,\bar{X}_i^\eta)\bigg|>N^{-\theta}\bigg\},
\end{equation}
and we split $\Omega$ into the sets $\mathcal{B}_\theta(s)$ and 
$\mathcal{B}_\theta(s)^c$. Then $M_{2}= M_{21}+M_{22}$, where
\begin{align*}
  M_{21} &:=2\mathbb{E}\bigg(\int_{0}^T \frac{1}{N}\sum_{i = 1}^N 
  \bigg|\frac{1}{N}\sum_{j = 1}^N 
  \big(\nabla V^\eta(\bar{X}_i^\eta-\bar{X}_j^\eta)
  - \nabla V^\eta*\bar{u}^\eta(\bar{X}_i^\eta)\big) \bigg|^2
  \mathrm{1}_{\mathcal{B}_\theta(s)}\dd s\bigg), \\
  M_{22} &:= 2\mathbb{E}\bigg(\int_{0}^T \frac{1}{N}\sum_{i = 1}^N 
  \bigg|\frac{1}{N}\sum_{j = 1}^N 
  \big(\nabla V^\eta(\bar{X}_i^\eta-\bar{X}_j^\eta)
  - \nabla V^\eta*\bar{u}^\eta(\bar{X}_i^\eta)\big) \bigg|^2
  \mathrm{1}_{\mathcal{B}_\theta(s)^c}\dd s\bigg).
\end{align*}
We use Lemma \ref{lem.law} with $\psi_\eta = \na V^\eta$ to find that
\begin{align*}
  M_{21} &\leq C(T) \big(\|\nabla V^\eta\|_{L^\infty(\R^d)}^2 
  + \|\nabla V^\eta*\bar{u}^\eta \|_{L^\infty(\R^d)}^2 \big) 
  \sup_{0<s<T} \Prob(\mathcal{B}_\theta(s)) \\
  &\leq C(T)(N^{2\beta(\lambda+1)}+1)N^{2m(\beta(\lambda +1)
  + \theta-1/2)+1} \leq N^{-1/2-\eps},
\end{align*}
where the last step follows by taking $0<\theta<1/2 - \beta(\lambda +1)$ and $m\in\N$ sufficiently large. For $M_{22}$, the construction of $\mathcal{B}_\theta$ leads to 
\begin{align*}
  M_{22} \leq C(T)N^{-2\theta} \leq C(T)N^{-1/2-\eps},
\end{align*}
where the latter inequality follows from the fact that we can choose $1/4<\theta< 1/2-\beta(\lambda +1)$ as long as $\beta < 1/(4\lambda+4)$. We summarize the bounds on $M_1$, $M_2$, and $M_3$:
\begin{align}\label{4.MM}
  M(T) \leq C(T)N^{-1/2-\eps}\quad\mbox{for some }\eps>0.
\end{align}

\subsection{Estimation of $L(T)$}%\label{sec.estL}

An expression like $L(T)$, defined in \eqref{4.L}, has been estimated in \cite{Oel87} using a Taylor approximation and Fourier estimates in one space dimension. This approach is also feasible in higher space dimensions but it would become very tedious in notation. Here, we can reduce the assumptions on the potential $V^\eta$ compared to \cite{Oel87}, since we do not need any condition on the Fourier transform of the potential and we do not need to impose a compact support. 

We stress the fact that the following calculations do not require that $\Phi$ is of sub-Coulomb type. This assumption is needed to apply Theorem \ref{thm.prob}. The calculations below are also true for more singular kernels under the condition that Theorem \ref{thm.prob} holds.

Our idea is, as above, to split the integral over $\Omega$ 
in a mean-field part and a law-of-large-numbers part.
We add and subtract the empirical measure 
$$
  \bar\mu^\eta(s) := \frac{1}{N}\sum_{i=1}^N\delta_{\bar{X}_i^\eta(s)}
$$
of the intermediate problem \eqref{1.barX} to $L(T)$, recalling definition \eqref{4.L} of $L(T)$. Then $L(T)\le L_1+L_2$, where
\begin{align*}%\label{4.L1L2}
  L_1 &= 2\E\bigg(\sup_{0<t<T}\bigg|\int_0^t\big\langle
  \mu^\eta-\bar\mu^\eta,(\na  V^\eta*\bar{u}^\eta)
  \cdot(\na Z^\eta*(f^\eta-g^\eta))\big\rangle\dd s\bigg|\bigg), \\
  L_2 &= 2\E\bigg(\sup_{0<t<T}\bigg|\int_0^t\big\langle
  \bar\mu^\eta-\bar{u}^\eta,
  (\na V^\eta*\bar{u}^\eta)\cdot(\na Z^\eta*(f^\eta-g^\eta))
  \big\rangle\dd s\bigg|\bigg). \nonumber
\end{align*}
The term $L_1$ can be considered as the mean-field part, while $L_2$
is the law-of-large-numbers part.
	
\textit{Step 1: Estimation of $L_1$.} We start with $L_1$ and add and subtract
$(\na V^\eta*\bar{u}^\eta)(X_i^\eta)(\na Z^\eta*(f^\eta-g^\eta))
(\bar{X}_i^\eta)$, leading to $L_1\le L_{11}+L_{12}$, where
\begin{align*}
  L_{11} &= 2\E\int_0^T\bigg|\frac{1}{N}\sum_{i=1}^N
  \big((\na V^\eta*\bar{u}^\eta)(s,X_i^\eta(s)) 
  - (\na V^\eta*\bar{u}^\eta)(s,\bar{X}_i^\eta(s))\big) \\
  &\phantom{xx}{}\times(\na Z^\eta*(f^\eta-g^\eta))
  (s,\bar{X}_i^\eta(s))\bigg|\dd s, \\
  L_{12} &= 2\E\int_0^T\bigg|\frac{1}{N}\sum_{i=1}^N
  (\na V^\eta*\bar{u}^\eta)(s,X_i^\eta(s)) \\
  &\phantom{xx}{}\times\big((\na Z^\eta*(f^\eta-g^\eta))
  (s,X_i^\eta(s))- (\na Z^\eta*(f^\eta-g^\eta))
  (s,\bar{X}_i^\eta(s))\big)\bigg|\dd s.
\end{align*}

\textit{Step 2a: Estimate for $L_{11}$.} We split $\Omega$ for each time $0<s<T$ into the sets $\mathcal{C}_\alpha(s)$ and
$\mathcal{C}_\alpha(s)^c$, use the definition  $\mathcal{C}_\alpha(s)^c=\{\max_i|X_i^\eta-\bar{X}_i^\eta|\le N^{-\alpha}\}$ in the integral over $\mathcal{C}_\alpha(s)^c$ (leading to the factor $N^{-\alpha}$) and Lemma \ref{lem.law} in the integral over $\mathcal{C}_\alpha(s)$ (leading to the factor $N^{-\gamma}$ for any $\gamma>0$). Then, by the mean-value theorem applied to $\na V^\eta*\bar{u}^\eta$,
\begin{align*}
  L_{11} &\le CN^{-\alpha}\|\mathrm{D}^2 V^\eta*\bar{u}^\eta
  \|_{L^\infty(\R^d)}\E\int_0^T\frac{1}{N}\sum_{i=1}^N
  \bigg|(Z^\eta*(f^\eta-g^\eta))(s,\bar{X}_i^\eta(s))\bigg|\dd s \\
  &\phantom{xx}{}+ C(T)\sup_{\omega\in\Omega}\sup_{0<s<T}
  \|Z^\eta*\na(f^\eta-g^\eta)\|_{L^\infty(\R^d)}
  \|\na V^\eta*\bar{u}^\eta\|_{L^\infty(\R^d)}
  \sup_{0<s<T}\Prob(\mathcal{C}_\alpha(s)).
\end{align*}	
We deduce from
\begin{align*}
  \|Z^\eta*\na(f^\eta-g^\eta)\|_{L^\infty(\R^d)} 
  &= \sup_{x\in\R^d} \bigg|\frac{1}{N} \sum_{i = 1}^N 
  \nabla V^\eta(x-X_i^\eta) - \nabla V^\eta*\bar{u}^\eta(x)\bigg| \\
  &\leq C\|\nabla V^\eta\|_{L^\infty(\R^d)} \le CN^{\beta(\lambda+1)}
\end{align*}
(see Lemma \ref{lem.DkV}) and  $\|\mathrm{D}^2V^\eta*\bar{u}^\eta\|_{L^\infty(\R^d)}\le C$ (Lemma \ref{lem.DkVbaru}) that
\begin{align*}%\label{L_11}
  L_{11} &\le CN^{-\alpha}\E\int_0^T\frac{1}{N}\sum_{i=1}^N
  \bigg|(Z^\eta*\na(f^\eta-g^\eta))(s,\bar{X}_i^\eta(s))\bigg|\dd s
  + C(T)N^{\beta(\lambda+1)-\gamma}.
\end{align*}
The admissible range for $\alpha$ will be chosen later.
It remains to estimate the first term on the right-hand side. To this end, we add and subtract the term $N^{-1}\sum_{j = 1}^N \nabla V^\eta(\bar{X}_j^\eta-\bar{X}_i^\eta)$:
\begin{align}\label{4.L_111_and_112}
  C&N^{-\alpha}\E\int_0^T\frac{1}{N}\sum_{i=1}^N
  \bigg|(Z^\eta*\na(f^\eta-g^\eta))(s,\bar{X}_i^\eta(s))\bigg|\dd s  \nonumber\\
  &= CN^{-\alpha}\E\int_0^T\frac{1}{N}\sum_{i=1}^N
  \bigg|(\na V^\eta*(\mu^\eta- \bar{u}^\eta))
  (s,\bar{X}_i^\eta(s))\bigg|\dd s  \nonumber \\
  &= CN^{-\alpha}\E\int_0^T\frac{1}{N}\sum_{i=1}^N
  \bigg|\frac{1}{N}\sum_{j = 1}^N \nabla V^\eta
  (X_j^\eta - \bar{X}_i^\eta) - \nabla V^\eta* \bar{u}^\eta
  (\bar{X}_i^\eta)\bigg|\dd s  \nonumber\\
  &\leq CN^{-\alpha}\E\int_0^T\frac{1}{N}\sum_{i=1}^N
  \bigg|\frac{1}{N}\sum_{j = 1}^N \nabla V^\eta(X_j^\eta 
  - \bar{X}_i^\eta) - \nabla V^\eta(\bar{X}_j^\eta - \bar{X}_i^\eta)
  \bigg|\dd s  \nonumber\\
  &\phantom{xx}+ CN^{-\alpha}\E\int_0^T\frac{1}{N}\sum_{i=1}^N
  \bigg|\frac{1}{N}\sum_{j = 1}^N \nabla V^\eta(\bar{X}_j^\eta 
  - \bar{X}_i^\eta) - \nabla V^\eta*\bar{u}^\eta
  (\bar{X}_i^\eta)\bigg|\dd s \nonumber \\
  &=: L_{111} + L_{112}.
\end{align}

For $L_{111}$, we again exploit the mean-field convergence in probability (Theorem \ref{thm.prob}) and the properties of the sets $\mathcal{C}_\alpha(s)$ and $\mathcal{C}_\alpha(s)^c$ (with possibly a different choice of $\alpha$ as for the terms before), defined in \eqref{Calpha}, similarly as before:
\begin{align*}
  L_{111} &\leq CN^{-\alpha} \|\nabla V^\eta\|_{L^\infty(\R^d)}
  N^{-\gamma} \\ 
  &\phantom{xx}+ C(T)\frac{N^{-\alpha}}{N^2}\sum_{i,j = 1}^N 
    \|\mathrm{D}^2 V^\eta\|_{L^\infty(\R^d)} \mathbb{E}\int_{0}^T 
  \big|X_j^\eta(s) - \bar{X}_j^\eta(s)\big| \mathrm{1}_{\mathcal{C}_\alpha(s)^c}\dd s \\
  &\leq CN^{-\alpha-\gamma}N^{\beta(\lambda+1)} 
  + C(T)N^{-\alpha}N^{\beta(\lambda+2)}N^{-\alpha} \\
  &= C(T)\big(N^{-\alpha-\gamma+ \beta(\lambda +1)}
  + N^{-2\alpha + \beta(\lambda +2)}\big).
\end{align*}
The exponent of the first term on the right-hand side is smaller than $-1/2$ if we choose $\gamma>0$ sufficiently large. The exponent of the second term is smaller than $-1/2$ if $\alpha>1/4-\beta(\lambda+2)/2$, which is possible if $\beta<1/(2\lambda)$, since then $1/4-\beta(\lambda/2+1)< 1/2-\beta(\lambda +1)$, which is needed to apply Theorem \ref{thm.prob}. This leads to
$L_{111} \leq C(T) N^{-1/2-\eps}$ for some $\eps>0$.

We turn to the term $L_{112}$, defined in \eqref{4.L_111_and_112}. We split $\Omega$ in the set $\mathcal{B}_\theta(s)$ and its complement $\mathcal{B}_\theta(s)^c$; see definition \eqref{DthetaN} for a suitable $\theta>0$. Hence, using Lemma \ref{lem.law}, because of $\|\na V^\eta\|_{L^\infty(\R^d)}\le CN^{\beta(\lambda+1)}$ by Lemma \ref{lem.DkV},
\begin{align*}
  L_{112} &\leq C(T)N^{-\alpha-\theta} 
  + C(T)\|\nabla V^\eta\Vert_{L^\infty(\R^d)}\sup_{0<s< T} \Prob(\mathcal{B}_\theta(s)) \\
  &\leq C(T)N^{-\alpha-\theta} +C(T)N^{\beta(\lambda+1)} N^{2m\beta(\lambda +1)}N^{m(\theta - 1/2) +1}.
\end{align*}
We choose $0<\theta < 1/2-\beta(\lambda +1)$ and $m$ sufficiently large to find that
$$
  \beta(\lambda+1) + 2m\bigg(\beta(\lambda+1)+\theta-\frac12\bigg) + 1
  < -\frac12.
$$
Furthermore, we have $-\alpha-\theta<-1/2$ by choosing $\theta>1/2-\alpha$, which is possible by the choice of $\alpha$. This yields $L_{112}\le CN^{-1/2-\eps}$ for some $\eps>0$. It follows from the estimates for $L_{111}$ and $L_{112}$ that, for some $\eps>0$,
\begin{align*}
  L_{11} \leq C(T)N^{-1/2-\eps}.
\end{align*}

\textit{Step 2b: Estimate for $L_{12}$.} We exploit the positivity of the diffusion constant $\sigma >0$. For this, we split $\Omega$ again into $\mathcal{C}_\alpha(s)$ and $\mathcal{C}_\alpha(s)^c$, defined in \eqref{Calpha}, and we estimate similarly as above. Lemma \ref{lem.DkVbaru} shows that 
$\|\nabla V^\eta*\bar{u}^\eta\|_{L^\infty(\R^d)} \leq C$. Then, by the mean-value theorem for the term involving the complement $\mathcal{C}_\alpha(s)^c$,
\begin{align*} 
  L_{12} &\le C(T)\E\int_0^T\frac{1}{N}\sum_{i=1}^N 
  \bigg|(\na Z^\eta*(f^\eta-g^\eta))(s,X_i^\eta(s))
  - (\na Z^\eta*(f^\eta-g^\eta))(s,\bar{X}_i^\eta(s))\bigg|\dd s \\
  &\leq C(T) N^{-\alpha}\E\int_0^T
  \|\mathrm{D}^2 Z^\eta*(f^\eta-g^\eta)\|_{L^\infty(\R^d)}\dd s \\
  &\phantom{xx}{}+ C(T)N^{-\gamma}\sup_{\omega \in \Omega}
  \|\na Z^\eta*(f^\eta-g^\eta)\|_{L^\infty(0,T;L^\infty(\R^d))}.
\end{align*}
We use Young's convolution inequality and Lemmas \ref{lem.DkV}--\ref{lem.DkVbaru}:
\begin{align*}
  \|\mathrm{D}^2 Z^\eta*(f^\eta-g^\eta)\|_{L^\infty(\R^d)}
  &\le \|\na Z^\eta\|_{L^2(\R^d)}\|\na(f^\eta-g^\eta)\|_{L^2(\R^d)} \\
  &\le CN^{\beta\lambda/2}\|\na(f^\eta-g^\eta)\|_{L^2(\R^d)}, \\
  \|\na Z^\eta*(f^\eta-g^\eta)\|_{L^\infty(\R^d)}
  &\le C\big(\|\na V^\eta\|_{L^\infty(\R^d)}
  + \|\na V^\eta*\bar{u}^\eta\|_{L^\infty(\R^d)}\big) \\
  &\le CN^{\beta(\lambda+1)}.
\end{align*}
Then, by Young's inequality,
\begin{align*}
  L_{12} &\le C(T)N^{-\alpha+\beta \lambda/2}\mathbb{E}\int_0^T\|\na(f^\eta-g^\eta)
  \|_{L^2(\R^d)}\dd s + C(T)N^{-\gamma+\beta(\lambda+1)} \\
  &\le \frac{\sigma}{2}\mathbb{E}\int_0^T\|\na(f^\eta-g^\eta)
  \|_{L^2(\R^d)}^2\dd s
  + C(\sigma,T)N^{-2\alpha+\beta d} + C(T)N^{-\gamma+\beta(\lambda+1)}.
\end{align*}
We choose $\alpha>1/4+\beta d/2$ and $\gamma>1/2+\beta(\lambda+1)$ to conclude that, for some $\eps>0$,
$$
  L_{12} \le \frac{\sigma}{2}\mathbb{E}\int_0^T
  \|\na(f^\eta-g^\eta)\|_{L^2(\R^d)}^2\dd s
  + C(\sigma,T)N^{-1/2-\eps}.
$$
The choice of $\alpha$ is consistent with the condition in Theorem \ref{thm.prob}, since $\beta < 1/(6\lambda +4)$. We summarize the estimates for $L_{11}$ and $L_{12}$:
\begin{align}\label{4.L1}
  L_1 \le \frac{\sigma}{2}\mathbb{E}\int_0^T
  \|\na(f^\eta-g^\eta)\|_{L^2(\R^d)}^2\dd s 
  + C(\sigma,T)N^{-1/2-\eps}.
\end{align}

\textit{Step 3: Estimation of $L_2$.} It remains to estimate the term
$$
  L_2 = 2\E\bigg(\sup_{0<t<T}\bigg|\int_0^t
  \big\langle\bar{\mu}^\eta-\bar{u}^\eta,(\na V^\eta*\bar{u}^\eta)
  \cdot(\na Z^\eta*(f^\eta-g^\eta))\big\rangle\dd s\bigg|\bigg).
$$
To simplify the presentation, we abuse the notation by using an integral notation instead of the dual product in 
\begin{align*}
  \int_{\R^d} Z^\eta(x-y)\bar{\mu}^\eta(y)\na V^\eta*\bar{u}^\eta(y)\dd y
  &:= \int_{\R^d}Z^\eta(x-y)\na V^\eta*\bar{u}^\eta(y)\dd\bar{\mu}^\eta(y) \\
  &:= \frac{1}{N}\sum_{i=1}^N Z^\eta(x-\bar{X}_i^\eta)
 \na V^\eta*\bar{u}^\eta(\bar{X}_i^\eta), \quad x\in \R^d.
\end{align*}
In this way, we can easier keep track of the variables. With this notation, we can rewrite the integrand of $L_2$ by exploiting the symmetry of $Z^\eta$ (see \eqref{4.symm}):
\begin{align*}
  \langle \bar\mu^\eta - \bar{u}^\eta,
  \na Z^\eta*(f^\eta-g^\eta)\cdot\na V^\eta*\bar{u}^\eta \rangle = \langle Z^\eta*(\nabla V^\eta*\bar{u}^\eta 
  (\bar\mu^\eta - \bar{u}^\eta)), \na(f^\eta-g^\eta)\rangle,
\end{align*}
where the last expression reads as
\begin{align*}
  \langle Z^\eta&*(\nabla V^\eta*\bar{u}^\eta 
  (\bar\mu^\eta-\bar{u}^\eta)), \nabla (f^\eta-g^\eta)\rangle \\
  &= \int_{\R^d}\na(f^\eta-g^\eta)(x)\int_{\R^d} 
  Z^\eta(x-y)(\nabla V^\eta *\bar{u}^\eta)(y)
  (\bar\mu^\eta(y)-\bar{u}^\eta(y))\dd y \dd x.
\end{align*}
It follows from the Cauchy--Schwarz and Young inequalities that
\begin{align}\label{L2}
  L_2 &\leq C(\sigma) \mathbb{E}\int_{0}^T 
  \|Z^\eta*(\nabla V^\eta*\bar{u}^\eta 
  (\bar\mu^\eta-\bar{u}^\eta))\|_{L^2(\R^d)}^2 \dd s \\ 
  &\phantom{xx}
  + \frac{\sigma}{2}\mathbb{E}\int_{0}^T 
  \|\na(f^\eta-g^\eta)\|_{L^2(\R^d)}^2\dd s. \nonumber 
\end{align}
The first term in \eqref{L2} can be estimated in a similar way as in the proof of Lemma \ref{lem.law}. Indeed, we write
\begin{align*}
  &Z^\eta*(\nabla V^\eta*\bar{u}^\eta(\bar\mu^\eta-\bar{u}^\eta))(x) \\
  &\phantom{xx}= \frac{1}{N} \sum_{i = 1}^N Z^\eta(x-\bar{X}_i^\eta)
  (\nabla V^\eta*\bar{u}^\eta)(\bar{X}_i^\eta) 
  - \int_{\R^d}Z^\eta(x-y)(\nabla V^\eta*\bar{u}^\eta)(y)
  \bar{u}^\eta(y) \dd y
\end{align*}
and introduce for $x\in \R^d$ and $i=1,\ldots,N$ the functions
\begin{align*}
  h_i(x):= Z^\eta(x-\bar{X}_i^\eta)\nabla V^\eta*\bar{u}^\eta
  (\bar{X}_i^\eta) - \int_{\R^d}Z^\eta(x-y)
  \nabla V^\eta*\bar{u}^\eta(y)\bar{u}^\eta(y)\dd y.
\end{align*}
For fixed $x\in \R^d$, the functions $h_i(x)$ are independent (since starting from i.i.d.\ initial conditions) with vanishing mean.
Hence, by Fubini's theorem,
\begin{align}\label{4.EZV}
  \mathbb{E}\int_{0}^T \|Z^\eta*(\nabla V^\eta*\bar{u}^\eta 
  (\bar\mu^\eta- \bar{u}^\eta))\|_{L^2(\R^d)}^2 \dd s 
  = \int_{0}^T\int_{\R^d}\mathbb{E}\bigg|\frac{1}{N}\sum_{i = 1}^N 
  h_i(x)\bigg|^2\dd x \dd s.
\end{align}
Due to the independence of $h_i(x)$, which means that $\mathbb{\E} h_i(x)h_j(x)=0$ for $i\neq j$, we see that 
\begin{align*}
  \mathbb{E}\bigg|\frac{1}{N}\sum_{i = 1}^N h_i(x)\bigg|^2  
  &= \frac{1}{N^2} \sum_{i = 1}^N\mathbb{E}|h_i(x)|^2 
  = \frac{1}{N}\int_{\R^d}Z^\eta(x-z)^2|\nabla V^\eta*\bar{u}^\eta(z)|^2
  \bar{u}^\eta(z)\dd z \\
  &\phantom{xx}+ \frac{2}{N}\bigg|Z^\eta*\big((\nabla V^\eta*\bar{u}^\eta)\bar{u}^\eta\big)(x) 
  \int_{\R^d} Z^\eta(x-z)(\nabla V^\eta*\bar{u}^\eta)(z)\bar{u}^\eta(z) 
  \dd z\bigg| \\
  &\phantom{xx}+ \frac{1}{N}|Z^\eta*\big((\nabla V^\eta*\bar{u}^\eta)\bar{u}^\eta\big)(x) |^2.
\end{align*}
We insert this expression into \eqref{4.EZV} and apply Young's convolution inequality: 
\begin{align*}
  \mathbb{E}\int_{0}^T&\|Z^\eta*(\nabla V^\eta*\bar{u}^\eta
  (\bar\mu^\eta-\bar{u}^\eta))\|_{L^2(\R^d)}^2 \dd s \\
  &\leq \frac{C(T)}{N}\sup_{0<t<T} \|\nabla V^\eta*\bar{u}^\eta\|_{L^\infty(\R^d)}^2 
  \|(Z^\eta)^2*\bar{u}^\eta  \|_{L^1(\R^d)} \\
  &\phantom{xx}+ \frac{C(T)}{N}\sup_{0<t<T}
  \|Z^\eta*\big(\bar{u}^\eta
  \nabla V^\eta*\bar{u}^\eta\big)\|_{L^2(\R^d)}^2 \\
  &\leq \frac{C(T)}{N}\sup_{0<t<T}\|\nabla V^\eta*\bar{u}^\eta
  \|_{L^\infty(\R^d)}^2 \\
  &\phantom{xx}\times\big(\|Z^\eta\|_{L^{2}(\R^d)}^2 
  \|\bar{u}^\eta\|_{L^1(\R^d)} 
  + \|Z^\eta\|_{L^2(\R^d)}^2\|\bar{u}^\eta\|_{L^1(\R^d)}^2\big).
\end{align*}
Lemma \ref{lem.DkZ} shows that $\|Z^\eta\|_{L^2(\R^d)}^2\leq N^{\beta \lambda}$, while Lemma \ref{lem.DkVbaru} yields the estimate $\|\nabla V^\eta*\bar{u}^\eta\|_{L^\infty(\R^d)} \leq C$. Hence, together with mass conservation for $\bar{u}^\eta$, we infer that
\begin{align*}
  \mathbb{E}\int_{0}^T &\|Z^\eta*(\na V^\eta*\bar{u}^\eta
  (\bar\mu^\eta- \bar{u}^\eta))\|_{L^2(\R^d)}^2 \dd s \\
  &\leq C(T)N^{-1+\beta \lambda}
  \leq C(T)N^{-1/2-\eps},
\end{align*}
choosing $\beta < 1/(2\lambda)$. Consequently, we have
\begin{align*}
  L_2 \leq \frac{\sigma}{2}\mathbb{E}\int_{0}^T
  \|\na(f^\eta-g^\eta)(s)\|_{L^2(\R^d)}^2 
  \dd s + C(\sigma,T)N^{-1/2-\eps}.
\end{align*}
Adding this inequality to estimate \eqref{4.L1} for $L_1$ gives
\begin{equation}\label{4.LL}
  L(T) \leq L_1 + L_2 \le \sigma\E\int_{0}^T\|\na(f^\eta-g^\eta)\|_{L^2(\R^d)}^2\dd s
  + C(\sigma,T)N^{-1/2-\eps}.
\end{equation}

\begin{remark}[Asymptotically vanishing diffusion]\label{rem.sigma}\rm
The condition $\sigma>0$ in Assumption (A1) is required to estimate the terms $L_{12}$ and $L_2$. We claim that we may replace $\sigma$ by $\sigma_N=N^{-s}$ for some $s>0$. Indeed, the estimate for $L_{12}$ becomes
$$
  L_{12} \le \frac{\sigma_N}{2}\mathbb{E}\int_0^T
  \|\na(f^\eta-g^\eta)\|_{L^2(\R^d)}^2\dd s
  + \frac{C(T)}{\sigma_N}\big(N^{-2\alpha+\beta d}
  + N^{-\gamma+\beta(\lambda+1)}\big).
$$
The first term on the right-hand side can be absorbed by the left-hand side of \eqref{4.KLM} with $\sigma_N$ instead of $\sigma$. The second term is bounded by $C(T)N^{-1/2-\eps}$ if 
$-2\alpha + \beta d + s < -1/2$ and $-\gamma+\beta(\lambda+1) + s 
< -1/2$, which is fulfilled if $\alpha>1/4+\beta d/2+s/2$ and $\gamma>1/2+\beta(\lambda+1)+s$. While $\gamma$ can be chosen arbitrarily large, the upper bound $\alpha < 1/2- \beta(\lambda +1)$ can be fulfilled for sufficiently small $s$ (with an upper bound depending on $\beta$). The estimate for $L_2$ reads as
$$
  L_2 \le \frac{\sigma_N}{2}\mathbb{E}\int_0^T
  \|\na(f^\eta-g^\eta)(s)\|_{L^2(\R^d)}^2\dd s
  + \frac{C}{\sigma_N}N^{-1+\beta \lambda},
$$
and the last term is of order $N^{-1/2-\eps}$ for some $\eps>0$ if
$-1+\beta\lambda+s<-1/2$. Again, this yields a restriction on $s>0$.
\qed\end{remark}

The following remark is used in Section \ref{sec.fluct} to characterize the limiting intermediate fluctuations process.

\begin{remark}[Generalized estimate for $L(T)$]\label{rem.F2}\rm
We claim that similar estimations as above show that there exists $C>0$ such that for any $F\in W^{2,\infty}(\R^d)$,
\begin{align*}
  L^*(T) &:= \E\bigg(\sup_{0<t<T}\bigg|\int_0^t
  \big\langle\mu^\eta-\bar{u}^\eta,
  (\na V^\eta*(\mu^\eta-\bar{u}^\eta))\cdot\na F\big\rangle\dd s\bigg| 
  \bigg) \\
  &\le C(\sigma,T)N^{-1/2-\eps}\big(\|\na F\|_{L^{\infty}(\R^d)} 
  + \|\na F\|_{L^{\infty}(\R^d)}^{2}
  + \|\mathrm{D}^2F\|_{L^\infty(\R^d)}\big) \\
  &\phantom{xx}
  + \sigma\E\int_0^T\|\na(f^\eta-g^\eta)\|_{L^2(\R^d)}^2\dd s. 
\end{align*}
Our final estimate \eqref{4.LL} for $L(T)$ shows this result with $\na F=\na V^\eta*\bar{u}^\eta$. We need to verify that the right-hand side depends on the $L^\infty(\R^d)$ norms of $\na F$ and $\mathrm{D}^2F$ as stated above. As in the estimate for $L(T)$, we add and subtract $\bar{u}^\eta$ such that $|L^*(T)|\le L^*_1+L_2^*$, where
\begin{align*}
  L^*_1 &= \E\bigg(\sup_{0<t<T}\bigg|\int_0^t\big\langle
  \mu^\eta-\bar{\mu}^\eta,(\na Z^\eta*(f^\eta-g^\eta))\cdot\na F
  \big\rangle\dd s\bigg|\bigg), \\
  L^*_2 &= \E\bigg(\sup_{0<t<T}\bigg|\int_0^t\big\langle
  \bar{\mu}^\eta-\bar{u}^\eta,(\na Z^\eta*(f^\eta-g^\eta))\cdot\na F
  \big\rangle\dd s\bigg|\bigg).
\end{align*}
We add and subtract $\na F(X_i^\eta)(\na Z^\eta*(f^\eta-g^\eta))(\bar{X}_i^\eta)$ in $L_1^*$, leading to $L_1^*\le L^*_{11}+L_{12}$, where
\begin{align*}
  L_{11}^* &= \E\int_0^T\bigg|\frac{1}{N}\sum_{i=1}^N
  (\na F(X_i^\eta) - \na F(\bar{X}_i^\eta))(\na Z^\eta*(f^\eta-g^\eta))
  (\bar{X}_i^\eta)\bigg|\dd s, \\
  L_{12}^* &= \E\int_0^T\bigg|\frac{1}{N}\sum_{i=1}^N
  \na F(X_i^\eta)\big((\na Z^\eta*(f^\eta-g^\eta)(X_i^\eta)
  - (\na Z^\eta*(f^\eta-g^\eta))(\bar{X}_i^\eta)\big)\bigg|\dd s.
\end{align*}
Step 2a (using the mean-valued theorem, which yields $\mathrm{D}^2 F$) shows that
$$
  L_{11}^* \le CN^{-1/2-\eps}\|\mathrm{D}^2F\|_{L^\infty(\R^d)}.
$$
Furthermore, we infer from Step 2b that 
$$
  L_{12}^*\le C(\sigma,T)N^{-1/2-\eps}\|\na F\|_{L^\infty(\R^d)}
  +  \frac{\sigma}{2}\int_0^T
  \|\na(f^\eta-g^\eta)\|_{L^2(\R^d)}^2\dd s.
$$
The term $L_2^*$ is estimated according to 
\begin{align*}
  L_2^* \le C(\sigma)\E\int_0^T\|Z^\eta*(\na F
  (\bar{\mu}^\eta-\bar{u}^\eta))\|_{L^2(\R^d)}^2\dd s 
  + \frac{\sigma}{2}\E\int_0^T\|\na(f^\eta-g^\eta)\|_{L^2(\R^d)}^2\dd s.
\end{align*}
In Step 3, we have shown that the first term is bounded by
$$
  \E\int_0^T\|Z^\eta*(\na F(\bar{\mu}^\eta-\bar{u}^\eta))
  \|_{L^2(\R^d)}^2\dd s \le C(T)N^{-1/2-\eps}\|\na F\|_{L^\infty(\R^d)}^{2}.
$$
Summarizing the previous inequalities, we have proved the claim.
\qed\end{remark}

%%%%%%%%%%%%%

\subsection{End of the proof}

We insert estimates \eqref{4.K1} for $K_1$, \eqref{4.K6} for $K_6$, \eqref{4.MM} for $M(T)$, and \eqref{4.LL} for $L(T)$ into \eqref{4.KLM} to find that 
\begin{align*}
  \E\Big(&\sup_{0<t<T}\|(f^\eta-g^\eta)(t)\|_{L^2(\R^d)}^2\Big)
  + \sigma\E\int_0^T\|\na(f^\eta-g^\eta)(s)\|_{L^2(\R^d)}^2\dd s \\
  &\le \E\|(f^\eta-g^\eta)(0)\|_{L^2(\R^d)}^2
  + C(\sigma,T)N^{-1/2-\eps}.
\end{align*}
It remains to estimate the first term on the right-hand side. We write
\begin{align*}
  \|(f^\eta&-g^\eta)(0)\|_{L^2(\R^d)}^2
  = \|f^\eta(0)\|_{L^2(\R^d)}^2 - 2\langle f^\eta(0),g^\eta(0)\rangle  
  + \|g^\eta(0)\|_{L^2(\R^d)}^2 \\
  &= \frac{1}{N^2}\sum_{i,j=1}^N V^\eta(X_i^\eta(0)-X_j^\eta(0)) 
  - \frac{2}{N}\sum_{i=1}^N(V^\eta* u_0)(X_i^\eta(0))
  + \langle u_0, V^\eta*u_0\rangle.
\end{align*}
Since $X_i^\eta(0)=\zeta_i$ and $\zeta_1,\ldots,\zeta_N$ are independent with common density function $u_0$, we find from Young's convolution inequality that
\begin{align*}
  \E\|(f^\eta&-g^\eta)(0)\|_{L^2(\R^d)}^2
  = \frac{1}{N^2}\sum_{i,j=1,\,i\neq j}^N\int_{\R^d}\int_{\R^d}
  V^\eta(x-y)u_0(x)u_0(y)\dd x\dd y \\
  &\phantom{xx}+ \frac{V^\eta(0)}{N} - \frac{2}{N}\sum_{i=1}^N\int_{\R^d}
  (V^\eta*u_0)(x)u_0(x)\dd x
  + \int_{\R^d}(V^\eta*u_0)(x)u_0(x)\dd x \\
  &= \frac{N(N-1)}{N^2}\int_{\R^d}(V^\eta*u_0)(x)u_0(x)\dd x
  + \frac{V^\eta(0)}{N} 
  - \int_{\R^d}(V^\eta*u_0)(x)u_0(x)\dd x \\
  &= -\frac{1}{N}\int_{\R^d}(V^\eta*u_0)(x)u_0(x)\dd x
  + \frac{V^\eta(0)}{N} \\
  &\le N^{-1}\|V^\eta*u_0\|_{L^\infty(\R^d)}\|u_0\|_{L^1(\R^d)}
  + N^{-1}\|V^\eta\|_{L^\infty(\R^d)} \\
  &\le CN^{-1} + CN^{\beta\lambda-1} \le CN^{-1/2-\eps},
\end{align*}
if $\beta<1/(2\lambda)$, where we used Lemma \ref{lem.DkV} to estimate $\|V^\eta\|_{L^\infty(\R^d)}\le CN^{\beta\lambda}$ and the property $\|V^\eta*u_0\|_{L^\infty(\R^d)}\le C$ from Lemma \ref{lem.DkVbaru}, since $u_0 \in L^1(\R^d)\cap L^\infty(\R^d)$. This finishes the proof of Theorem \ref{thm.L2}.

\begin{remark}[Convergence of $\mu^\eta-u$]\label{rem.diff}\rm
We claim that 
\begin{align}\label{4.remdiff}
  \lim_{N\to \infty}\E\Big(\sup_{0<t<T}\sup_{\phi\in\mathcal{W}}
  \big|\langle\mu^\eta(t)-u(t),\phi\rangle\big|\Big) = 0,
\end{align}
where $\mathcal{W}=\{f\in W^{1,\infty}(\R^d)\cap L^2(\R^d)\cap L^1(\R^d):
\|f\|_{W^{1,\infty}(\R^d)}+\|f\|_{L^2(\R^d)} + \|f \|_{L^1(\R^d)}=1\}$. We estimate
\begin{align*}
  \langle\mu^\eta-u,\phi\rangle
  &= \langle\mu^\eta-\bar{u}^\eta,Z^\eta*\phi\rangle
  + \langle\bar{u}^\eta-u,Z^\eta*\phi\rangle
  + \langle\mu^\eta-u,\phi-Z^\eta*\phi\rangle \\
  &\le \|(\mu^\eta-\bar{u}^\eta)*Z^\eta\|_{L^2(\R^d)}\|\phi\|_{L^2(\R^d)}
  + \|\bar{u}^\eta*Z^\eta-u*Z^\eta\|_{L^2(\R^d)}\|\phi\|_{L^2(\R^d)} \\
  &\phantom{xx}+ \langle\mu^\eta-u,\phi-Z^\eta*\phi\rangle \\
  &\le \|f^\eta-g^\eta\|_{L^2(\R^d)} 
  + \|Z^\eta \ast(\bar{u}^\eta - u)\|_{L^2(\R^d)} + \|\phi-Z^\eta*\phi\|_{L^\infty(\R^d)}.
\end{align*}
Theorem \ref{thm.L2} implies that the first term on the right-hand side converges to zero as $\eta\to 0$ uniformly in $(0,T)$ for any $T>0$. It can be shown, by using Lemma \ref{lem.wLinfty}, that also the second term tends to zero, since $\|Z^\eta*(\bar{u}^\eta - u)\|_{L^2(\R^d)} \leq C \|\bar{u}^\eta - u\|_{L^1(\R^d)\cap L^\infty(\R^d)}$. Finally, 
$$
  \|\phi-Z^\eta*\phi\|_{L^\infty(\R^d)}
  \le C\eta \|\phi \|_{L^1(\R^d)\cap L^\infty(\R^d)} \to 0
$$ 
shows that the third term converges too. 
\qed\end{remark}

%%%%%%%%%%%%%%%%%%%%%%%%%%%%%%%%%%%%%%%%%%%%%%%%%%%%%%%%%%%%%%%%%%%%

\section{Fluctuations}\label{sec.fluct}

In this section, we prove Theorem \ref{thm.fluct}. We define the intermediate fluctuations
$$
  \mathcal{F}^\eta(t) = \sqrt{N}(\mu^\eta-\bar{u}^\eta)(t),
$$
which can be written in weak form for test functions $\psi$ as 
\begin{align*}
  \langle\mathcal{F}^\eta(t),\psi\rangle  
  = \frac{1}{\sqrt{N}}\sum_{i=1}^N\psi(t,X_i^\eta(t))
  - \sqrt{N}\langle\bar{u}^\eta(t),\psi\rangle.
\end{align*}
We derive first the stochastic differential equation satisfied by the intermediate fluctuations. The definition of the dual operator $\mathcal{L}^*$, applied to test functions $\psi$, is as follows (see Appendix \ref{sec.dual} for details)
\begin{align*}
  \mathcal{L}^*\psi = -\pa_{t}\psi - \sigma\Delta\psi 
  - \kappa\big((\na\Phi*u)\cdot\na\psi - \na\Phi*(u\na\psi)\big).
\end{align*} 
In the following lemma, we give a connection between the SDE fulfilled by the intermediate fluctuations process and the dual operator $\mathcal{L}^*$. This will be crucial for the proof of the main theorem (Theorem \ref{thm.fluct}).

\begin{lemma}\label{lem.dF}
Let $\psi\in C^\infty(\R^d\times(0,\infty))$ be a test function. Then for any $t>0$, 
\begin{align*}
  &\dd\langle\mathcal{F}^\eta,\psi\rangle
  = -\langle\mathcal{F}^\eta,\mathcal{L}^*\psi\rangle\dd t
  + \frac{\sqrt{2\sigma}}{\sqrt{N}}\sum_{i=1}^N\na\psi(X_i^\eta(t))
  \cdot\dd W_i(t) \\
  &\phantom{x}+ \kappa\bigg(\frac{1}{\sqrt{N}}\langle\mathcal{F}^\eta,
  (\na V^\eta*\mathcal{F}^\eta)\cdot\na\psi\rangle
  + \big\langle\mathcal{F}^\eta,(\na V^\eta*\bar{u}^\eta -  \na\Phi*u)
  \cdot\na\psi \big\rangle \\
  &\phantom{xxxxxx}- \big\langle\mathcal{F}^\eta,
 \na V^\eta*(\bar{u}^\eta\na\psi) - \na\Phi*(u\na\psi)\big\rangle\bigg)\dd t.
\end{align*}
\end{lemma}

\begin{proof}
We wish to calculate
\begin{align}\label{8.dF}
  \dd\langle\mathcal{F}^\eta,\psi\rangle
  = \sqrt{N}\dd\langle\mu^\eta-\bar{u}^\eta,\psi\rangle
  = \frac{1}{\sqrt{N}}\sum_{i=1}^N\dd\psi(\cdot,X_i^\eta)
  - \sqrt{N}\dd\langle\bar{u}^\eta,\psi\rangle.
\end{align}
We infer from It\^{o}'s lemma that
\begin{align*}
  \dd\psi(t,X_i^\eta(t)) &= \big[\pa_t\psi(t,X_i^\eta(t))
  + \kappa\big(\mu^\eta(t)*\na V^\eta (X_i^\eta(t))\big)\cdot\na\psi(t,X_i^\eta(t)) \\
  &\phantom{xx}+ \sigma\Delta\psi(t,X_i^\eta(t))\big]\dd t
  + \sqrt{2\sigma}\na\psi(t,X_i^\eta(t))\cdot\dd W_i(t).
\end{align*}
Summing this equation over $i=1,\ldots,N$ and using $\sqrt{N}\langle\mu^\eta,\psi\rangle
= N^{-1/2}\sum_{i=1}^N\psi(X_i^\eta)$, we find that
\begin{align}\label{8.aux1}
  \frac{1}{\sqrt{N}}\sum_{i=1}^N\dd\psi(t,X_i^\eta(t))
  &= \sqrt{N}\langle\mu^\eta(t),\pa_t\psi(t)\rangle
  + \sqrt{N}\kappa\langle\mu^\eta(t),
  (\mu^\eta(t)*\na V^\eta)\cdot\na\psi(t)\rangle \\
  &\phantom{xx}+ \sqrt{N}\sigma\langle\mu^\eta,\Delta\psi(t)\rangle
  + \frac{\sqrt{2\sigma}}{\sqrt{N}}\sum_{i=1}^N\na\psi(t,X_i^\eta(t))
  \cdot\dd W_i(t). \nonumber
\end{align}
The intermediate equation \eqref{1.baru} is written in weak form as
\begin{align*}
  0 &= -\int_0^t\langle\bar{u}^\eta(s),\pa_t\psi(s)\rangle\dd s 
  + \langle\bar{u}^\eta(t),\psi(t)\rangle 
  - \langle\bar{u}^\eta(0),\psi(0)\rangle \\
  &\phantom{xx}
  - \sigma\int_0^t\langle\bar{u}^\eta(s),\Delta\psi(s)\rangle\dd s 
  - \kappa\int_0^t\langle\bar{u}^\eta\na V^\eta*\bar{u}^\eta(s),
  \na\psi(s)\rangle\dd s,
\end{align*}
or in differential form, multiplied by $\sqrt{N}$,
\begin{align}\label{8.aux2}
  \sqrt{N}\dd\langle\bar{u}^\eta,\psi\rangle
  = \sqrt{N}\big(\langle\bar{u}^\eta,\pa_t\psi+\sigma\Delta\psi\rangle
  + \kappa\langle\bar{u}^\eta\na V^\eta*\bar{u}^\eta,\na\psi\rangle
  \big)\dd t.
\end{align}
Inserting \eqref{8.aux1} and \eqref{8.aux2} into \eqref{8.dF} yields
\begin{align}\label{8.aux3}
  \dd\langle\mathcal{F}^\eta,\psi\rangle
  &= \langle\mathcal{F}^\eta,\pa_t\psi+\sigma\Delta\psi\rangle
  + \frac{\sqrt{2\sigma}}{\sqrt{N}}\sum_{i=1}^N
  \na\psi(\cdot,X_i^\eta)\cdot\dd W_i \\
  &\phantom{xx}+ \sqrt{N}\kappa\langle\mu^\eta,
  (\mu^\eta*\na V^\eta)\cdot\na\psi\rangle
  - \sqrt{N}\kappa\langle\bar{u}^\eta,
  (\bar{u}^\eta*\na V^\eta)\cdot\na\psi\rangle. \nonumber
\end{align}
We wish to formulate the last two expressions in terms of $\mathcal{F}^\eta$. For this, we compute, abusing slightly the notation,
\begin{align*}
  &\sqrt{N}\langle\mu^\eta,
  (\mu^\eta*\na V^\eta)\cdot\na\psi\rangle
  - \sqrt{N}\langle\bar{u}^\eta,
  (\bar{u}^\eta*\na V^\eta)\cdot\na\psi\rangle \\
  &\phantom{x}= \sqrt{N}\int_{\R^d}\int_{\R^d}
  \na V^\eta(x-y)\cdot\na\psi(x)\dd\mu^\eta(y)\dd\mu^\eta(x) \\
  &\phantom{xxx}
  - \sqrt{N}\int_{\R^d}\int_{\R^d}\na V^\eta(x-y)\cdot\na\psi(x)
  \bar{u}^\eta(x)\bar{u}^\eta(y)\dd y\dd x \\
  &\phantom{x}
  = \sqrt{N}\int_{\R^d}\int_{\R^d}\na V^\eta(x-y)\cdot\na\psi(x)
  \dd(\mu^\eta(x)-\bar{u}^\eta(x))\dd(\mu^\eta(y)-\bar{u}^\eta(y)) \\
  &\phantom{xxx}
  + \sqrt{N}\int_{\R^d}\int_{\R^d}\na V^\eta(x-y)\cdot\na\psi(x)
  \dd(\mu^\eta(x)-\bar{u}^\eta(x))\dd\bar{u}^\eta(y) \\
  &\phantom{xxx}
  + \sqrt{N}\int_{\R^d}\int_{\R^d}\na V^\eta(x-y)\cdot\na\psi(x)
  \dd\bar{u}^\eta(x)\dd(\mu^\eta(y)-\bar{u}^\eta(y)) \\
  &\phantom{x}= N^{-1/2}\langle\mathcal{F}^\eta,
  (\mathcal{F}^\eta*\na V^\eta)\cdot\na\psi\rangle
  + \langle\mathcal{F}^\eta,(\bar{u}^\eta*\na V^\eta)\cdot\na\psi\rangle
  + \langle\bar{u}^\eta,(\mathcal{F}^\eta*\na V^\eta)\cdot\na\psi\rangle
  \\
  &\phantom{x}
  = N^{-1/2}\langle\mathcal{F}^\eta,(\mathcal{F}^\eta*\na V^\eta)
  \cdot\na\psi\rangle
  + \langle\mathcal{F}^\eta,(\bar{u}^\eta*\na V^\eta)\cdot\na\psi\rangle
  - \langle\mathcal{F}^\eta,\na V^\eta*(\bar{u}^\eta\na\psi)\rangle,
\end{align*}
where the last step follows from the anti-symmetry of $\na V^\eta$ (see \eqref{4.symm}). We add and subtract the dual operator $\mathcal{L}^*$ to rewrite the last two terms (multiplied by $\kappa$):
\begin{align*}
  \kappa\langle\mathcal{F}^\eta&,(\bar{u}^\eta*\na V^\eta)\cdot\na\psi\rangle
  - \kappa\langle\mathcal{F}^\eta,
  \na V^\eta*(\bar{u}^\eta\na\psi)\rangle \\
  &= -\langle\mathcal{F}^\eta,\mathcal{L}^*\psi\rangle
  + \kappa\big\langle\mathcal{F}^\eta,(\na V^\eta*\bar{u}^\eta)\cdot\na\psi
  - \na V^\eta*(\bar{u}^\eta\na\psi)\big\rangle \\
  &\phantom{xx}- \kappa\big\langle\mathcal{F}^\eta,(\na\Phi*u)
  \cdot\na\psi- \na\Phi*(u\na\psi)\big\rangle 
  - \big\langle\mathcal{F}^\eta,\pa_t \psi + \sigma \Delta \psi\big\rangle.
\end{align*}
Inserting these expressions into \eqref{8.aux3} shows the statement of the lemma.
\end{proof}

Let $\psi=T_\phi^t$, where $T_\phi^t$ is the solution to the backward dual problem $\mathcal{L}^*\psi=0$ with initial datum $\phi$, which is a sufficiently regular test function on $\R^d$. (Later, we will take $\phi$ to be a Schwartz function in the proof of Theorem \ref{thm.fluct}.) Then the first term on the right-hand side of the expression in Lemma \ref{lem.dF} vanishes such that, in integrated form and for $0<s<t$,
\begin{align}\label{8.F123}
  &\langle\mathcal{F}^\eta(t),\phi\rangle 
  - \langle\mathcal{F}^\eta(s),T_\phi^t(s)\rangle
  = F_1 + F_2 + F_3, \quad\mbox{where } \\
  & F_1 = \frac{\sqrt{2\sigma}}{\sqrt{N}}\sum_{i=1}^N\int_s^t
  \na T_\phi^t(r,X_i^\eta(r))\cdot\dd W_i(r), \nonumber \\
  & F_2 = \frac{\kappa}{\sqrt{N}}\int_s^t\langle\mathcal{F}^\eta(r),
  (\na V^\eta*\mathcal{F}^\eta(r))\cdot\na T_\phi^t(r)\rangle\dd r, 
  \nonumber \\
  & F_3 = \kappa\int_s^t\big\langle\mathcal{F}^\eta(r),
  (\na V^\eta*\bar{u}^\eta(r))\cdot\na T_\phi^t(r)
  - \na V^\eta*(\bar{u}^\eta\na T_\phi^t)(r)\big\rangle\dd r 
  \nonumber \\
  &\phantom{F_3=} - \kappa\int_s^t\big\langle\mathcal{F}^\eta(r),
  (\na\Phi*u)\cdot\na T_\phi^t(r) - \na\Phi*(u\na T_\phi^t)(r)
  \big\rangle\dd r. \nonumber
\end{align}
The term $F_1$ describes the Gaussian fluctuations in the limit, while we claim that the terms $|F_2|$ and $|F_3|$ vanish (in expectation) in the limit.

\begin{lemma}\label{lem.F23}
Let $0<s<t$ and $\eta=N^{-\beta}$ with $0<\beta< 1/(8\lambda +12)$ and let all assumptions of Theorem \ref{thm.fluct} hold. Then
$\E(|F_2|+|F_3|)\to 0$ as $N\to\infty$ (or equivalently $\eta\to 0$).
\end{lemma}

\begin{proof}
The proof is inspired by Oelschl\"ager \cite{Oel87} who proved a related result for the viscous porous medium equation. We note that $F_2$ almost equals $L(T)$, defined in \eqref{4.L}, except for the different scaling factor $\sqrt{N}$, $\na V^\eta$ instead of $\na Z^\eta$, and $\na T_\phi^t$ instead of $\na V^\eta*\bar{u}^\eta$. Then the estimate for $L^*(T)$ from Remark \ref{rem.diff} leads to
$$
  \E|F_2| \le CN^{-\eps} + CN^{1/2}\E\int_0^T
  \|\na(f^\eta-g^\eta)\|_{L^2(\R^d)}^2\dd s\to 0,
$$
as $N \to \infty$, by Theorem \ref{thm.L2}. We estimate $F_3$ by an iteration argument. Setting
\begin{align}\label{8.defpsi1}
  \psi^1_{\eta,t,\phi}(r) 
  &:= (\na V^\eta*\bar{u}^\eta(r))\cdot\na T_\phi^t(r)
  - \na V^\eta*(\bar{u}^\eta\na T_\phi^t)(r) \\
  &\phantom{xx}- (\na\Phi*u)\cdot\na T_\phi^t(r) 
  + \na\Phi*(u\na T_\phi^t)(r), \nonumber
\end{align}
we can write $F_3^1:=F_3$ more compactly as
\begin{align*}%\label{8.F_3_1}
  F_3^1 = \kappa\int_s^t\langle\mathcal{F}^\eta(r),
  \psi^1_{\eta,t,\phi}(r)\rangle\dd r.
\end{align*}
(The upper index in $F_3^1$ will be later used in the iteration argument.) We decompose $\psi_{\eta,t,\phi}^1$ as
\begin{align*}
  & \sup_{0<t<T}\sup_{0<r<t}\|\psi_{\eta,t,\phi}^1(r)\|_{L^\infty(\R^d)}
  \le F_{31} + F_{32} + F_{33} + F_{34}, \quad\mbox{where} \\
  & F_{31} = \sup_{0<t<T}\sup_{0<r<t}\|((\na V^\eta - \na\Phi)
  *\bar{u}^\eta)\cdot\na T_\phi^t(r)\|_{L^\infty(\R^d)}, \\
  & F_{32} = \sup_{0<t<T}\sup_{0<r<t}\|\na\Phi*(\bar{u}^\eta-u)
  \cdot\na T_\phi^t(r)\|_{L^\infty(\R^d)}, \\
  & F_{33} = \sup_{0<t<T}\sup_{0<r<t}\|\na\Phi*((u-\bar{u}^\eta)
  \na T_\phi^t)(r)\|_{L^\infty(\R^d)}, \\
  & F_{34} = \sup_{0<t<T}\sup_{0<r<t}\|(\na\Phi-\na V^\eta)
  *(\bar{u}^\eta\na T_\phi^t)(r)\|_{L^\infty(\R^d)}.
\end{align*}
It follows from  the mean-value theorem and $\int_{\R^d}|x| \chi^\eta(x) \dd x \leq C\eta$ that  $\|\chi^\eta * w - w\|_{L^\infty(\R^d)} \leq C\eta \|\nabla w \|_{L^\infty(\R^d)}$. Hence,
\begin{align*}
  \|(\na V^\eta - \na\Phi)*\bar{u}^\eta \|_{L^\infty(\R^d)} 
  \leq C\eta \|\mathrm{D^2}\Phi*\bar{u}^\eta \|_{L^\infty(\R^d)},
\end{align*}
where the latter norm is bounded, since $\lambda < d-2$ and
\begin{align*}
  |\mathrm{D^2}\Phi*\bar{u}^\eta(x)|
  &\leq \int_{B_1(0)}|\mathrm{D^2} \Phi|*\bar{u}^\eta \dd x 
  + \int_{B_1(0)^c}|\mathrm{D^2} \Phi| * \bar{u}^\eta \dd x \\
  &\leq C \|\bar{u}^\eta \|_{L^1(\R^d)\cap L^\infty(\R^d)} 
  \leq C(T),
\end{align*}
by Lemma \ref{lem.wLinfty}. We infer from $\sup_{0<t<T}\sup_{0<r<t}\|\nabla T_\phi^t (r)\|_{L^\infty(\R^d)}\leq C\|\phi\|_{W^{1,\infty}(\R^d)}$ by Lemma \ref{lem.estT} that
\begin{align*}
  F_{31} \leq C(T)\eta.
\end{align*}
A similar argument leads to
\begin{align*}
  F_{34}&\le C(T)\eta \sup_{0<t<T}\sup_{0<r<t}
  \|\bar{u}^\eta\na T_\phi^t(r)\|_{L^1(\R^d)\cap L^\infty(\R^d)} \\
  &\le C(T)\eta\|\bar{u}^\eta \|_{L^1(\R^d)\cap L^\infty(\R^d)}
  \|\phi\|_{W^{1,\infty}(\R^d)},
\end{align*}
where the last estimate follows from Lemmas \ref{lem.Linfty} and \ref{lem.estT}. We deduce from a similar estimate as in the proof of Lemma \ref{lem.DkVbaru} that
\begin{align*}
  F_{32} &\le \sup_{0<t<T}\sup_{0<r<t}
  \|\na\Phi*(\bar{u}^\eta-u)(r)\|_{L^\infty(\R^d)}
  \|\na T_\phi^t(r)\|_{L^\infty(\R^d)} \\
  &\le C(T)\|\phi \|_{W^{1,\infty}(\R^d)}
  \sup_{0<r<T}\|(\bar{u}^\eta-u)(r)\|_{L^1(\R^d)\cap L^\infty(\R^d)} 
  \le C(T)\eta,
\end{align*}
where the last step follows from Lemmas \ref{lem.wLinfty} and \ref{lem.estT}. Similar arguments show that
\begin{align*}
  F_{33} &\le \sup_{0<r<T}\|(u-\bar{u}^\eta)\na T_\phi^t(r)
  \|_{L^1(\R^d)\cap L^\infty(\R^d)} 
  \le C(T)\eta \|\phi\|_{W^{1,\infty}(\R^d)}.
\end{align*}
Consequently, 
$$
  \sup_{0<t<T}\sup_{0<r<t}\|\psi_{\eta,t,\phi}^1(r)\|_{L^\infty(\R^d)}
  \le C(T)\eta\|\phi\|_{W^{1,\infty}(\R^d)}.
$$
A similar estimate can be derived for the derivatives of $\psi_{\eta,t,\phi}^1$ by putting them on $\bar{u}^\eta$, $u$, and $T_\phi^t$:
\begin{equation}\label{8.psi1}
  \sup_{0<t<T}\sup_{0<r<t}\|\mathrm{D}^k
  \psi_{\eta,t,\phi}^1(r)\|_{L^\infty(\R^d)}
  \le C(T)\eta \|\phi\|_{W^{k+1,\infty}(\R^d)} 
  \quad\mbox{for all }k\le K^*,
\end{equation}
where $K^*\in\N$ will be determined at the end of the proof. Here, we have used estimates for $\bar{u}^\eta$, $u$ in $W^{K^*,\infty}(\R^d)\cap W^{K^*,1}(\R^d)$ and for $T_\phi^t$ in the space $W^{K^*+1,\infty}(\R^d)\cap W^{K^*+1,1}(\R^d)$; see Lemma \ref{lem.higher} and Lemma \ref{lem.estT}. These estimates need the higher regularity Assumption (A4b). 

Notice that a naive estimation now would lead to 
\begin{align*}
  |F_3^1| &= \bigg| \int_s^t\langle\mathcal{F}^\eta(r),
    \psi^1_{\eta,t,\phi}(r)\rangle\dd r \bigg| 
  \le \int_{0}^t|\langle\mathcal{F}^\eta(r),
  \psi^1_{\eta,t,\phi}(r)\rangle|\dd r \\
  &\le C(T)\sqrt{N}\sup_{0<t<T}\sup_{0<r<t}
  \|\psi_{\eta,t,\phi}^1(r)\|_{L^\infty(\R^d)} \le C(T)\eta\sqrt{N} 
  = C(T)N^{1/2-\beta},
\end{align*}
which diverges as $N\to\infty$ because of $\beta < 1/2$. Therefore, we need a more sophisticated argument. For this, we use formulation \eqref{8.F123} for $s=0$, replace $\phi$ by $\psi_{\eta,t,\phi}^t(r)$ for fixed $r$, and integrate over $r\in(s,t)$ (where $s$ is now another variable). This yields
\begin{align}\label{8.dF2}
  & \int_{s}^t\langle\mathcal{F}^\eta(r),
  \psi^1_{\eta,t,\phi}(r)\rangle\dd r 
  - \int_{s}^t\langle\mathcal{F}^\eta(0),
  T_{\psi_{\eta,t,\phi}^1(r)}^r(0)\rangle\dd r
  = F_1^2 + F_2^2 + F_3^2, \quad\mbox{where } \\
  & F_1^2 = \frac{\kappa\sqrt{2\sigma}}{\sqrt{N}}\sum_{i=1}^N
  \int_{s}^t\int_{0}^r\na T_{\psi_{\eta,t,\phi}^1(r)}^r
  (v,X_i^\eta(v))\cdot\dd W_i(v)\dd r, \nonumber \\
  & F_2^2 = \frac{\kappa}{\sqrt{N}}\int_{s}^t\int_{0}^r
  \langle\mathcal{F}^\eta(v),
  (\na V^\eta*\mathcal{F}^\eta)
  \cdot\na T^{r}_{\psi^1_{\eta,t,\phi}(r)}(v)\rangle\dd v\dd r, 
  \nonumber \\
  & F_3^2 = \kappa\int_{s}^t\int_{0}^r\langle\mathcal{F}^\eta(v),
  \psi^2_{\eta,t,r,\phi}(v)\rangle\dd v\dd r, \nonumber 
\end{align}
where we introduced the next iteration as
\begin{align*}
  \psi^2_{\eta,t,r,\phi}(v) 
  &:= (\na V^\eta*\bar{u}^\eta)
  \cdot\na T^{r}_{\psi^1_{\eta,t,\phi}(r)}(v) 
  - \na V^\eta*(\bar{u}^\eta\na T^r_{\psi^1_{\eta,t,\phi}(r)}(v)) \\
  &\phantom{xx}- (\na\Phi*u)\cdot\na T^{r}_{\psi^1_{\eta,t,\phi}(r)}(v) 
  + \na\Phi*(u\na T^r_{\psi^1_{\eta,t,\phi}(r)}(v)).
\end{align*}
With this iteration, we obtain a better bound than for $\psi^1_{\eta,t,\phi}$. Indeed, by similar estimates as for $\psi_{\eta,t,\phi}^1$ and the fact that
\begin{align}\label{8.DkT2}
  \|\mathrm{D}^k T_{\psi^1_{\eta,t,\phi}(r)}^r\|_{L^\infty(\R^d)}
  \le C(T)\|\psi^1_{\eta,t,\phi}(r)\|_{W^{k,\infty}(\R^d)}
  \le C(T)\eta
\end{align}
(because of \eqref{8.psi1} and Lemma \ref{lem.estT}), the following bound holds:
\begin{align*}
  |F_3^2(t,s)| \leq C(T) \sqrt{N}\sup_{s<r<t}\sup_{s<v<r} 
  \|\psi^2_{\eta,t,r,\phi}(v)\|_{L^\infty(\R^d)}.
\end{align*}
The norm is estimated by using \eqref{8.DkT2} and Lemmas \ref{lem.wLinfty} and \ref{lem.estT}:
\begin{align*}%\label{8.psi2}
  \|\psi^2_{\eta,t,r,\phi}(v)\|_{L^\infty(\R^d)}
  &\le C(T)\eta\|\na T_{\psi^1_{\eta,t,\phi}(r)}^r (v)\|_{L^{\infty}(\R^d)} \|\bar{u}^\eta (v)
  \|_{L^1(\R^d)\cap L^\infty(\R^d)}\\
  &\phantom{xx}+ C(T) \|T_{\psi^1_{\eta,t,\phi}(r)}^r (v) \|_{W^{1,\infty}(\R^d)}\sup_{0<v<T}
  \|(\bar{u}^\eta-u)(v)\|_{L^1(\R^d)\cap L^\infty(\R^d)}\\
  &\leq C(T)\eta \|\psi_{\eta,t,\phi}^1(r)\|_{W^{1,\infty}(\R^d)} 
  \leq C(T)\eta^2.
\end{align*}
Again, by putting the derivatives on $u$, $\bar{u}^\eta$, and $T_{\psi^1_{\eta,t,\phi}(r)}^r (s)$, we find from \eqref{8.psi1} that
\begin{align*}
  \|\mathrm{D}^k\psi^2_{\eta,t,r,\phi}(s)\|_{L^\infty(\R^d)} 
  \leq C(T) \eta \|\mathrm{D}^{k+1} \psi_{\eta,t,\phi}^1(r) \|_{L^\infty(\R^d)}
  \leq C(T)\eta^2 \|\phi\|_{W^{k+2,\infty}(\R^d)}. 
\end{align*}
Observe that in every iteration, we gain a power of $\eta$, which is ensured by a finite norm of $\bar{u}^\eta$ and $u$ in $W^{k+1,\infty}(\R^d)\cap W^{k+1,1}(\R^d)$.

After the second iteration, we infer from \eqref{8.dF2} and
\begin{align*}
  |F_3^2| \leq \bigg|\int_{s}^{t} \int_{0}^r 
  \langle\mathcal{F}^\eta(s),\psi^2_{\eta,t,r,\phi}(v)\rangle
  \dd v\dd r \bigg| 
  \leq C(T) \sqrt{N}\|\psi^2_{\eta,t,r,\phi}(v)\|_{L^\infty(\R^d)} 
\end{align*}
an estimate for $F_3^1$:
\begin{align*}
  \mathbb{E}|F_3^1|
  &\le \mathbb{E}\bigg|\int_{s}^t\langle\mathcal{F}^\eta(r),
  \psi^1_{\eta,t,\phi}(r)\rangle\dd r \bigg| \\
  &\le \mathbb{E} \big|F_1^2 + F_2^2 \big|  
  + \mathbb{E}\bigg|\int_{s}^t\langle\mathcal{F}^\eta(0),
  T_{\psi_{\eta,t,\phi}^1(r)}^r(0)\rangle\dd r\bigg|
  +  C(T) \sqrt{N}\eta^2,
\end{align*}
where the term with respect to $\mathcal{F}^\eta(0)$ can be bounded uniformly in $N$ and $\eta$, since $(\eta^{-1} T_{\psi_{\eta,t,\phi}^1(r)}^r(0))$ is a sequence of uniformly bounded smooth functions with bounded derivatives and since the initial conditions are i.i.d. This leads to
\begin{align*}
  \eta\mathbb{E}\bigg|\int_{s}^t\langle\mathcal{F}^\eta(0),
  \eta^{-1}T_{\psi_{\eta,t,\phi}^1(r)}^r(0)\rangle\dd r\bigg| 
  \leq \eta\int_{s}^t\mathbb{E}\big|\langle\mathcal{F}^\eta(0),
  \eta^{-1}T_{\psi_{\eta,t,\phi}^1(r)}^r(0)\rangle\big| \dd r 
  \leq C(T) \eta 
\end{align*}
and hence to
\begin{align*}
  \mathbb{E}|F_3^1|\le \mathbb{E} \big|F_1^2 + F_2^2 \big| 
  + C(T)\eta +  C(T) \sqrt{N}\eta^2,
\end{align*}
for a different constant $C(T)>0$.

After $\ell$ iterations, we need the condition $u_0 \in W^{\ell+1}(\R^d)$ in order to bound the gradient of $\psi^\ell_{\eta}$ (to shorten the notation, we omit most of the indices here). This gives
\begin{align}\label{8.F_3_1_iteration}
  \mathbb{E}|F_3^1|&\le\mathbb{E}\bigg|\int_{s}^t
  \langle\mathcal{F}^\eta(r),\psi^1_{\eta,t,\phi}(r)\rangle\dd r 
  \bigg| \\
  &\le \big|\mathbb{E}(F_1^2 + F_2^2) \big| + C(T)\eta
  + C(T)\mathbb{E}\bigg|\int_{s}^{t} \int_{0}^r \langle\mathcal{F}^\eta(v),\psi^2_{\eta,t,r,\phi}(v)\rangle
  \dd v\dd r \bigg| \nonumber \\
  & \leq \sum_{k=2}^{\ell}\big(\big|\mathbb{E}(F_1^k + F_2^k) \big| + C(T)\eta^{k-1}\big) + C(T)\sqrt{N}\eta^\ell,  \nonumber
\end{align}
where $F_1^k$ and $F_2^k$ are defined iteratively by replacing $\psi^{k-1}_\eta(s)$ with the associated dual backward process $T^r_{\psi^{k-1}_\eta(r)}(s)$ (we omit the indices for a clearer readability). Note that in the $k$-th iteration, terms involving the initial fluctuations $\mathcal{F}^\eta(0)$ can be bounded by $C(T)\eta^k$, since $\psi^{k-1}_\eta(s)$ is bounded by $\eta^{k-1}$ times a uniform constant, and we use Lemma \ref{lem.estT} and the fact that the initial conditions are i.i.d.\ with density $u_0$.

It remains to estimate $F_1^k$ and $F_2^k$ for $k\geq 2$. First, for $k=2$, we use the Burkholder--Davis--Gundy inequality (similarly as for the term $K_6$ in Step 6 of the proof of Theorem \ref{thm.L2}; see Section \ref{sec.second}):
\begin{align*}
  \E|F_1^2| \le \sqrt{N}\bigg(\frac{C(\sigma,T)}{N}
  \E\int_{s}^t\int_{0}^r\big\langle\mu^\eta(v),
  |\na T_{\psi^1_{\eta,t,\phi}(r)}^{r}(v)|^2
  \big\rangle\dd v\bigg)^{1/2}.
\end{align*}
We infer from 
$$
  \|\na T_{\psi^1_{\eta,t,\phi}(r)}^{r}(v)
  \|_{L^\infty(\R^d)}^2
  \le C(T)\|\psi^1_{\eta,t,\phi}(r)\|_{W^{1,\infty}(\R^d)}^2
  \le C(T)\eta^2
  = C(T)N^{-2\beta}
$$ 
that $\E|F_1^2|$ converges to zero as $N\to\infty$:
$$
  \E|F_1^2|\le C(T)N^{-\beta}\to 0\quad\mbox{as }N\to\infty.
$$

Taking into account the estimate of $L^*(T)$ from Remark \ref{rem.F2}, we find for $F_2^2$ that
\begin{align*}
  \E|F_2^2| &\leq \sqrt{N}\E\int_0^t\int_0^r\bigg|
  \langle(\mu^\eta-\bar{u}^\eta),
  (\na V^\eta*(\mu^\eta-\bar{u}^\eta))
  \cdot\na T_{\psi^1_{\eta,t,\phi}}^r(s)\rangle\bigg|\dd s\dd r \\
  &\le C(\sigma,T)N^{-\eps}\sup_{0<t,r<T}
  \big(\|\na T_{\psi^1_{\eta,t,\phi}}^r\|_{L^\infty(\R^d)} 
  +  \|\na T_{\psi^1_{\eta,t,\phi}}^r\|_{L^\infty(\R^d)}^2
  + \|\mathrm{D}^2 T_{\psi^1_{\eta,t,\phi}}^r\|_{L^\infty(\R^d)}\big) \\
  &\phantom{xx}+ \sigma N^{1/2}\E\int_0^T
  \|\na(f^\eta-g^\eta)\|_{L^2(\R^d)}^2\dd s.
\end{align*}
Thus, it follows from Theorem \ref{thm.L2} that
\begin{align*}
  \E|F_2^2|\le C(\sigma,T)N^{-\eps} \to 0.
\end{align*}

Repeating the previous arguments, the $k$th iterations $F_1^k$ and $F_2^k$ are of order $N^{-\eps}$ and also converge to zero. Hence,
\begin{align*}
  \sum_{k=2}^\ell \mathbb{E}\big|F_1^k + F_2^k\big| \to 0
\end{align*} 
for any $\ell \geq 2$. Finally, choosing $K^*\in\N$ such that $K^*>1/(2\beta)$, we conclude from \eqref{8.F_3_1_iteration} that
\begin{align*}
  \mathbb{E}|F_3| = \mathbb{E}|F_3^1|
  \leq \sum_{k=2}^{K^*}\big(\mathbb{E}|F_1^k + F_2^k|
  + C(T)\eta^{k-1}\big) + C(T)N^{1/2-K^*\beta} \to 0,
\end{align*}
where $C(T)$ depends on the $W^{K^*,\infty}(\R^d)\cap W^{K^*,1}(\R^d)$ norm of $\bar{u}^\eta$ and $u$.
\end{proof}

\begin{lemma}\label{lem.fourier}
Let the assumptions of Theorem \ref{thm.fluct} hold. Then for all $0\le s<t<\infty$ and $\theta\in\R$,
\begin{align*}
  \E\bigg|\E\bigg(\exp\big[\mathrm{i}\theta\langle\mathcal{F}^\eta(t),
  \phi\rangle \big]\bigg|\mathscr{F}_s\bigg) - \exp\bigg(\mathrm{i}\theta
  \langle\mathcal{F}^\eta(s),T_\phi^t(s)\rangle 
   - \sigma\theta^2\int_s^t\langle u(r),|\na T_\phi^t(r)|^2\rangle
  \dd r \bigg)\bigg|\to 0
\end{align*}
as $\eta=N^{-\beta}\to 0$, where $u$ solves \eqref{1.u}, $T_\phi^t$ solves \eqref{4.back}, and $\mathscr{F}_s = \sigma(\langle\mathscr{F}_{r}, \phi \rangle:0\leq r\leq s$, $\phi$ is a Schwartz function$)$.
\end{lemma}

The proof of Lemma \ref{lem.fourier} follows the lines of \cite[Sec.~3]{Oel87}, but since we are working in a different setting and since we compare with the intermediate PDE solution $\bar{u}^\eta$ instead of a correction term, we present it here for completeness.

\begin{proof}
Let $\phi$ be a Schwartz function. We recall from \eqref{8.F123} that
\begin{align*}
  &\langle\mathcal{F}^\eta(t),\phi\rangle 
  - \langle\mathcal{F}^\eta(s),T_\phi^t(s)\rangle
  = (F_{1} + F_{2} + F_{3})(s,t), \quad\mbox{where } \\
  & F_{1}(s,t) = \frac{\sqrt{2\sigma}}{\sqrt{N}}\sum_{i=1}^N\int_s^t
  \na T_\phi^t(r,X_i^\eta(r))\cdot\dd W_i(r), \\
  & F_{2}(s,t) = \frac{\kappa}{\sqrt{N}}\int_s^t\langle\mathcal{F}^\eta(r),
  (\na V^\eta*\mathcal{F}^\eta(r))\cdot\na T_\phi^t(r)\rangle\dd r, \\
  & F_{3}(s,t) = \kappa\int_s^t\langle\mathcal{F}^\eta(r),
  \psi_{\eta,t,\phi}^1(r)\rangle\dd r,
\end{align*}
and $\psi^1_{\eta,t,\phi}$ is defined in \eqref{8.defpsi1}. In particular, we can write 
\begin{align}\label{8.Feta_first}
  \exp(\mathrm{i}\theta\langle\mathcal{F}^\eta(t),\phi\rangle)
  = \exp(\mathrm{i}\theta\langle\mathcal{F}^\eta(s),T_\phi^t(s)\rangle)
  e^{\mathrm{i}\theta F_{1}(s,t)}
  e^{\mathrm{i}\theta F_{2}(s,t)}
  e^{\mathrm{i}\theta F_{3}(s,t)}.
\end{align}
Furthermore, we introduce
\begin{equation}\label{8.F4}
  F_4(s,t) = \sigma\int_s^t\langle u(r),|\na T^t_\phi(r)|^2\rangle\dd r,
\end{equation}
which represents the variance factor of the limiting characteristic function. It follows from the definition of $\mathscr{F}_s$ and formulation \eqref{8.Feta_first} that
\begin{align}\label{8.G12}
  \E\big|&\E\big(\exp(\mathrm{i}\theta\langle\mathcal{F}^\eta(t),
  \phi\rangle)|\mathscr{F}_s\big)
  - \exp\big(\mathrm{i}\theta\langle\mathcal{F}^\eta(s),T_\phi^t(s)
  \rangle-\theta^2 F_{4}(s,t)\big)\big| \\
  &= \E\Big[\Big|\E\Big(\exp(\mathrm{i}\theta\langle\mathcal{F}^\eta(t),
  \phi\rangle)
  - \exp\big(\mathrm{i}\theta\langle\mathcal{F}^\eta(s),T_\phi^t(s)
  \rangle-\theta^2 F_4(s,t)\big)\Big|\mathscr{F}_s\Big)\Big| \Big] \nonumber \\
  &= \E\Big[\Big|\E\Big(\exp(\mathrm{i}\theta\langle\mathcal{F}^\eta(s),
  T_\phi^t(s)\rangle)e^{\mathrm{i}\theta F_1(s,t)}
  e^{\mathrm{i}\theta F_2(s,t)}e^{\mathrm{i}\theta F_3(s,t)} \nonumber \\
  &\phantom{xx}
  - \exp(\mathrm{i}\theta \langle\mathcal{F}^\eta(s),
  T_\phi^t(s)\rangle)e^{-\theta^2F_4(s,t)}
  \Big|\mathscr{F}_s\Big)\Big|\Big]
  \le G_1 + G_2, \nonumber 
\end{align}
where
\begin{align*}
  G_1 &= \E\Big|\E\Big[\exp(\mathrm{i}\theta\langle\mathcal{F}^\eta(s),
  T_\phi^t(s)\rangle)\big(e^{\mathrm{i}\theta F_1(s,t)} 
  - e^{-\theta^2F_4(s,t)}\big)\big|\mathscr{F}_s\Big]\Big|, \\
  G_2 &= \E\Big|\E\Big[\exp(\mathrm{i}\theta\langle\mathcal{F}^\eta(s), T_\phi^t(s)
  \rangle)e^{\mathrm{i}\theta F_1(s,t)}
  \big(e^{\mathrm{i}\theta F_2(s,t)}e^{\mathrm{i}\theta F_3(s,t)}
  - 1\big)\big|\mathscr{F}_s\Big]\Big|.
\end{align*}
The term $G_2$ converges to zero as $N \to \infty$ (or $\eta \to 0$), since, taking into account Lemma \ref{lem.F23}, the mean-value theorem, and the limit $\mathbb{E}|F_2(s,t) + F_3(s,t)| \to 0$ from Lemma \ref{lem.F23},
\begin{align*}
  G_2 &\le \E\big[\E\big(|e^{\mathrm{i}\theta F_{2}(s,t)}
  e^{\mathrm{i}\theta F_{3}(s,t)} - 1|\big|\mathscr{F}_s\big)\big]
  = \E\big[|e^{\mathrm{i}\theta( F_{2}(s,t)
  + F_{3}(s,t))} - 1|\big] \\
  &\leq \theta\E\big[|F_2(s,t) + F_3(s,t)|\big] \to 0.
\end{align*} 

For the term $G_1$, we first note that $\exp(\mathrm{i}\theta\langle\mathcal{F}^\eta(s), T_\phi^t(s)\rangle)$ is $\mathscr{F}_s$-measurable, which gives
\begin{align*}
  G_1 &= \E\big|\E\big(e^{\mathrm{i}\theta F_{1}(s,t)}
  - e^{-\theta^2F_{4}(s,t)}\big|\mathscr{F}_s\big)\big| \\
  &= \E\big|\E\big((e^{\mathrm{i}\theta F_{1}(s,t)+\theta^2F_{4}(s,t)}
  - 1)e^{-\theta^2 F_{4}(s,t)}\big|\mathscr{F}_s\big)\big|. 
\end{align*}
In the following, we rewrite the expression $e^{\mathrm{i}\theta F_{1}(s,t)+\theta^2F_{4}(s,t)}$. To this end, we apply It\^o's formula to the nonlinear function $f(x)=e^{\mathrm{i}\theta x}$ and the process
\begin{align*}
  \dd Y_s(r) = \frac{\sqrt{2\sigma}}{\sqrt{N}}\sum_{i=1}^N
  \na T_\phi^t(r,X_i^\eta(r))\cdot\dd W_i(r)
  - \mathrm{i}\sigma\theta\langle u(r),|\na T_\phi^t(r)|^2\rangle\dd r,
\end{align*}
with $Y_s(s)=0$ and $s>0$ fixed, implying that
\begin{align*}
  e^{\mathrm{i}\theta F_{1}(s,t)+\theta^2F_{4}(s,t)}
  &= f(Y_s(t)) = f(Y_s(s))
  + \sigma\theta^2\int_s^t e^{\mathrm{i}\theta Y_s(r)}
  \langle u(r),|\na T_\phi^t(r)|^2\rangle\dd r \\
  &\phantom{xx}- \frac{\sigma}{N}\theta^2\sum_{i=1}^N\int_s^t
  e^{\mathrm{i}\theta Y_s(r)}|\na T_\phi^t(r,X_i^\eta(r))|^2\dd r \\
  &\phantom{xx}+ \mathrm{i}\theta\frac{\sqrt{2\sigma}}{\sqrt{N}}
  \sum_{i=1}^N\int_s^t e^{\mathrm{i}\theta Y_s(r)}\na T_\phi^t(r,X_i^\eta(r))\cdot\dd W_i(r).
\end{align*}
It follows from $f(Y_s(s))=f(0)=1$ and the definition of $\mu^\eta$ that
\begin{align*}
  &e^{\mathrm{i}\theta F_{1}(s,t)+\theta^2F_{4}(s,t)} - 1
  = \mathrm{i}\theta\frac{\sqrt{2\sigma}}{\sqrt{N}}
  \sum_{i=1}^N\int_s^t e^{\mathrm{i}\theta Y_s(r)}
  \na T_\phi^t(r,X_i^\eta(r))\cdot\dd W_i(r) \\
  &\phantom{xxxx}+ \sigma\theta^2\int_s^te^{\mathrm{i}\theta Y_s(r)}
  \bigg(\langle u(r),|\na T_\phi^t(r)|^2\rangle
 - \frac{1}{N}\sum_{i=1}^N|\na T_\phi^t(r,X_i^\eta(r))|^2\bigg)\dd r \\
  &\phantom{xx}= \mathrm{i}\theta\frac{\sqrt{2\sigma}}{\sqrt{N}}
  \sum_{i=1}^N\int_s^t e^{\mathrm{i}\theta Y_s(r)}
  \na T_\phi^t(r,X_i^\eta(r))\cdot\dd W_i(r) \\
  &\phantom{xxxx}- \sigma\theta^2\int_s^te^{\mathrm{i}\theta Y_s(r)}
  \langle\mu^\eta(r)-u(r),|\na T_\phi^t(r,X_i^\eta(r))|^2\rangle\dd r.
\end{align*}
Since
\begin{align*}
  |e^{i\theta Y_s(\omega,r)}|
  &\leq \bigg|\exp\bigg(\mathrm{i}\theta \int_{s}^{r}\frac{\sqrt{2\sigma}}{\sqrt{N}} 
  \sum_{i = 1}^N  \na T^{t}_{\phi}(\tau,X^\eta_i(\tau)) 
  \cdot\dd W_i(\tau)\bigg)\bigg| \\ 
  &\phantom{xx}\times
  \exp\bigg(\theta^2 \sigma \int_{s}^r \langle u(\tau), 
  |\nabla T_\phi^t(\tau,\cdot)|^2\rangle \dd \tau\bigg) \\
  &\leq \exp\bigg(\theta^2\sigma\int_{0}^t 
  \|\nabla T_\phi^t(r) \|_{L^\infty(\R^d)} \dd r\bigg),
\end{align*}
and trivially $\|e^{-\theta^2 F_{4}(s,t)}\|_{L^\infty(\R^d)}\leq 1$, 
we arrive eventually at
\begin{align*}
  \E\big(&\big|\E\big(\big[e^{\mathrm{i}\theta F_1(s,t)
  +\theta^2 F_{4}(s,t)}-1\big]e^{-\theta^2 F_{4}(s,t)}
  \big| \mathscr{F}_s \big)\big|\big) \\
  &\le \sigma\theta^2 \exp\bigg(\sigma\theta^2\int_{0}^t 
  \|\nabla T_\phi^t(\tau) \|_{L^\infty(\R^d)} \dd \tau
  \bigg) \E\bigg(\int_{s}^t\big|\langle\mu^{\eta}(r) - u(r), 
  |\nabla T_{\phi}^t(r)|^2\rangle\big|\dd r\bigg),
\end{align*}
where we have used the martingale property of It\^o integrals to conclude that 
$$
  \mathbb{E}\bigg(\mathrm{i}\theta\frac{\sqrt{2\sigma}}{\sqrt{N}}
  \sum_{i=1}^N\int_s^t e^{\mathrm{i}\theta Y_s(r)} \na T_\phi^t(r,X_i^\eta(r))\cdot\dd W_i(r)| \mathscr{F}_s\bigg) = 0.
$$ 
We deduce from limit \eqref{4.remdiff} in Remark \ref{rem.diff} that
$$
  \mathbb{E} \big(\big| \big\langle\mu^\eta(r)-u(r),
  |\na T_\phi^t(r)|^2\big\rangle \big|\big) \to 0,
$$
and consequently,
\begin{align*}
  G_1 &\le \E\big(\big|\E\big(\big[e^{\mathrm{i}\theta F_1(s,t)
  +\theta^2 F_{4}(s,t)}-1\big]e^{-\theta^2 F_{4}(s,t)}
  \big| \mathscr{F}_s \big)\big|\big) \\
  %\E\big(\big|e^{\mathrm{i}\theta F_{3}(s,t)
  %+ \theta^2F_{4}(s,t)} - 1\big|\big)
  %\|e^{-\theta^2F_{4}(s,t)}\|_{L^\infty(\R^d)} \\
  &\le \sigma\theta^2\E\bigg(\int_s^t
  \big|\big\langle\mu^\eta(r)-u(r),
  |\na T_\phi^t(r)|^2\big\rangle \big|\dd r\bigg)\to 0.
\end{align*}
We have shown that the left-hand side of \eqref{8.G12} converges to zero,
$$
  \E\big|\E\big(\exp(\mathrm{i}\theta\langle\mathcal{F}^\eta(t),
  \phi\rangle)|\mathscr{F}_s\big)
  - \exp\big(\mathrm{i}\theta\langle\mathcal{F}^\eta(s),T_\phi^t(s)
  \rangle-\theta^2 F_{4}(s,t)\big)\big|\to 0,
$$
which, by definition \eqref{8.F4}, finishes the proof.
\end{proof}

\begin{remark}[Explanation of Lemma \ref{lem.fourier}]\rm 
By the notion of the regular conditional law, it holds asymptotically for $s=0$ that, for $N\to\infty$ (or $\eta\to 0$),
\begin{align*}%\label{Main.Result.Conditioned_Law_s}
  \mathscr{L}(\langle\mathcal{F}^\eta(t), \phi \rangle | \mathscr{F}_{0}) 
  \to \mathcal{N}\bigg( \langle \mathcal{F}_0, 
  T_\phi^t(0,\cdot)\rangle, 2\sigma\int_{0}^{t} \langle u(\tau), 
  |\nabla T_\phi^t(\tau,\cdot)|^2\rangle \dd \tau\bigg).
\end{align*}
As we start with i.i.d.\ initial conditions distributed with density function $u_0$, $\mathcal{F}_0$ is a Gaussian field with zero mean and variance
\begin{align*}
  \mathbb{E}(\langle \mathcal{F}_0, \phi \rangle ^2)
  = \int_{\R^d} \phi(x)^2u_0(x)\dd x 
  - \bigg(\int_{\R^d} \phi(x) u_0(x)\dd x\bigg)^2.
\end{align*}
This implies that our asymptotic limiting process of the intermediate fluctuations $\langle\mathcal{F}_t, \phi \rangle $ is Gaussian for any $t>0$ and any Schwartz function $\phi$, since asymptotically,
\begin{align*}
  \E\big(e^{\mathrm{i}\theta \langle \mathcal{F}_t, 
  \phi \rangle}\big) 
  &= \E\big(\mathbb{E}( e^{\mathrm{i}\theta
  \langle \mathcal{F}_t, \phi \rangle}| \mathscr{F}_0)\big) \\
  &= \E\bigg(\exp\bigg(\mathrm{i}\theta
  \langle \mathcal{F}_0, T_\phi^t(0,\cdot)\rangle 
  - \sigma\theta^2 \int_{0}^{t} 
  \langle u(\tau), |\nabla T_\phi^t(\tau,\cdot)|^2\rangle 
  \dd \tau\bigg)\bigg) \\
  &= \exp\bigg(-\sigma\theta^2\int_{0}^{t} 
  \langle u(\tau), |\nabla T_\phi^t(\tau,\cdot)|^2\rangle 
  \dd \tau\bigg)\E\big(e^{\mathrm{i}\theta\langle \mathcal{F}_0, 
  T_\phi^t(0,\cdot)\rangle}\big).
\end{align*}
Since $\langle \mathcal{F}_0, \psi \rangle$ is Gaussian for any $\psi$, this shows that $\langle\mathcal{F}_t, \phi \rangle$ is also Gaussian. In particular, we have
\begin{align*}%\label{8.limitGauss}
  \langle\mathcal{F}_t, \phi \rangle 
  \sim \mathcal{N}\bigg(0,\langle u_0,(T^t_\phi (0))^2 \rangle 
  - \langle u_0, T^t_\phi (0)\rangle^2 
  + 2\sigma\int_{0}^t \langle u(s), |\nabla T_\phi (s)|^2\rangle 
  \dd s\bigg).
\end{align*}
\qed\end{remark}

\begin{remark}[Interpretation of an Ornstein--Uhlenbeck process]\rm
By the characterisation of an Ornstein--Uhlenbeck process $\mathcal{F}_t$ such that the process $\langle \mathcal{F}_t, \phi \rangle$ conditional to $\mathscr{F}_s$ is normal with mean $\langle \mathcal{F}_s, T_\phi^t(s)\rangle$ and variance 
$$
  \int_{s}^t \big\|(2\sigma u(\tau)\nabla T_\phi^t(\tau))^{1/2}
  \big|_{L^2(\R^d)}^2 \dd \tau
$$ 
(see, e.g., \cite{FGN13}), the limiting process $\mathcal{F}_t$ can be formally interpreted as an Ornstein--Uhlenbeck process. This process is a formal solution to the following linear stochastic PDE:
\begin{align*}
  \partial_t\mathcal{F} = \sigma\Delta\mathcal{F} 
  - \kappa\diver(\mathcal{F}\na\Phi*u+u\na\Phi*\mathcal{F})
  - \sqrt{2\sigma}\nabla \cdot(\sqrt{u}\xi)
\end{align*} 
in $\R^d\times(0,T)$ with the initial datum $\mathcal{F}(0)=\mathcal{F}_0$, where $\xi\in L^2((0,T)\times\R^d;\R^d)\to L^2(\Omega;\R^d)$ is a vector-valued space-time white noise and $u$ solves \eqref{1.u}. Deriving a rigorous justification of this formal argument is open and left for future research.
\qed\end{remark}

%%%%%%%%%%%%%%%%%%%%%%%%%%%%%%%%%%%%%%%%%%%%%%%%%%%%%%%%%%%%%%%%%%%%
\appendix

\section{Estimates for the Riesz kernel}\label{sec.riesz}

We show some estimates related to the Riesz kernel $\Phi(x)=|x|^{-\lambda}$ for $x\in\R^d$ with $\lambda>0$. Some results are stated for the more general setting $0<\lambda < d$ although the results of this paper are concerned with sub-Coulomb potentials. Recall that $\Phi=\Psi*\Psi$, $\chi^\eta=\xi^\eta*\xi^\eta$ and accordingly $V^\eta=\chi^\eta*\Phi$, $Z^\eta=\xi^\eta*\Psi$.

\begin{lemma}\label{lem.Psi}
Let $0<\lambda<d$ and let $\Phi=\Psi*\Psi$ for some function $\Psi$. Then $\Psi(x)=c_{\lambda, d}|x|^{-(\lambda+d)/2}$ for some $c_{\lambda,d}>0$.
\end{lemma}

\begin{proof}
We use the fact that the Fourier transform of the Riesz kernel $|x|^{-\lambda}$ for $0<\lambda<d$ equals $F[|x|^{-\lambda}](z) = C_{\lambda,d}|z|^{\lambda-d}$ for some $C_{\lambda,d}>0$. Then $F[\Psi]^2 = F[\Psi*\Psi]=F[\Phi]$ implies that
$F[\Psi](z) = C_{\lambda,d}^{1/2}|z|^{(\lambda-d)/2}$.  Transforming back yields $\Psi(x)=c_{\lambda,d}|x|^{-(\lambda+d)/2}$ for some $c_{\lambda,d}>0$. 
\end{proof}

We prove some estimates for $V^\eta$ and $Z^\eta$ and their derivatives.

\begin{lemma}\label{lem.DkV}
Let $0<\lambda<d$ and $k\in\N\cup\{0\}$. Then there exists $C>0$ such that for any $\eta=N^{-\beta}$,
$$
  \|\mathrm{D}^k V^\eta\|_{L^\infty(\R^d)}
  \le CN^{\beta(\lambda+k)}.
$$
\end{lemma}

\begin{proof}
We have by definition $V^\eta=\chi^\eta*\Phi$ and
\begin{align*}
  |\mathrm{D}^k V^\eta(x)| &= \bigg|\int_{\R^d}\mathrm{D}^k
  \chi^\eta(x-y)\Phi(y)\mathrm{1}_{B_\rho(0)}(y)\dd y
  + \int_{\R^d}\mathrm{D}^k\chi^\eta(x-y)\Phi(y)
  \mathrm{1}_{B_\rho(0)^c}(y)\dd y\bigg| \\
  &\le \|\mathrm{D}^k\chi^\eta\|_{L^\infty(\R^d)}
  \|\Phi\|_{L^1(B_{\rho}(0))} 
  + \|\mathrm{D}^{k}\chi^\eta\|_{L^1(\R^d)}
  \|\Phi\|_{L^\infty(B_{\rho}(0)^c)},
\end{align*}
where $B_\rho(0)$ denotes the ball around the origin with radius $\rho:=N^{-\beta}$. Using $\|\mathrm{D}^k\chi^\eta\|_{L^\infty(\R^d)}$ $\le CN^{\beta(d+k)}$ and $\|\mathrm{D}^k\chi^\eta\|_{L^1(\R^d)}\le CN^{\beta k}$ as well as
\begin{align*}
  \|\Phi\|_{L^1(B_{\rho}(0))} \le C\int_0^{N^{-\beta}}
  r^{-\lambda}r^{d-1}\dd r \le CN^{\beta(\lambda-d)}, \quad
  \|\Phi\|_{L^\infty(B_{\rho}(0)^c)}
  \le CN^{\beta\lambda}, 
\end{align*}
we find that $|\mathrm{D}^k V^\eta(x)| \le CN^{\beta(\lambda+k)}$. Note that the condition $\lambda<d$ is used to estimate the $L^1$ norm of $\Phi$ around the origin. 
\end{proof}

\begin{lemma}\label{lem.DkZ} 
There exists $C>0$ such that for any $\eta=N^{-\beta}$,
\begin{align*}
  \|Z^\eta\|_{L^2(\R^d)} \le CN^{\beta \lambda/2}, \quad
  \|\na Z^\eta\|_{L^2(\R^d)} \le CN^{\beta(\lambda+1)/2} 
  &\quad\mbox{if }0<\lambda<d-2.
\end{align*}
\end{lemma}

\begin{proof}
We infer from $Z^\eta=\xi^\eta*\Psi$, recalling that $\Psi = c_{\lambda,d}|x|^{-(\lambda +d)/2}$, and the Hardy--Littlewood--Sobolev inequality \cite[Theorem 4.3]{LiLo01} that for any $g\in L^2(\R^d)$,
\begin{align*}
  \langle Z^\eta, g\rangle_{L^2(\R^d)}
  = c_{\lambda,d}\int_{\R^d}\int_{\R^d} \frac{\xi^\eta(x)g(y)}{|x-y|^{(\lambda+d)/2}}\dd x\dd y 
  \leq C(d,\lambda)\|\xi^\eta\|_{L^p(\R^d)} \|g\|_{L^2(\R^d)},
\end{align*}
where $1/p +(\lambda +d)/2d + 1/2 = 2$, which is equivalent to $p= 2d/(2d-\lambda) >1$. The $L^p(\R^d)$ norm of $\xi^\eta$ can be estimated as
\begin{align*}
  \|\xi^\eta\|_{L^p(\R^d)} = N^{\beta d(p-1)/p} \|\xi\|_{L^p(\R^d)} 
  \leq CN^{\beta d(p-1)/p}.
\end{align*}
Then $(p-1)/p = \lambda/2d$ yields $\|Z^\eta\|_{L^2(\R^d)} \leq CN^{\beta\lambda/2}$. 
%The $L^\infty(\R^d)$ norm follows from the proof of Lemma \ref{lem.DkV} after replacing $\lambda$ by $(\lambda+d)/2$.
For the last bound, we estimate similarly as before:
\begin{align*}
  \|\na Z^\eta\|_{L^2(\R^d)} 
  &\le C(\lambda, d)\|\xi^\eta\|_{L^q(\R^d)},
\end{align*}
where $1/q + (\lambda+1+d)/2d +1/2 =2$, which is equivalent to $1/q = (2d-\lambda -1)/2d$. Hence, since $(q-1)/q = (\lambda +1)/2d$, 
\begin{align*}
 \|\na Z^\eta\|_{L^2(\R^d)} 
  &\le C(\lambda, d)N^{\beta(\lambda+1)/2},
\end{align*}
which finishes the proof.
\end{proof}

\begin{lemma}\label{lem.DkVbaru}
%and $p>d/(d-\lambda-2)$. 
Let $k\in\N\cup\{0\}$ and $\bar{u}^\eta$ be the solution to \eqref{1.baru} (see Appendix \ref{sec.pde}). Then there exist constants $C_1$, $C_2>0$ independent of $\eta$ such that
\begin{align*}
  \sup_{0<t<T}\|\mathrm{D}^k V^\eta*\bar{u}^\eta\|_{L^\infty(\R^d)}
  \le C_1\|\bar{u}^\eta\|_{L^\infty(0,T;L^1(\R^d)\cap L^\infty(\R^d))} &\quad\mbox{if }0<\lambda<d, \\ 
  \sup_{0<t<T}\||\mathrm{D}^2 V^\eta|*\bar{u}^\eta\|_{L^\infty(\R^d)}
  \le C_2\|\bar{u}^\eta\|_{L^\infty(0,T;L^1(\R^d)\cap L^\infty(\R^d))} &\quad\mbox{if }0<\lambda<d-2.
\end{align*}
\end{lemma}

\begin{proof}
%Set $q:=p/(p-1)$. Then $1<q<d/(\lambda+2)$. 
By definition of $V^\eta=\chi^\eta*\Phi$, for any $x\in\R^d$,
\begin{align*}
  |(\mathrm{D}^k &V^\eta*\bar{u}^\eta)(x)|
  = |(\chi^\eta*\Phi*\mathrm{D}^{k}\bar{u}^\eta)(x)| \\
  &\le \int_{\R^d}\int_{\R^d}\chi^\eta(z-y)\Phi(z)
  \big(\mathrm{1}_{B_1(0)}+\mathrm{1}_{B^1(0)^c}\big)(y)
  |\mathrm{D}^{k}\bar{u}^\eta(x-y)|\dd y\dd z \\
  &\le \|\chi^\eta\|_{L^1(\R^d)}\big(\|\Phi  \|_{L^1(B_1(0))}
  \|\mathrm{D}^{k}\bar{u}^\eta\|_{L^\infty(\R^d)}
  + \|\Phi\|_{L^\infty(B_1(0)^c)}
  \|\mathrm{D}^{k}\bar{u}^\eta\|_{L^1(\R^d)}\big) \\
  &= \|\Phi\|_{L^1(B_1(0))}
  \|\mathrm{D}^{k}\bar{u}^\eta\|_{L^\infty(\R^d)}
  + \|\Phi\|_{L^\infty(B_1(0)^c)}
  \|\mathrm{D}^{k}\bar{u}^\eta\|_{L^1(\R^d)}.
\end{align*}
We know from the proof of Lemma \ref{lem.DkV} that $\|\Phi\|_{L^1(B_1(0))}$ is finite if $\lambda<d$. 
Moreover, $|\Phi(x)|\le C|x|^{-\lambda}$ is bounded for $|x|>1$ since $\lambda>0$. This shows the first estimate. 

For the second estimate, we need to argue in a different way, since we cannot transfer the derivatives from $\mathrm{D}^2 V^\eta$ to $\bar{u}^\eta$ because of the modulus. Here, we need the assumption $\lambda<d-2$. We split the integral again into two parts:
\begin{align*}
  (|\mathrm{D}^2 V^\eta|*\bar{u}^\eta)(x)
  &= \int_{\R^d}|\mathrm{D}^2 V^\eta(y)|\bar{u}^\eta(x-y)\dd y \\
  &= \int_{\R^d}\bigg|\int_{\R^d}\chi^\eta(y-z)\mathrm{D}^2\Phi(z)
  \big(\mathrm{1}_{B_1(0)}+\mathrm{1}_{B_1(0)^c}\big)(z)\dd z\bigg|
  \bar{u}^\eta(x-y)\dd y.
\end{align*}
This shows that
\begin{align}\label{D2Vaux}
  \||\mathrm{D}^2 V^\eta|*\bar{u}^\eta\|_{L^\infty(\R^d)} 
  &\le \|\chi^\eta*(|\mathrm{D}^2\Phi|\mathrm{1}_{B_1(0)})
  *\bar{u}^\eta\|_{L^{\infty}(\R^d)} \\
  &\phantom{xx}+ \|\chi^\eta*(|\mathrm{D}^2\Phi|\mathrm{1}_{B_1(0)^c})
  *\bar{u}^\eta\|_{L^{\infty}(\R^d)}. \nonumber
\end{align}
The first term is estimated by exploiting the fact that the $L^1(B_1(0))$ norm of $|\mathrm{D}^2\Phi|=C|x|^{-\lambda-2}$ is bounded if $\lambda<d-2$:
\begin{align*}
  \|\chi^\eta*(|\mathrm{D}^2\Phi|\mathrm{1}_{B_1(0)})
  *\bar{u}^\eta\|_{L^{\infty}(\R^d)} 
  &\le \|\chi^\eta\|_{L^1(\R^d)}
  \|\mathrm{D}^2\Phi\|_{L^1(B_1(0))}\|\bar{u}^\eta\|_{L^\infty(\R^d)} \\
  &\le C\|\bar{u}^\eta\|_{L^\infty(\R^d)} \le C.
\end{align*}
The boundedness of the second term on the right-hand side of \eqref{D2Vaux} follows from the boundedness of $|\mathrm{D}^2\Phi|$ on $B_1(0)^c$:
\begin{align*}
  \|\chi^\eta*(|\mathrm{D}^2\Phi|\mathrm{1}_{B_1(0)^c})
  *\bar{u}^\eta\|_{L^{\infty}(\R^d)}
  \le \|\chi^\eta\|_{L^1(\R^d)}\|\mathrm{D}^2\Phi\|_{L^\infty(B_1(0)^c)}
  \|\bar{u}^\eta\|_{L^1(\R^d)} \le C.
\end{align*}
Collecting these estimates finishes the proof.
\end{proof}

\begin{remark}[Coulomb potential with cutoff]\label{rem.coulomb2}\rm
The previous lemma can be extended to the Coulomb potential with cutoff, as shown in \cite[Lemma 6.1]{LaPi17}. In this case, we need to define $V^\eta=\chi^\eta*\Phi_N$ and $\Phi_N$ satisfies $|\Phi_N(x)|\le C\max\{N^\nu,|x|^{-\lambda}\}$ for some $\nu>0$ and $C>0$. Then $\||\mathrm{D}^2V^\eta|*\bar{u}^\eta\|_{L^\infty(\R^d)}\le C\log N$. It can be seen that we may allow this logarithmic bound in our computations.
\qed\end{remark}

%%%%%%%%%%%%%%%%%%%%%%%%%%%%%%%%%%%%%%%%%%%%%%%%%%%%%%%%%%%%%%%%%%%%

\section{PDE analysis}\label{sec.pde}

In this section, we provide a priori estimates for the solutions $\bar{u}^\eta$ and $u$ to \eqref{1.baru} and \eqref{1.u}, respectively, and for a backward dual evolution problem associated to the linearized version of \eqref{1.u}. The well-posedness of similar problems as \eqref{1.u} has been studied exhaustively in the last years; see, e.g., \cite{Lau07}. We see from the mild formulation that the solutions \eqref{1.baru} and \eqref{1.u} are positive.

\subsection{Estimates for the solutions to \eqref{1.u} and \eqref{1.baru}}

Observe that the nonlocal equation \eqref{1.baru} is well-posed for any given $\eta>0$, since the nonlocal term $\na V^\eta*\bar{u}^\eta$ is bounded in any $L^p(\R^d)$ norm thanks to the integrability of $\bar{u}^\eta$. 
We aim to derive $\eta$-independent $L^\infty(\R^d)$ bounds for $\bar{u}^\eta$ and $u$. For this, we introduce the parameter
\begin{equation*}
  p^* := \frac{d}{d-\lambda}\in\bigg(1,\frac{d}{2}\bigg),
  \quad\mbox{where }0<\lambda<d-2.
\end{equation*}
The following $L^p$ estimates follow from ideas in \cite{CLPZ04}, where the Coulomb case is considered.

\begin{lemma}[$L^\infty(\R^d)$ bounds]\label{lem.Linfty}
Let $u_0\in L^\infty(\R^d)$. There exists a constant $C(p^*)>0$, only depending on $d$, $\lambda$, $\sigma$, and $p^*$ such that if $\|u_0\|_{L^{p^*}(\R^d)} < C(p^*)$, then
\begin{align*}
  \sup_{0<t<T}\big(\|\bar{u}^\eta(t)\|_{L^\infty(\R^d)}
  + \|u(t)\|_{L^\infty(\R^d)}\big) &\le C(T), \\
  \sup_{0<t<T}\|(\bar{u}^\eta-u)(t)\|_{L^1(\R^d)\cap L^\infty(\R^d)}
  &\le C(T)\eta.
\end{align*}
\end{lemma}

The constant $C(p^*)$ is determined explicitly from the (sharp) constants of the Hardy--Littlewood--Sobolev and Gagliardo--Nirenberg inequalities. This means that the initial data cannot be arbitrarily large in the $L^{p^*}(\R^d)$ norm, but we do {\em not} need to impose ``small'' initial data. Our result is natural, since equation \eqref{1.u} is aggregating for $\kappa=1$, and a restriction on the initial datum is expected. We split the proof of Lemma \ref{lem.Linfty} into two steps, first proving the bounds for any $p<\infty$ and then for $p=\infty$.

\begin{lemma}[$L^p(\R^d)$ estimate for $p<\infty$]\label{lem.p}
There exists a constant $C(p^*)>0$, only depending on $d$, $\lambda$, $p^*$, and $\sigma$ such that if $\|u_0\|_{L^{p^*}(\R^d)}\le C(p^*)$, then
\begin{align*}
  \sup_{0<t<T}\big(\|\bar{u}^\eta(t)\|_{L^{p^*}(\R^d)}
  + \|u(t)\|_{L^{p^*}(\R^d)}\big) \le \|u_0\|_{L^{p^*}(\R^d)}.
\end{align*}
Let $p<\infty$. If additionally $\|u_0\|_{L^{p^*}(\R^d)} < C(p^*)$ and $u_0\in L^p(\R^d)$ holds, then
\begin{align*}
  \sup_{0<t<T}\big(\|\bar{u}^\eta(t)\|_{L^p(\R^d)}
  + \|u(t)\|_{L^p(\R^d)}\big) \le C(T).
\end{align*}
\end{lemma}

\begin{proof}
We show the bounds only for $w:=\bar{u}^\eta$, since the computations for $u$ are similar. The idea is to use $pw^{p-1}$ as a test function in the weak formulation of \eqref{1.baru}. Writing $pw^{p-1}\na w=\na w^p$, integrating by parts, and recalling that $V^\eta=\chi^\eta*\Phi$, we arrive at
\begin{align*}%\label{4.wp}
  \frac{\dd}{\dd t}&\int_{\R^d}w^p\dd x
  + \frac{4\sigma}{p}(p-1)\int_{\R^d}|\na w^{p/2}|^2\dd x
  = \kappa p(p-1)\int_{\R^d}w^{p-2}\na w\cdot(w\na V^\eta*w)\dd x \\
  &= -\kappa(p-1)\int_{\R^d}w^p(\Delta V^\eta*w)\dd x
  = \kappa(p-1)\int_{\R^d}w^p(\chi^\eta*\Delta\Phi)*w\dd x. \nonumber 
\end{align*}
From now on, $C>0$ denotes a generic constant with values changing from line to line. We apply the Hardy--Littlewood--Soblev inequality, with the exponents $1/p^* + (d-2)/d + (\lambda+2)/d=2$, and the Sobolev inequality:
\begin{align*}
  \int_{\R^d}&w^p(\chi^\eta*\Delta\Phi)*w\dd x
  \le C\|w^p\|_{L^{d/(d-2)}(\R^d)}\|w\|_{L^{p^*}(\R^d)} \\
  &= C\|w^{p/2}\|_{L^{2d/(d-2)}(\R^d)}^2\|w\|_{L^{p^*}(\R^d)}
  \le C'\|\na w^{p/2}\|_{L^2(\R^d)}^2\|w\|_{L^{p^*}(\R^d)}.
\end{align*}
This shows that
\begin{align*}
  \frac{\dd}{\dd t}\int_{\R^d}w^p\dd x
  + (p-1)\bigg(\frac{4\sigma}{p} - C'\|w\|_{L^{p^*}(\R^d)}\bigg)
  \int_{\R^d}|\na w^{p/2}|^2\dd x \le 0,
\end{align*}
and choosing $\|u_0\|_{L^{p^*}(\R^d)}\le C(p^*):=4\sigma/(p^*C')$, we infer with $p=p^*$ that $\|w(t)\|_{L^{p^*}(\R^d)}\le \|u_0\|_{L^{p^*}(\R^d)}\le C(p^*)$ for all $t>0$.

Under the slightly stronger condition $\|u_0\|_{L^{p^*}(\R^d)}<C(p^*)$, we can find some $p_0>p^*$ such that 
$4\sigma/p^*-C'\|w(t)\|_{L^{p^*}(\R^d)}>4\sigma/p_0-C'\|w(t)\|_{L^{p^*}(\R^d)}\ge 0$ for all $t>0$ such that we can improve the previous bound,
\begin{equation}\label{4.p0}
  \sup_{0<t<\infty}\|w(t)\|_{L^{p_0}(\R^d)}\le \|u_0\|_{L^{p_0}(\R^d)}.
\end{equation}

We extend this result to any $1\le p<\infty$. To this end, we use $p(w-K)_+^{p-1}$ as a test function in the weak formulation of \eqref{1.baru}, where $z_+=\max\{z,0\}$ for $z\in\R$, and $K>0$ will be determined later. Then, by integrating by parts,
\begin{align*}
  \frac{\dd}{\dd t}&\int_{\R^d}(w-K)_+^p\dd x 
  + \frac{4\sigma}{p}(p-1)\int_{\R^d}|\na(w-K)_+^{p/2}|^2\dd x \\
  &= -\kappa p\int_{\R^d}(w-K)_+^{p-1}\big(\na w\cdot(\na V^\eta*w)
  + w(\Delta V^\eta*w)\big)\dd x \\
  &= -\kappa\int_{\R^d}\na(w-K)_+^p\cdot(\na V^\eta*w)\dd x \\
  &\phantom{xx}
  - \kappa p\int_{\R^d}\big((w-K)_+^p + (w-K)_+^{p-1}K\big)
  (\Delta V^\eta*w)\dd x \\
  &= -\kappa(p-1)\int_{\R^d}(w-K)_+^p(\Delta V^\eta*w)\dd x
  - \kappa pK\int_{\R^d}(w-K)_+^{p-1}(\Delta V^\eta*w)\dd x \\
  &= -\kappa(p-1)\int_{\R^d}(w-K)_+^p(\chi^\eta*\Delta\Phi*w)\dd x \\
  &\phantom{xx}
  - \kappa pK\int_{\R^d}(w-K)_+^{p-1}(\chi^\eta*\Delta\Phi*w)\dd x.
\end{align*}
We decompose the space $\R^d$ into $\{|y|\le 1\}$ and $\{|y|>1\}$:
\begin{align*}
  (\chi^\eta&*\Delta\Phi*w)(x) \\
  &= \int_{\{|y|\le 1\}}\Delta\Phi(y)(\chi^\eta*w)(x-y)\dd y
  + \int_{\{|y|>1\}}\Delta\Phi(y)(\chi^\eta*w)(x-y)\dd y \\
  &\le \int_{\{|y|\le 1\}}\Delta\Phi(y)(\chi^\eta*w)(x-y)\dd y
  + \|\Delta\Phi\|_{L^\infty(B_1(0)^c)}
  \|\chi^\eta\|_{L^1(\R^d)}\|w\|_{L^1(\R^d)} \\
  &\le \int_{\{|y|\le 1\}}\Delta\Phi(y)(\chi^\eta*w)(x-y)\dd y
  + C\|w\|_{L^1(\R^d)}.
\end{align*}
where we used $|\Delta\Phi(y)|\le C|y|^{-\lambda-2}\le C$ for $|y|>1$ and Young's convolution inequality. It follows from $\|w\|_{L^1(\R^d)}=1$ that
\begin{align*}
  \frac{\dd}{\dd t}&\int_{\R^d}(w-K)_+^p\dd x 
  + \frac{4\sigma}{p}(p-1)\int_{\R^d}|\na(w-K)_+^{p/2}|^2\dd x \\
  &\le -(p-1)\int_{\R^d}(w(x)-K)_+^p\int_{\{|y|\le 1\}}
  \Delta\Phi(y)(\chi^\eta*w)(x-y)\dd y\dd x \\
  &\phantom{xx} -pK\int_{\R^d}(w(x)-K)_+^{p-1}\int_{\{|y|\le 1\}}
  \Delta\Phi(y)(\chi^\eta*w)(x-y)\dd y\dd x \\
  &\phantom{xx}+ (p-1)C\int_{\R^d}(w-K)_+^p\dd x
  + pCK\int_{\R^d}(w-K)_+^{p-1}\dd x.
\end{align*}
We add and subtract $K$ to obtain only expressions involving $w-K$. Then
\begin{align}\label{4.G}
  & \frac{\dd}{\dd t}\int_{\R^d}(w-K)_+^p\dd x 
  + \frac{4\sigma}{p}(p-1)\int_{\R^d}|\na(w-K)_+^{p/2}|^2\dd x 
  \le G_1 + G_2, \quad\mbox{where} \\
  & G_1 = (p-1)\int_{\R^d}(w(x)-K)_+^p\int_{\{|y|\le 1\}}
  \Delta\Phi(y)(\chi^\eta*(w-K)_+)(x-y)\dd y\dd x \nonumber \\
  &\phantom{G_1=} 
  + (p-1)K\int_{\R^d}(w(x)-K)_+^p\bigg(\int_{\{|y|\le 1\}}
  |\Delta\Phi(y)|\dd y\bigg)\dd x \nonumber \\
  &\phantom{G_1=} + (p-1)C\int_{\R^d}(w-K)_+^p\dd x, \nonumber \\
  & G_2 = pK\int_{\R^d}(w(x)-K)_+^{p-1}\int_{\{|y|\le 1\}}
  \Delta\Phi(y)(\chi^\eta*(w-K)_+)(x-y)\dd y\dd x \nonumber \\
  &\phantom{G_2=}
  + pK^2\int_{\R^d}(w(x)-K)_+^{p-1}\bigg(\int_{\{|y|\le 1\}}
  |\Delta\Phi(y)|\dd y\bigg)\dd x \nonumber \\
  &\phantom{G_2=}+ pCK\int_{\R^d}(w-K)_+^{p-1}\dd x. \nonumber 
\end{align}
Since $\int_{\{|y|\le 1\}}|\Delta \Phi(y)|\dd y\le C$ for $\lambda<d-2$, we can decompose $G_i\le G_{i1}+G_{i2}$ for $i=1,2$, where
\begin{align*}
  G_{11} &= (p-1)\int_{\R^d}(w(x)-K)_+^p\int_{\{|y|\le 1\}}
  \Delta\Phi(y)(\chi^\eta*(w-K)_+)(x-y)\dd y\dd x, \\
  G_{12} &= C(p-1)(K+1)\int_{\R^d}(w-K)_+^p\dd x, \\
  G_{21} &= pK\int_{\R^d}(w(x)-K)_+^{p-1}\int_{\{|y|\le 1\}}
  \Delta\Phi(y)(\chi^\eta*(w-K)_+)(x-y)\dd y\dd x, \\
  G_{22} &= CpK(K+1)\int_{\R^d}(w-K)_+^{p-1}\dd x.
\end{align*}

The term $G_{11}$ is handled by using the Hardy--Littlewood--Sobolev inequality with exponents $1/p^* + (d-2)/d + (\lambda+2)/d = 2$ and the Sobolev inequality:
\begin{align*}
  G_{11} &\le (p-1)\|(w-K)_+\|_{L^{p^*}(\R^d)}
  \|(w-K)_+^p\|_{L^{d/(d-2)}(\R^d)} \\
  &\le C(p-1)\|(w-K)_+\|_{L^{p^*}(\R^d)}
  \|\na(w-K)_+^{p/2}\|_{L^2(\R^d)}^2.
\end{align*}
Similarly, the Hardy--Littlewood--Sobolev inequality with the same exponents as before leads to
\begin{align*}
  G_{21} &\le Cp\|(w-K)_+\|_{L^{p^*}(\R^d)}
  \|(w-K)_+^{p-1}\|_{L^{d/(d-2)}(\R^d)} \\
  &= Cp\|(w-K)_+\|_{L^{p^*}(\R^d)}
  \|(w-K)_+^{p/2}\|_{L^{2d(p-1)/(d-2)p}(\R^d)}^{2(p-1)/p} \\
  &\le Cp\|(w-K)_+\|_{L^{p^*}(\R^d)}
  \|\na(w-K)_+^{p/2}\|_{L^{2}(\R^d)}^{2(p-1)\theta/p}
  \|(w-K)_+\|_{L^1(\R^d)}^{2(p-1)(1-\theta)/p},
\end{align*}
where we used in the last step the Gagliardo--Sobolev inequality with $\theta = 2d(1-1/r)/(d+2)$ and $r=2d(p-1)/((d-2)p)$. Notice that $r\ge 1$ if and only if $p\ge 2d/(d+2)$. Since $2(p-1)\theta/p<2$, we can apply Young's inequality, denoting $M^* := \|u_0\|_{L^{p^*}(\R^d)}$,
\begin{align*}
  G_{21} &\le CpM^*
  \|\na(w-K)_+^{p/2}\|_{L^2(\R^d)}^{2(p-1)\theta/p} \\
  &\le \frac{2\sigma}{p}(p-1)\|\na(w-K)_+^{p/2}\|_{L^2(\R^d)}^2 
  + C(M^*,p).
\end{align*}
Finally, we turn to the term $G_{22}$. By interpolation with $\theta=p(p-2)/(p-1)^2<1$ (for $p>2$) and Young's inequality, 
\begin{align*}
  G_{22} &\le C(K)p\|(w-K)_+\|_{L^p(\R^d)}^{\theta(p-1)}
  \|(w-K)_+\|_{L^1(\R^d)}^{(1-\theta)(p-1)} \\
  &\le \|(w-K)_+\|_{L^p(\R^d)}^p + C(K,p).
\end{align*}
We insert these estimates into \eqref{4.G} to infer that
\begin{align*}
  \frac{\dd}{\dd t}\int_{\R^d}&(w-K)_+^p\dd x
  + (p-1)\bigg(\frac{2\sigma}{p} - C\|(w-K)_+\|_{L^{p^*}(\R^d)}\bigg)
  \int_{\R^d}|\na(w-K)_+^{p/2}|^2\dd x \\
  &\le C(K)\int_{\R^d}(w-K)_+^p\dd x + C(K,p).
\end{align*}

Now, for any fixed $p>\max\{p^*,2\}$, we choose $K>0$ such that the coefficient of the second term on the left-hand side is positive. This is possible since, by interpolation and in view of \eqref{4.p0},
\begin{align*}
  C\|(w-K)_+\|_{L^{p^*}(\R^d)} &\le C\bigg(\int_{\{w>K\}}
  (w-K)^{p_0}\dd x\bigg)^{1/p_0}
  \bigg(\int_{\{w>K\}}\mathrm{1}\dd x\bigg)^{1/p^*-1/p_0} \\
  &\le C\|(w-K)_+\|_{L^{p_0}(\R^d)}\bigg(\int_{\{w>K\}}\frac{w}{K}\dd x
  \bigg)^{1/p^*-1/p_0} \\
  &\le C\|u_0\|_{L^{p_0}(\R^d)}
  \bigg(\frac{1}{K}\bigg)^{1/p^*-1/p_0} < \frac{2\sigma}{p},
\end{align*}
and the last inequality is true for sufficiently large $K$. Here, the fact that $p_0>p^*$ is crucial. We deduce from Gronwall's inequality that there exists $K(p)>0$ such that for any $K\ge K(p)$,
\begin{align*}
  \sup_{0<t<T}\int_{\R^d}(w(t)-K)_+^p\dd x
  \le C(K,T)\int_{\R^d}(u_0-K)_+^p\dd x + C(K,p,T).
\end{align*}
Thus, for any $0\le t\le T$ and for $K=K(p)$,
\begin{align*}
  \int_{\R^d}w(t)^p\dd x 
  &= \int_{\{w(t)>K\}}w(t)^p\dd x  + \int_{\{w(t)\le K\}}w(t)^p\dd x \\
  &\le C(p)\int_{\{w(t)>K\}}\big((w(t)-K)^p+K^p\big)\dd x 
  + K^{p-1}\int_{\{w(t)\le K\}}w(t)\dd x \\
  &\le C(p)\int_{\R^d}(w(t)-K)_+^p\dd x 
  + K^p\int_{\{w(t)>K\}}\frac{w(t)}{K}\dd x + K^{p-1} \\
  &\le C(K,p,T,\|u_0\|_{L^p(\R^d)}) = C(p,T,\|u_0\|_{L^p(\R^d)}).
\end{align*}
These estimates also hold for $u$, finishing the proof. 
\end{proof}

\begin{lemma}[$L^\infty(\R^d)$ estimate]\label{lem.wLinfty}
Let $\|u_0\|_{L^{p^*}(\R^d)}<C(p^*)$ for the constant $C(p^*)$ from Lemma \ref{lem.p} and let $u_0\in L^\infty(\R^d)$. Then
\begin{align*}
  \sup_{0<t<T}\big(\|\bar{u}^\eta(t)\|_{L^\infty(\R^d)}
  + \|u(t)\|_{L^\infty(\R^d)}\big) &\le C(T), \\
  \sup_{0<t<T}\|(\bar{u}^\eta-u)(t)\|_{L^1(\R^d)\cap L^\infty(\R^d)}
  &\le C(T)\eta.
\end{align*}
\end{lemma}

\begin{proof}
We show the first estimate only for $w=\bar{u}^\eta$, a similar argument leads to the bounds for $u$. The solution $w$ to \eqref{1.baru} can be written in mild formulation as 
\begin{align}\label{4.mild}
  w(t,x) = G(t,x)*u_0 - \int_0^t\int_{\R^d}G(t-s,x-y)\diver\big(
  w(\na V^\eta*w)\big)(s,y)\dd y\dd s,
\end{align}
where $G(t,x)=(4\pi t)^{-d/2}\exp(-|x|^2/(4t))$ is the heat kernel. This formulation is possible for sufficiently smooth solutions $w$, for instance in the framework of an approximated problem. Integrating by parts and taking the supremum over $\R^d$ on both sides and Young's inequality gives
\begin{align}\label{4.winfty}
  \sup_{0<t<T}\|w(t)\|_{L^\infty(\R^d)}
  &\le \sup_{0<t<T} \|G(t,\cdot)\|_{L^1(\R^d)}\|u_0\|_{L^\infty(\R^d)} 
  + \sup_{0<t<T}\int_0^t\|\na G(t-s,\cdot)\|_{L^q(\R^d)} \\
  &\phantom{xx}\times\|w(s)\|_{L^{q'}(\R^d)}
  \|(\na V^\eta*w)(s)\|_{L^\infty(\R^d)}\dd s, \nonumber
\end{align}
where $q,q'>1$ and $1/q+1/q'=1$. We estimate the three norms in the time integral. Lemma \ref{lem.p} shows that any $L^{q'}(\R^d)$ norm of $w$ is bounded uniformly in $(0,T)$ such that $\|w(s)\|_{L^{q'}(\R^d)}\le C(T)$ for $0<s<T$. Furthermore, since $|\na\Phi(y)|\le C$ for $|y|>1$ and $|\na\Phi(y)|\le C|y|^{-\lambda-1}$ for $|y|\le 1$, for $1<p<d/(\lambda +1)$ and $1/p+1/p'=1$,
\begin{align*}
  |(&\na V^\eta*w)(s,x)| \\
  &\le \bigg|\int_{\{|y|>1\}}\na\Phi(y)(\chi^\eta*w)(s,x-y)\dd y\bigg| 
  + \bigg|\int_{\{|y|\le 1\}}\na\Phi(y)(\chi^\eta*w)(s,x-y)\dd y\bigg| \\
  &\le C\bigg|\int_{\{|y|>1\}}(\chi^\eta*w)(s,x-y)\dd y\bigg| 
  + C\bigg(\int_{\{|y|\le 1\}}|y|^{(-\lambda-1)p}\dd y\bigg)^{1/p}
  \|\chi^\eta*w(s)\|_{L^{p'}(\R^d)} \\
  &\le C\|\chi^\eta\|_{L^1(\R^d)}\|w\|_{L^1(\R^d)}
  + C\|\chi^\eta\|_{L^1(\R^d)}\|w\|_{L^{p'}(\R^d)}
  \le C(\|u_0\|_{L^{p'}(\R^d)}).
\end{align*}
Moreover, the heat kernel satisfies for all $q \geq 1$ the estimate
\begin{equation}\label{4.naG}
  \|\na G(t-s,\cdot)\|_{L^q(\R^d)}
  \le \frac{C(d)}{|t-s|^{d(q-1)/(2q)+1/2}}.
\end{equation}
To estimate \eqref{4.winfty}, we choose $1<q<d/(d-1)$, which implies that $d(q-1)/2q+1/2<1$, such that the time integral over $\|\na G(t-s,\cdot)\|_{L^q(\R^d)}$ is finite. We conclude from \eqref{4.winfty} that
\begin{align*}
  \sup_{0<t<T}\|w(t)\|_{L^\infty(\R^d)}
  \le C\|u_0\|_{L^\infty(\R^d)} + C(T,\|u_0\|_{L^{p'}(\R^d)}).
\end{align*}

It remains to estimate the difference $v:=\bar{u}^\eta-u$. It satisfies the mild formulation $v(t)=P_1+P_2+P_3$, where
\begin{align*}
  P_1 &= \int_0^t\na G(t-s,\cdot)v\na V^\eta*\bar{u}^\eta(s)\dd s,\\
  P_2 &= \int_0^t\na G(t-s,\cdot)u\na V^\eta*v(s)\dd s, \\
  P_3 &= \int_0^t\na G(t-s,\cdot)u\na(\chi^\eta*\Phi-\Phi)*u(s)\dd s.
\end{align*}
Let $1\le p\le\infty$. Then, by estimate \eqref{4.naG} and Lemma \ref{lem.DkVbaru} (here we need $\lambda<d-2$),
\begin{align*}
  \|P_1\|_{L^p(\R^d)} &\le C(t)\int_0^t
  \|\na G*(v\na V^\eta*\bar{u}^\eta)\|_{L^p(\R^d)}\dd s \\
  &\le C(T)\int_0^t\|\na G\|_{L^1(\R^d)}\|v\|_{L^p(\R^d)}
  \|\na V^\eta*\bar{u}^\eta\|_{L^\infty(\R^d)}\dd s \\
  &\le C(T)\int_0^t(t-s)^{-1/2}\|v(s)\|_{L^p(\R^d)}\dd s \\
  &\le C(T)\int_0^t(t-s)^{-1/2}\big(
  \|v(s)\|_{L^1(\R^d)}+\|v(s)\|_{L^\infty(\R^d)}\big)\dd s,
\end{align*}
and the last step follows by interpolation. The constant $C(T)$ also depends on the $L^1(\R^d)$ and $L^\infty(\R^d)$ norms of $u_0$. Similarly, by interpolation and since $\|u\|_{L^p(\R^d)}$ is bounded,
\begin{align*}
  \|P_2\|_{L^p(\R^d)} &\le C(T)\int_0^t
  \|\na G\|_{L^1(\R^d)}\|u\|_{L^p(\R^d)}
  \|\na V^\eta*v\|_{L^\infty(\R^d)}\dd s \\
  &\le C(T)\int_0^t(t-s)^{-1/2}
  \|v(s)\|_{L^1(\R^d)\cap L^\infty(\R^d)}\dd s.
\end{align*}
Finally, using the mean-value estimate $\|\chi^\eta*w-w\|_{L^\infty(\R^d)}\le \eta\|\na w\|_{L^\infty(\R^d)}$ ,
\begin{align*}
  \|P_3\|_{L^p(\R^d)} &\le C(T)\int_0^t
  \|\na G\|_{L^1(\R^d)}\|u\|_{L^p(\R^d)}
  \|\chi^\eta*(\na\Phi*u)-(\na\Phi*u)\|_{L^\infty(\R^d)}\dd s \\
  &\le C(T)\eta\int_0^t(t-s)^{-1/2}
  \|\mathrm{D}^2\Phi*u(s)\|_{L^\infty(\R^d)}\dd s.
\end{align*}

Similarly as in the proof of Lemma \ref{lem.DkVbaru}, we have
\begin{align*}
  \|\mathrm{D}^2\Phi*u\|_{L^\infty(\R^d)}
  \le \|\mathrm{D}^2\Phi\|_{L^1(B_1(0))}\|u\|_{L^\infty(B_1(0))}
  + \|\mathrm{D}^2\Phi\|_{L^\infty(B_1(0)^c)}\|u\|_{L^1(B_1(0)^c)}
  \le C,
\end{align*}
again using the condition $\lambda<d-2$. Hence, $\|P_3\|_{L^p(\R^d)}\le C(T)\eta$. It follows that
\begin{align*}
  \|v(t)\|_{L^p(\R^d)} \le C(T)\eta + C(T)\int_0^t(t-s)^{-1/2}
  \|v(s)\|_{L^1(\R^d)\cap L^\infty(\R^d)}\dd s.
\end{align*}
Adding this inequality for $p=1$ and $p=\infty$ yields
\begin{align*}
  \|v(t)\|_{L^1(\R^d)\cap L^\infty(\R^d)}
  &= \|v(t)\|_{L^1(\R^d)}+\|v(t)\|_{L^\infty(\R^d)} \\
  &\le  C(T)\eta + C(T)\int_0^t(t-s)^{-1/2}
  \|v(s)\|_{L^1(\R^d)\cap L^\infty(\R^d)}\dd s,
\end{align*}
and the result follows after applying Gronwall's inequality, observing that $\int_0^t(t-s)^{-1/2}\dd s$ is integrable for $t\ge 0$.
\end{proof}

Finally, we show some higher-order estimates.

\begin{lemma}[Higher-order estimates]\label{lem.higher}
Let $k\in\N$, $\|u_0\|_{L^{p^*}(\R^d)}< C(p^*)$ with the constant $C(p^*)$ from Lemma \ref{lem.p}, and let $u_0\in X_k:=W^{k,1}(\R^d)\cap W^{k,\infty}(\R^d)$. Then
$$
  \sup_{0<t<T}\big(\|\bar{u}^\eta(t)\|_{X_k} + \|u(t)\|_{X_k}\big)
  \le C(k,T).
$$
Moreover, if $k=1$, then it holds that
$$
  \sup_{0<t<T}\|(\bar{u}^\eta-u)(t)\|_{X_1}\le C(T)\eta.
$$
\end{lemma}

It is possible to achieve the estimate $C(T)\eta^2$ by exploiting the radial symmetry of the potential and using the boundedness of $\|\mathrm{D}^3\Phi* u\|_{L^\infty(\R^d)}$; see \cite[(22)]{CHH23} for a computation.

\begin{proof}
The proof follows from the mild formulation \eqref{4.mild} of the solution $w=\bar{u}^\eta$ and an iteration argument. We need to estimate the drift term. For this, let first $k=1$ and let $1\le p\le\infty$. Then
\begin{align*}
  \|\diver(w\na V^\eta*w)\|_{L^p(\R^d)}
  &\le \|\na w\cdot(\na V^\eta*w)\|_{L^p(\R^d)}
  + \|w(\Delta V^\eta*w)\|_{L^p(\R^d)} \\
  &\le \|\na w\|_{L^p(\R^d)}\|\na V^\eta*w\|_{L^\infty(\R^d)}
  + \|w\|_{L^p(\R^d)}\|\Delta V^\eta*w\|_{L^\infty(\R^d)} \\
  &\le C(T)\|\na w\|_{L^p(\R^d)} + C(T),
\end{align*}
where we used Lemmas \ref{lem.DkVbaru} and \ref{lem.DkVbaru} to estimate the convolutions as well as the bound $\|w\|_{L^1(\R^d)}+\|w\|_{L^\infty(\R^d)}\le C(T)$ from Lemma \ref{lem.wLinfty}. Then, applying the gradient to the mild formulation \eqref{4.mild} and using Young's convolution inequality,
\begin{align*}
  \|\na w(t)\|_{L^p(\R^d)}
  &\le \|\na u_0\|_{L^p(\R^d)} + C(T)\int_0^t
  \|\na G(t-s,\cdot)*\diver(w\na V^\eta*w)(s)\|_{L^p(\R^d)}\dd s \\
  &\le \|\na u_0\|_{L^p(\R^d)} + C(T)\int_0^t(t-s)^{-1/2}
  \|\na w(s)\|_{L^p(\R^d)}\dd s  + C(T).
\end{align*}
An application of Gronwall's inequality yields the claim for $k=1$.

Next, we prove the estimates by an induction argument. Assume that for some $k\in\N$ and for any $1\le p\le\infty$,
$$
  \sup_{0<t<T}\|w(t)\|_{W^{k,p}(\R^d)}\le C(T).
$$ 
Let $\alpha\in\N_0^d$ with $|\alpha|=k+1$ be a multi-index. We need to estimate the mild formulation
\begin{align}\label{4.alphaw}
  \|\mathrm{D}^\alpha w(t)\|_{L^p(\R^d)} 
  \le \|\mathrm{D}^\alpha u_0\|_{L^p(\R^d)}
  + C(T)\int_0^t\|\na G*\mathrm{D}^\alpha
  (w\na V^\eta*w)\|_{L^p(\R^d)}\dd s.
\end{align}
We deduce from the induction assumption and Lemma \ref{lem.DkVbaru} for $\lambda<d-2$ that
\begin{align*}
  \|\mathrm{D}^\alpha(w\na V^\eta*w)\|_{L^p(\R^d)}
  &\le \|\mathrm{D}^\alpha w\|_{L^p(\R^d)}
  \|\na V^\eta*w\|_{L^\infty(\R^d)} \\
  &\phantom{xx}
  + C(k)\sum_{\ell=0}^k\|\mathrm{D}^\ell w\|_{L^p(\R^d)}
  \|\mathrm{D}^2 V^\eta*\mathrm{D}^{k-\ell}w\|_{L^\infty(\R^d)} \\
  &\le C(T)\|\mathrm{D}^\alpha w\|_{L^p(\R^d)} + C(k,T).
\end{align*}
Then \eqref{4.alphaw} shows that
\begin{align*}
  \|\mathrm{D}^\alpha w(t)\|_{L^p(\R^d)} 
  \le \|\mathrm{D}^\alpha u_0\|_{L^p(\R^d)}
  + C(T)\int_0^t(t-s)^{-1/2}\|\mathrm{D}^\alpha w(s)\|_{L^p(\R^d)}
  \dd s + C(k,T),
\end{align*} 
and Gronwall's inequality implies an estimate for $w(t)$ in $W^{k+1,\infty}(\R^d)$ uniformly in $t\in(0,T)$. The bound for $\mathrm{D}^\alpha u(t)$ is proved similarly.

The bound for the difference $\bar{u}^\eta-u$ is derived in the same way as in Lemma \ref{lem.wLinfty} by performing an induction argument as above, taking into account the estimate
$$
  \|\mathrm{D}^{\alpha}(\chi^\eta*\na\Phi-\na\Phi)*u\|_{L^\infty(\R^d)}
  \le C\eta\|\mathrm{D^2}\Phi*\mathrm{D}^{k} u\|_{L^\infty(\R^d)}
  \le C\eta, 
$$
where $|\alpha| \leq k+1$ and $C>0$ depends on $\|u\|_{X_k}$, which is bounded by the induction hypothesis. This finishes the proof. 
\end{proof}

%%%%%%%%%%%%%%%%%%

\subsection{Estimates for the solution to a linear dual evolution problem}
\label{sec.dual}

In the study of fluctuations, we need the dual of the operator associated to the linearized version of \eqref{1.u}, 
\begin{align*}
  \mathcal{L}f := \pa_t f  - \sigma\Delta f 
  + \kappa\diver(f\na\Phi*u + u\na\Phi*f),
\end{align*}
where $u$ solves \eqref{1.u}. The dual operator $\mathcal{L}^*$ reads as
\begin{align*}
  \mathcal{L}^*\psi = -\pa_{t}\psi - \sigma\Delta\psi 
  - \kappa\big((\na\Phi*u)\cdot\na\psi - \na\Phi*(u\na\psi)\big),
\end{align*}
since, by the antisymmetry of $\na\Phi$,
$$
  \langle u\na\Phi*\phi,\na f\rangle
  =\langle\na\Phi*\phi,u\na f\rangle
  =-\langle f,\na\Phi*(u\na\psi)\rangle.
$$
We need a priori estimates for the corresponding backward dual evolution problem 
\begin{align}\label{4.back}
  -\pa_s v - \sigma\Delta v 
  = \kappa(\na\Phi*u)\cdot\na v - \kappa\na\Phi*(u\na v)
  \quad\mbox{in }\R^d,\ s\in(0,t), \quad v(t)=\phi,
\end{align}
where $\phi$ is the end datum. We show the existence of a solution $T_\phi^t(s,x):=v(s,x)$ to \eqref{4.back} and some a priori estimates. Let $u\in L^\infty(0,T;X_{k+1})$ be the solution to \eqref{1.u} or $u_0 \in X_{k+1}$ and $k\in \N$, recalling the definition $X_k:=W^{k,1}(\R^d)\cap W^{k,\infty}(\R^d)$.

\begin{lemma}\label{lem.estT}
Let $k\in\N$, $0<t<T$, and $\phi\in W^{k,\infty}(\R^d)$. Then there exists a unique solution $T_\phi^t$ to \eqref{4.back} satisfying
$$
  \|\mathrm{D}^k T_\phi^t\|_{L^\infty(\R^d)}
  \le C(k,T)\|\phi\|_{W^{k,\infty}(\R^d)},
$$
where $C(k,T)>0$ depends on $L^\infty(0,T;X_{k+1})$ norm of $u$.
\end{lemma}

\begin{proof}
We transform the backward problem via $w(s,x):=T_\phi^t(t-s,x)$ into the initial-value problem
\begin{align*}
  \pa_s w - \sigma\Delta w - \kappa \big((\na\Phi*u(t-\boldsymbol\cdot))
  \cdot\na w - \na\Phi*(u(t-\boldsymbol\cdot)\na w)\big) &= 0
  \quad \mbox{in }\R^d\times(0,t), \\
  w(0) &= \phi\quad\mbox{in }\R^d.
\end{align*}
We prove only the a priori estimates, since the existence and uniqueness of the linear nonlocal problem can be deduced from them by using standard fixed-point arguments and the superposition principle. Let $\alpha\in\N_0^d$ with $|\alpha|\le k$ be a multi-index. We obtain from the mild formulation that
\begin{align*}
  &\|\mathrm{D}^\alpha w(s)\|_{L^\infty(\R^d)}
  \le \|\mathrm{D}^\alpha\phi\|_{L^\infty(\R^d)} 
  + C(t)\int_0^s (P_4(r) + P_5(r))\dd r,
  \quad\mbox{where}  \\
  & P_4(r) = \|G(s-r,\cdot)*\mathrm{D}^\alpha
  (\na\Phi*u(t-r)\cdot\na w(r))\|_{L^\infty(\R^d)}, \\
  & P_5(r) = \|G(s-r,\cdot)*\na\mathrm{D}^\alpha\Phi
  * (u(t-r)\na w(r))\|_{L^\infty(\R^d)}.
\end{align*}

For the estimate of $P_4(r)$, we observe that
$$
  \na\Phi*u(t-\boldsymbol\cdot)\cdot\na w 
  = \diver\big(\na\Phi*u(t-\boldsymbol\cdot)w\big)
  - (\Delta\Phi*u(t-\boldsymbol\cdot))w.
$$
Then, integrating by parts to remove the divergence, we have
\begin{align*}
  P_4(r) &\le \|\na G(s-r,\boldsymbol\cdot)*\mathrm{D}^\alpha(\na\Phi*u(t-r)w(r))
  \|_{L^\infty(\R^d)} \\
  &\phantom{xx}
  + \|G(s-r,\boldsymbol\cdot)*\mathrm{D}^\alpha(\Delta\Phi*u(t-r)
  w(r))\|_{L^\infty(\R^d)} \\
  &\le \frac{C(|\alpha|)}{|s-r|^{1/2}}\sum_{\ell=0}^{|\alpha|}
  \|\mathrm{D}^\ell\na\Phi*u(t-r)\|_{L^\infty(\R^d)}
  \|\mathrm{D}^{|\alpha|-\ell}w(r)\|_{L^\infty(\R^d)} \\
  &\phantom{xx} + C(|\alpha|)\sum_{\ell=0}^{|\alpha|}
  \|\mathrm{D}^\ell\Delta\Phi*u(t-r)\|_{L^\infty(\R^d)}
  \|\mathrm{D}^{|\alpha|-\ell}w(r)\|_{L^\infty(\R^d)}.
\end{align*}
We infer from
\begin{align*}
  \|\mathrm{D}^\ell\na\Phi*u(t-r)\|_{L^\infty(\R^d)}
  &\le C\|u\|_{L^\infty(0,T;X_{\ell-1})} \le C, \\
  \|\mathrm{D}^\ell\Delta\Phi*u(t-r)\|_{L^\infty(\R^d)}
  &\le C\|u\|_{L^\infty(0,T;X_{\ell})} \le C,
\end{align*}
which can be obtained by similar computations as in Lemma \ref{lem.DkVbaru}, that
\begin{align*}
  P_4(r) \le C(|\alpha|)(|s-r|^{-1/2}+1)
  \|w(r)\|_{W^{|\alpha|,\infty}(\R^d)}.
\end{align*}

Next, taking into account that
$$
  \|\Phi*f\|_{L^\infty(\R^d)} + \|\na\Phi*f\|_{L^\infty(\R^d)}
  \le C(\|f\|_{L^1(\R^d)} + \|f\|_{L^\infty(\R^d)}),
$$
by Lemma \ref{lem.DkVbaru} and $u(t-r)\na w(r) = \na (u(t-r)w(r)) - \na u(t-r) w(r)$, we find that
\begin{align*}
  P_5(r) &\le \|\na G(s-r,\boldsymbol\cdot)*\mathrm{D}^\alpha\Phi
  *\na(u(t-r)w(r))\|_{L^\infty(\R^d)} \\
  &\phantom{xx}+ \|\na G(s-r,\boldsymbol\cdot)*\mathrm{D}^\alpha\Phi
  *(\na u(t-r)w(r))\|_{L^\infty(\R^d)} \\
  &\le \frac{C(|\alpha|)}{|s-r|^{1/2}}\big(\|\mathrm{D}\Phi
  * \mathrm{D}^{|\alpha|}(u(t-r)w(r))\|_{L^\infty(\R^d)} \\
  &\phantom{xx}
  + \|\Phi*\mathrm{D}^{|\alpha|}(\na u(t-r)w(r))\|_{L^\infty(\R^d)}\big)
  \\
  &\le \frac{C(|\alpha|)}{|s-r|^{1/2}}
  \|\mathrm{D}^{|\alpha|}(u(t-r)w(r))\|_{L^1(\R^d)\cap L^\infty(\R^d)} \\
  &\phantom{xx}+ \frac{C(|\alpha|)}{|s-r|^{1/2}}
  \|\mathrm{D}^{|\alpha|}(\na u(t-r)w(r))\|_{L^1(\R^d)
  \cap L^\infty(\R^d)} \\
  &\le \frac{C(|\alpha|)}{|s-r|^{1/2}}
  \|w(r)\|_{W^{|\alpha|,\infty}(\R^d)},
\end{align*}
where $C(|\alpha|)$ also depends on the $L^\infty(0,T;X_{k+1})$ norm of $u$. We sum the previous estimates over $|\alpha|\le k$, which yields
\begin{align*}
  \|w(s)\|_{W^{k,\infty}(\R^d)} \le \|\phi\|_{W^{k,\infty}(\R^d)}
  + C(k)\int_0^s\big(|s-r|^{-1/2}+1\big)
  \|w(r)\|_{W^{k,\infty}(\R^d)}\dd r.
\end{align*}
Gronwall's inequality implies that
$$
  \sup_{0<s<t}\|w(s)\|_{W^{k,\infty}(\R^d)}
  \le C(k,T)\|\phi\|_{W^{k,\infty}(\R^d)},
$$
which ends the proof.
\end{proof}

%%%%%%%%%%%%%%%%%%%%%%%%%%%%%%%%%%%%%%%%%%%%%%%%%%%%%%%%%%%%%%%%%%%%%%%

\end{document}